\documentclass[final,onefignum,onetabnum]{siamonline220329}

\usepackage[utf8]{inputenc} 
\usepackage[T1]{fontenc}    
\usepackage{lipsum}
\usepackage{hyperref}       
\usepackage{url}            
\usepackage{booktabs}       
\usepackage{amsfonts}       
\usepackage{nicefrac}       
\usepackage{microtype}      
\usepackage{graphicx}
\usepackage[numbers]{natbib}
\usepackage{doi}
\usepackage{graphicx}
\usepackage{array}
\usepackage{amssymb}
\usepackage{cases} 
\usepackage{enumerate}
\usepackage{enumitem}
\usepackage{bbm}
\usepackage{float}
\usepackage{mathtools}
\usepackage{setspace}
\usepackage{natbib}
\usepackage{newtxmath}
\usepackage{mathtools}
\mathtoolsset{showonlyrefs = false}
\usepackage{cases} 

\usepackage{doi}
\usepackage{cases}
\usepackage{empheq}
\usepackage{caption}
\usepackage{algorithm}
\usepackage[noend]{algpseudocode}

\usepackage{subcaption}

\ifpdf
  \DeclareGraphicsExtensions{.eps,.pdf,.png,.jpg}
\else
  \DeclareGraphicsExtensions{.eps}
\fi

\usepackage{enumitem}
\setlist[enumerate]{leftmargin=.5in}
\setlist[itemize]{leftmargin=.5in}

\newsiamremark{remark}{Remark}
\newsiamremark{hypothesis}{Hypothesis}
\crefname{hypothesis}{Hypothesis}{Hypotheses}
\newsiamthm{claim}{Claim}
\newsiamthm{assumption}{Assumption}

\headers{An adaptively inexact first-order method for bilevel optimization}{M. S. Salehi, S. Mukherjee, L. Roberts, M. J. Ehrhardt}

\title{An adaptively inexact first-order method for bilevel optimization with application to hyperparameter learning \thanks{Submitted to the editors DATE. 
\funding{The work of Mohammad Sadegh Salehi was supported by a scholarship from the EPSRC Centre for Doctoral Training in Statistical Applied Mathematics at Bath (SAMBa), under the project EP/S022945/1. Matthias J. Ehrhardt acknowledges support from the EPSRC (EP/S026045/1, EP/T026693/1, EP/V026259/1). Lindon Roberts is supported by the Australian Research Council Discovery Early Career Award DE240100006.}}}

\author{Mohammad Sadegh Salehi\thanks{Department of Mathematical Sciences, University of Bath, Bath, BA2 7AY, UK (\email{mss226@bath.ac.uk}).}\and Subhadip Mukherjee\thanks{Department of Electronics \& Electrical Communication Engineering, Indian Institute of Technology (IIT), Kharagpur, India (\email{smukherjee@ece.iitkgp.ac.in}).} \and Lindon Roberts\thanks{School of Mathematics and Statistics, University of Sydney, Camperdown NSW 2006, Australia (\email{lindon.roberts@sydney.edu.au)}.} \and Matthias J. Ehrhardt \thanks{Department of Mathematical Sciences, University of Bath, Bath, BA2 7AY, UK (\email{m.ehrhardt@bath.ac.uk})}.}

\usepackage{amsopn}

\ifpdf
\hypersetup{
  pdftitle={An adaptively inexact first-order method for bilevel optimization with application to hyperparameter learning},
  pdfauthor={Mohammad Sadegh Salehi, Subhadip Mukherjee, Lindon Roberts, and Matthias~J.~Ehrhardt}
}
\fi

\newcommand{\rv}[2][final]{%
  \ifthenelse{\equal{#1}{revision}}%
    {\textcolor{blue}{#2}}
    {#2}
}

\newcommand{\rvv}[2][final]{%
  \ifthenelse{\equal{#1}{revision}}%
    {\textcolor{blue}{#2}}
    {#2}
}

\begin{document}

\maketitle

\begin{abstract}
Various tasks in data science are modeled utilizing the variational regularization approach, where manually selecting regularization parameters presents a challenge. The difficulty gets exacerbated when employing regularizers involving a large number of hyperparameters. To overcome this challenge, bilevel learning can be employed to learn such parameters from data. However, neither exact function values nor exact gradients with respect to the hyperparameters are attainable, necessitating methods that only rely on inexact evaluation of such quantities. State-of-the-art inexact gradient-based methods a priori select a sequence of the required accuracies and cannot identify an appropriate step size since the Lipschitz constant of the hypergradient is unknown. 
In this work, we propose an algorithm with backtracking line search that only relies on inexact function evaluations and hypergradients and show convergence to a stationary point. Furthermore, the proposed algorithm determines the required accuracy dynamically rather than manually selected before running it. 
Our numerical experiments demonstrate the efficiency and feasibility of our approach for hyperparameter estimation on a range of relevant problems in imaging and data science such as total variation and field of experts denoising and multinomial logistic regression. Particularly, the results show that the algorithm is robust to its own hyperparameters such as the initial accuracies and step size.
\end{abstract}


\section{Introduction}
In this work, we consider bilevel learning, where the hyperparameter learning problem is formulated as a bilevel optimization problem \cite{Crockett_2022,Reyes2023} where our goal is to solve
\begin{subequations}\label{GENERAL}
    \begin{align}
        \min_{\theta \in \mathbb{R}^d} \Bigl\{f(\theta) &\coloneqq \frac{1}{m}\sum_{i=1}^{m}g_i(\hat{x}_i(\theta)) +r(\theta)\Bigr\}\label{UPPER_LEVEL}\\ s.t. \quad
         {\hat{x}_i(\theta)} &\coloneqq \arg\min_{x\in \mathbb{R}^n} h_i(x,\theta), \quad i = 1,\dots, m.
         \label{LOWER_LEVEL}
    \end{align}
\end{subequations}
The upper-level functions $g_i$ are a measure of quality for the solution of the lower-level problem~\eqref{LOWER_LEVEL}. For example, in supervised learning, they may take the form $g_i(x)= \|x - x^*_i\|^2$ where $x^*_i$ is the desired solution of the lower-level problem.

Problems of the form \eqref{GENERAL} arise in many areas of data science, for example, when the task at hand is modeled using variational regularization approaches, which are commonly used in the fields of image reconstruction, image processing and machine vision. For instance, one can learn the weights of the regularizer in regression \cite{Pedregosa}, or the parameters of the regularizer~\cite{Kunisch2013ABO,Yunjin_Chen_2014,Ochs} for tasks such as image denoising, deblurring, inpainting, segmentation, and single-image super-resolution. Moreover, sophisticated data-adaptive regularizers can have millions of parameters; examples include input-convex neural networks \cite{InputConvex,amos2017input} and convex-ridge regularizer neural networks \cite{goujon2022neuralnetworkbased}. However, currently, neither of these data-adaptive regularizers are learned by bilevel learning due to computational difficulties. The relevance of bilevel learning in variational regularization methods is not restricted to regularizers. It can also be employed to learn the parameters in analytical deep priors \cite{Dittmer_2019, Arndt_2022}, and the sampling operator in magnetic resonance imaging (MRI) \cite{sherry2020learning} and seismic imaging \cite{downing2023optimisation}. These tasks can also involve a large number of parameters; for example, the authors in \cite{sherry2020learning} demonstrated the use of more than a million parameters. Furthermore, based on the theoretical approach introduced in \cite{Holler_graph}, bilevel learning can be employed to learn the optimal structure of the regularizer.

To address the bilevel learning problem \eqref{GENERAL}, one can employ model-free approaches like grid search and random search \cite{randomsearch} when dealing with a limited set of hyperparameters and a small, bounded hyperparameter space. Moving beyond the capabilities of model-free approaches, model-based approaches such as Bayesian methods \cite{bayesian} and derivative-free optimization \cite{cartis2017derivativefree,DFOLS1,DFO} prove to be efficient in solving~\eqref{GENERAL} when a small number of hyperparameters are involved. However, here we are interested in computationally scalable algorithms that can solve bilevel learning problems potentially with millions of parameters, thus, in this work, we consider gradient-based methods. In the context of the gradient-based approach, firstly, a first-order or quasi-Newton algorithm \cite{Yunjin_Chen_2014,maclaurin2015gradientbased} is employed to solve the lower-level problem \eqref{LOWER_LEVEL}. Then, gradient descent is applied to the upper-level problem \eqref{UPPER_LEVEL} after computing the gradient $\nabla f(\theta)$ (also referred to as the \textit{hypergradient}) with respect to parameters $\theta$.

Under certain regularity assumptions, when the problem is sufficiently smooth \cite{DFO, Pedregosa}, the hypergradient can be calculated by utilizing the implicit function theorem (IFT) \cite{DFO} or by using automatic differentiation (AD) \cite{AD}. Another type of algorithm called piggyback \cite{piggyback} utilizes AD to compute the hypergradient, and they can differ in the required assumptions on the lower-level problem \eqref{LOWER_LEVEL}. This approach is favorable for nonsmooth lower-level problems. The works~\cite{Piggy_chambolle,Piggy_nonsmooth_bolte} provided an extensive analysis of this method, where \cite{Piggy_nonsmooth_bolte} considered approximating the hypergradient under the usage of different splitting methods, and \cite{Piggy_chambolle} studied the convergence of piggyback AD when the lower-level problem is a saddle point problem and employed primal-dual algorithms for solving it. As an application, the authors in \cite{bubba2023bilevel} applied piggyback AD to learn discretization of total variation.
%
Quasi-Newton algorithms \cite{Nocedal2006Numerical} are used in \cite{ramzi2023shine} for solving the lower-level problem and utilized their approximation of the inverse of the Hessian of the lower-level objective to replace \rv{linear solvers like conjugate gradient (CG)}.

In practice, due to the usage of numerical solvers and the large-scale nature of the problems of interest, computing the exact hypergradient is not feasible \cite{boundMatthias}; hence, it should be approximated with an accuracy that leads to optimization with inexact gradients in the upper-level problem~\eqref{UPPER_LEVEL}. Computing an approximate hypergradient using the IFT approach requires solving a linear system and solving it up to a tolerance using CG. The first error-bound analysis of the approximate gradient computed in this approach was done in \cite{Pedregosa} and improved in \cite{Zucchet_2022}. \rv{In \cite{ghadimi2018approximation}, convergence rates for bilevel problems with approximate hypergradients were studied using a pre-specified number of lower-level iterations. Improved rates were later presented in \cite{bilevel_complexity_warmstart} by fixing the number of lower-level iterations and using warm-start.} \rvv{Numerous other works (see \cite{TTSA,JMLR:stochastic} and references therein) discuss convergence rates for solving bilevel problems in the stochastic setting.} \rv{In \cite{grazzi20a}, an iteration complexity analysis and a priori error bounds for the approximate hypergradient and all sub-problems were introduced by separately considering IFT + CG and AD approaches and reformulating the lower-level problem as a fixed-point problem.} Subsequently, the error bound analysis of IFT + CG was further extended by providing computable a priori and a posteriori bounds in \cite{boundMatthias}. On the other hand, the error-bound analysis of approximating the hypergradient using AD was studied in \cite{AD}. Furthermore, the authors in \cite{boundMatthias} proposed a unified perspective in which AD can be seen as equivalent to IFT+CG.
\subsection{Challenges}
\paragraph{Required accuracy for hypergradient approximation} Although much work has been done on different methods of finding the hypergradient and its error bound, convergence theory for hyperparameter optimization with an inexact hypergradient has often assumed the presence of an accurate approximate hypergradient. However, in practice, approximating the hypergradient to a high degree of accuracy results in prohibitively high computational costs. Moreover, empirical evidence, as demonstrated in studies such as \cite{boundMatthias}, indicates that even when the accuracy is not very high or strictly increasing under a priori assumptions like summability, inexact gradient descent on the upper-level problem can still yield progress. As elevating accuracy necessitates additional computations, implementing a dynamic strategy for selecting the required accuracy, allowing it to adjust according to the demands of the optimization, opens the possibility of reducing the overall computational cost and increase the efficiency of the optimization process.

\paragraph{Selecting upper-level step sizes}
The convergence of inexact gradient descent on the upper-level problem has primarily been studied under the assumption of a sufficiently small fixed step size \cite{Pedregosa}. However, due to the unavailability of a closed-form solution for the lower-level problem \eqref{LOWER_LEVEL}, estimating the Lipschitz constant of the hypergradient and consequently employing a method with the largest possible fixed step size is unrealistic. While line search approaches such as backtracking may seem plausible to address this issue, they typically require evaluating the exact upper-level objective value, which is unattainable since only an approximate lower-level solution is available. Additionally, line search methods like the Armijo rule \cite{Nocedal2006Numerical} necessitate the approximated hypergradient to be a descent direction, which is not obvious. 


\subsection{Existing work and contributions}
The IFT+CG approach for approximating the hypergradient, given a decreasing sequence of accuracies, was initially utilized in the hyperparameter optimization approximate gradient (HOAG) algorithm \cite{Pedregosa}. In this algorithm, convergence of the inexact gradient descent in the gradient mapping was demonstrated under the conditions of having an inexact gradient with a summable error bound and a sufficiently small step size \cite{Pedregosa}. While HOAG offered improved computational efficiency compared to methods relying on highly accurate approximate hypergradients, the assumption of summability may require more lower-level computations than necessary. Moreover, as discussed regarding the challenge of finding a suitable small enough fixed step size, these assumptions could potentially limit the practical performance of the method. In addressing the challenge of determining a suitable step size for inexact gradient descent in the upper-level, the author in \cite{Pedregosa} proposed a heuristic line search in numerical experiments. However, this line search lacked any convergence guarantee. Building on this, an accelerated version with a restarting scheme was later introduced in \cite{yang2023accelerating} to improve the performance. \rv{Additionally, a bilevel method with adaptive step sizes was presented in \cite{Huang2021BiAdamFA}. This method, which employs a fixed number of lower-level solver steps, falls within the category of stochastic bilevel methods.}

In \cite{valkonen_single_step}, the authors introduced an IFT-based algorithm that performs only a single step to solve the lower-level problem. The algorithm utilizes CG with high accuracy or analytic exact inversion of the Hessian of the lower-level problem. \rv{Another class of single-level methods, which replace the lower-level problem with a so-called \textit{value function} \cite{dempe_bilevel_2020}, are the fully first-order methods \cite{bome,fullyFirstOrderStochastic,penaltyMethodsNonconvexBO}. These methods reformulate the bilevel problem as a constrained optimization problem, eliminating the need for second-order differentiation. They use a fixed number of iterations to approximate the lower-level solution and require careful tuning of step sizes and penalty multipliers. The majority of works in these approaches primarily consider stochastic bilevel problems.}

Moreover, some recent work has developed line search methods for single-level problems where the objective function value is computed with errors and the estimated gradient is inexact \cite{LineSearchNoise_deterministic} and potentially stochastic \cite{LineSearchNoise_stochastic}. However, in these methods, the error is bounded but irreducible and not controllable. \rv{Additionally, two other stochastic line search methods were proposed in \cite{Painless_Stochastic_linesearch} and \cite{non_monotone_stochastic_linesearch}. The method in \cite{Painless_Stochastic_linesearch} was later applied within a stochastic bilevel framework in \cite{Auto-tune_stochastic_bi}. Meanwhile, \cite{non_monotone_stochastic_linesearch} addresses the issue of overly small step sizes in \cite{Painless_Stochastic_linesearch} by relaxing the monotonic decrease condition, thus providing a non-monotone stochastic line search method for single-level problems.}
    
    An existing adaptive and inexact framework for bilevel learning, known as DFO-LS \cite{DFO}, operates by solving the lower-level problem up to a specified tolerance. Subsequently, leveraging an approximate lower-level solution, the framework utilizes a derivative-free algorithm \cite{DFOLS1} tailored for solving nonlinear least-squares to update the parameters in the upper-level. This method dynamically determines the necessary accuracy for solving the lower-level problem by considering the trust-region radius of the derivative-free algorithm, ensuring that the upper-level solver can progress effectively. While this approach shows promise in tasks such as determining the parameters of total variation denoising, it faces scalability challenges as the number of parameters increases \cite{Crockett_2022}. 
    
    In this work, we propose a verifiable backtracking line search scheme in \cref{sec:bt_linesearch} that relies solely on inexact function evaluations and the inexact hypergradient. This scheme guarantees sufficient decrease in the exact upper-level objective function. We also establish the required conditions for finding a valid step size using this line search method in \cref{sec:step_existence} and prove the convergence of inexact gradient descent in the upper-level objective function when employing our inexact line search in \cref{sec:convergence}.
    \rv{As part of our line search scheme, we analyze the sufficient conditions to ensure that the hypergradient is a descent direction in \cref{sec:descent_direction}. In practice, we use the the posteriori error bound for the approximate hypergradient from \cite{boundMatthias}, which will be reviewed in \cref{sec:background}, to verify these conditions in numerical settings.}
    
    We present a practical algorithm in \cref{sec:Theory_algorithm} that connects all of our theoretical results, determining the required accuracy for the inexact hypergradient to be considered a descent direction and the inexact line search (sufficient decrease condition) to be held, adaptively. As a result, our algorithm provides a robust and efficient method for bilevel learning. Furthermore, as the numerical experiments in \cref{sec:numerical_experiments} demonstrate, our approach outperforms state-of-the-art methods such as HOAG \cite{Pedregosa} and DFO-LS \cite{DFO} in solving problems like multinomial logistic regression and variational image denoising, respectively.
\subsection{Notation}
We denote the parameters as $\theta \in \Theta \subseteq \mathbb{R}^d$, the $k$-th iterate of parameters as $\theta_k$, and each element of the vector of parameters as $\theta[i]$, where $1\leq i \leq d$. The notation $\|\cdot\|$ represents the Euclidean norm of vectors and the 2-norm of matrices. The gradient of the lower-level objective function $h_i: \mathbb{R}^n \times \mathbb{R}^d \rightarrow \mathbb{R}$ with respect to $x$ is expressed by $\nabla_x h_i$. Throughout the paper, we denote the minimizer of $h_i$ w.r.t.\ $x$ for a fixed $\theta$ by $\hat{x}_i$ and the approximation of it by $\tilde{x}_i$. Additionally, the Hessian of $h_i$ with respect to $x$ is represented as $\nabla^2_x h_i$ and the Jacobian of the lower-level objective $h_i$ is denoted as $\nabla^2_{x\theta} h_i$. Moreover, $\partial\hat{x}_i: \mathbb{R}^d \rightarrow \mathbb{R}^n \times \mathbb{R}^d$ stands for the derivative of $\theta \mapsto \hat{x}_i(\theta)$ with respect to $\theta$.
\section{Background}\label{sec:background}
For ease of notation, we will only consider one data point \rv{in \eqref{generalBilevel}} and the index of the data points $i$ is omitted in the following discussion. Similarly, we do not consider the regularizer $r$ \rv{in \eqref{UPPER_LEVEL}}. Thus, our general bilevel learning problem we will study in the main part of the paper takes the form:
\begin{subequations}\label{generalBilevel}
    \begin{align}
        \min_{\theta \in \mathbb{R}^d} \{f(\theta) &\coloneqq g(\hat{x}(\theta))\}\label{uppergeneral}\\
         \rvv{s.t. \quad} {\hat{x}(\theta)} &\coloneqq \arg\min_{x\in \mathbb{R}^n} h(x,\theta).\label{lowergeneral}
    \end{align}
\end{subequations}
Now, we consider the following assumptions.
\begin{assumption}\label{ass1}
\rv{
We make the following assumptions on the lower level loss $h$:
\begin{enumerate}[label = \Roman *.]
    \item There exist $\mu(\theta), \ L(\theta) \in \mathbb{R}$, $0<\mu(\theta)\leq L(\theta)$ such that $\mu(\theta)I \preceq \nabla_{x}^2h(x,\theta) \preceq L(\theta) I$. Moreover, $\nabla_xh$ and $\nabla_{x}^2h$ are continuous in $\theta$.\label{assumption1.1}
    \item $\nabla_x^2h(x,\theta)$ is $L_H$ Lipschitz continuous in $x$ and $\nabla_{x\theta}^2h(x,\theta)$ is $L_J$ Lipschitz in $x$ uniformly for all $\theta$.\label{assumption1.2}
\end{enumerate}
}
\end{assumption}
\begin{assumption}
\rv{
    \label{ass2}
    The function $f$ is $L_{\nabla f}$-smooth and $g$ is $L_{\nabla g}$-smooth, which means $f$ and $g$ are continuously differentiable with $L_{\nabla f}$ and $L_{\nabla g}$ Lipschitz gradients, respectively. Moreover, $L_{\nabla f} > 0$ and $L_{\nabla g} > 0$, and the upper-level loss $g$ is bounded below.}
\end{assumption}

\begin{remark}\label{rk:different_UL_loss}
\rv{
     \cref{ass1} \ref{assumption1.1} implies that the lower-level objective function is strongly convex. \cref{ass1} \ref{assumption1.2} will be needed solely in \cref{errorBound}, which we utilize in the practical part of our work and not in our analysis. Our assumptions on the upper-level loss $g$ in \cref{ass2} encompass commonly used loss functions such as the \textit{mean-squared error} (MSE) and \textit{multinomial logistic loss} \cite{pml1Book}, as well as popular non-convex losses, including the \textit{peak signal-to-noise ratio} (PSNR), which is related to MSE, the \textit{structural similarity index measure} (SSIM) \cite{Crockett_2022}, and the \textit{biweight loss} \cite{Tuckey,Carmon2017ConvexUP}. Our method and theory considers non-convex $g$, but remain applicable—and can be improved with tighter bounds—for a convex  $g$, with slight modifications to the proofs and the line search conditions we derive.}
\end{remark}

\begin{remark}\label{about_smoothness}
    \rvv{Note that having $\mu(\theta)$ and $L(\theta)$ uniformly bounded in \cref{ass1} \ref{assumption1.1} (i.e., there exist $0<\mu \leq L<\infty$ such that $\mu \leq \mu(\theta)\leq L(\theta) \leq L$)  is sufficient to ensure that $f$ is $L_{\nabla f}$-smooth, but it is not necessary (e.g., when $\nabla g(x) = 0$). Moreover, the existence of  $L_{\nabla f}$  is required for our convergence analysis, but we do not use its value in our algorithm or practical applications.
    }
\end{remark}
The gradient of $f$ takes the form below 
\begin{equation}
    \nabla f(\theta) = \nabla (g \circ \hat{x})(\theta)= \partial \hat{x}(\theta)^T \nabla g(\hat{x}(\theta)).
\end{equation}
For all $x\in \mathbb{R}^n$, under \cref{ass1} and with the use of the implicit function theorem \cite{IFT_Benjio}, we can state 
$$\partial \hat{x}(\theta)^T \coloneqq  -\nabla_{x\theta}^2h(x,\theta)^T\nabla_x^2h(x,\theta)^{-1}.$$ 

Based on \cref{ass1} and \cref{ass2}, we have the following useful inequalities, which we will use in later sections.
   As stated in \cite[Lemma 1.2.3]{NestrovBook}, if  $f:\mathbb{R}^d \rightarrow \mathbb{R}$ is a $L_{\nabla f}$-smooth function, then for any $x,y\in\mathbb{R}^d$ we have
  \begin{equation}\label{descent_lemma}
      f(y) \leq f(x) +  \nabla f(x)^T( y-x ) +\frac{L_{\nabla f}}{2} \|y - x\|^2.
  \end{equation}
Note that (see e.g.\ \cite{NestrovBook}) if $f:\mathbb{R}^d \rightarrow \mathbb{R}$ is convex and smooth, then for any $x,y\in\mathbb{R}^d$ we have
  \begin{equation}\label{convexity}
      f(y)  \geq  f(x) +  \nabla f(x)^T( y-x ).
  \end{equation}
At each iteration $k\in \mathbb{N}_0$, as finding $  {\hat{x}(\theta_k)}$ exactly is not feasible due to the usage of numerical solvers, we solve \eqref{lowergeneral} inexactly with tolerance $\epsilon_k$ to find    {$\Tilde{x}(\theta_k)$} such that $\|   {\Tilde{x}(\theta_k)} -   {\hat{x}(\theta_k)}\| \leq \epsilon_k$ using the a posteriori bound $\|\nabla_xh(   {\Tilde{x}(\theta_k)},\theta_k)\|\leq \epsilon_k \mu(\theta_k)$. This a posteriori bound is derived from the strong convexity of the lower-level objective function, as explained in \cite{DFO}. Also, instead of computing $\nabla_x^2h(   {\Tilde{x}(\theta_k)},\theta_k)^{-1}\nabla g(\tilde{x}(\theta_k))$, we solve the linear system $\nabla_x^2h(\Tilde{x}(\theta_k),\theta_k)q_k = \nabla g(\Tilde{x}(\theta_k))$ with residual $\delta_k$. Then, we denote the inexact gradient (hypergradient) of the upper-level problem \eqref{uppergeneral} at each iteration $k$ by $z_k \coloneqq -\nabla_{x \theta}h(   {\Tilde{x}(\theta_k)},\theta_k)^Tq_k$ and the corresponding error by $e_k \coloneqq z_k - \nabla f(\theta_k)$.

The following theorem provides a computable a posteriori bound for the error of the approximate hypergradient, which we utilize in the next section.
\begin{theorem}
\label{errorBound}
(\cite[Theorem 10]{boundMatthias}) Suppose Assumption \ref{ass1} hold. Let us define
$$c(x(\theta)) \coloneqq \frac{L_{\nabla_g} \|\nabla_{x\theta}^2h(x(\theta),\theta)\|}{\mu\rv{(\theta)}}+L_{H^{-1}}\|\nabla g(x(\theta))\|\|\nabla_{x\theta}^2h(x(\theta),\theta)\| + \frac{L_J \|\nabla g(x(\theta))\|}{\mu\rv{(\theta)}},$$ where $L_{H^{-1}}$ is the Lipschitz constant of $\nabla_x^2h(x,\theta)^{-1}$ in $x$ uniformly for all $\theta$.\footnote{From \cref{ass1}, it follows that $\nabla_x^2 h^{-1}$ is Lipschitz continuous.} Then, at each iteration $k=1,2, \dots$, we have the following a posteriori bound for the inexact gradient:
\begin{equation}\label{AP_bound_hg}
    \|e_k\| = \|z_k - \nabla f(\theta_k) \| \leq c(\Tilde{x}(\theta_k))\epsilon_k + \frac{\|\nabla_{x\theta}^2h(\Tilde{x}(\theta_k),\theta_k)\|}{\mu\rv{(\theta)}}\delta_k + \frac{L_JL_{\nabla g}}{\mu\rv{(\theta)}}\epsilon_k^2.
\end{equation}
\end{theorem}

\section{Method of Adaptive Inexact Descent (MAID) and Convergence Analysis} \label{sec:Theory_algorithm}

In this section, we propose the \textit{Method of Adaptive Inexact Descent} (MAID) for solving the problem \cref{generalBilevel}, stated in \cref{HOAH_back_new}. This algorithm uses the update 
\begin{equation}\label{gdUpdate}
    \theta_{k+1} = \theta_k - \alpha_k z_k, \quad k=0,1,\dots,
\end{equation} 
where $z_k$ represents the inexact hypergradient at iteration $k$, as computed in \cref{Inexact_grad}. An appropriate step size $\alpha_k > 0$ is computed using backtracking. However, conventional backtracking line search rules, such as the \textit{Armijo rule} \cite{Nocedal2006Numerical}, require access to an exact function evaluation. This could be available up to machine precision in \eqref{generalBilevel} but would be computationally prohibitively expensive. To address this challenge, we introduce a verifiable line search method that generalizes the Armijo rule to the inexact setting utilizes only an approximation of $f$. We control the accuracies to ensure that $-z_k$ is a descent direction. Our analysis begins with the Armijo rule in the exact setting, where we have access to $f$ and we can utilize the negative of the exact gradient, that is $-\nabla f(\theta_k)$, as the descent direction.

The computation of the approximate hypergradient $z_k$ in \cref{Inexact_grad} employs the implicit function theorem and numerical solvers, following the same procedure as outlined in the previous section, with tolerances $\epsilon_k$ and $\delta_k$. It checks \rv{whether} the sufficient condition \rv{in \cref{sec:descent_direction} (\cref{prop1,boundNablaZ_lemma}) is met} for $-z_k$ to be a descent direction for $f$. If not, it shrinks $\epsilon_k$ and $\delta_k$, recalculates $z_k$, and repeats this process until $-z_k$ satisfies the requirement \rv{in \cref{prop1,boundNablaZ_lemma}}. With the approximate hypergradient $z_k$ and the tolerance $\epsilon_k$ which was used to compute it, \cref{HOAH_back_new} employs the line search rule $\psi(\alpha_k) \leq 0$, where $\psi$ is defined in \eqref{psi}, utilizing approximations and bounds for $f(\theta_k)$ and $\nabla f(\theta_k)$. If it fails to find a suitable step size and reaches the maximum number of backtracking iterations, it reduces $\epsilon_k$, recalculates $z_k$, and retries the backtracking process with an increased maximum number of iterations until it successfully finds an appropriate step size and completes the descent update. This backtracking failure can occur due to insufficient iterations to capture a very small step size or $\epsilon_k$ that is not small enough. \rv{Note that in the process of increasing the number of backtracking iterations upon failure, as shown in line \ref{bt_loop} of \cref{HOAH_back_new}, it may appear that the maximum number of backtracking steps can grow indefinitely. However, in our analysis in subsequent sections, particularly in \cref{inner_loop_theorem}, we demonstrate that backtracking failures are bounded, so the maximum number of backtracking steps remains finite.}
\begin{algorithm}[htbp]
\caption{Method of Adaptive Inexact Descent (MAID). Hyperparameters: $\underline{\rho} \in (0,1)$ and $\overline{\rho}>1$ control the reduction and increase of the step size $\alpha_k$, respectively; $\underline{\nu} \in (0,1)$ and $\overline{\nu}>1$ govern the reduction and increase of accuracies $\delta_k$ and $\epsilon_k$; $\max_\text{BT} \in \mathbb{N}$ is the maximum number of backtracking iterations.}\label{HOAH_back_new}
\begin{algorithmic}[1]
\State Input $\theta_{0} \in \mathbb{R}^d$, accuracies $\epsilon_{0}, \delta_{0}>0$, step size $\alpha_{0} >0$.
\For{$k=0, 1, \dots$}
\For{$j = \max_\text{BT}, \max_\text{BT}+1, \dots$ }\label{bt_loop}
\State{$z_k, \epsilon_k, \delta_k \leftarrow$ INEXACTGRADIENT($\theta_k, \epsilon_k,  \delta_k$)} \label{updated_descend_direction}
\For{$i=0,1,\dots,j-1$}\label{inner_loop}
\If{inexact sufficient decrease $\psi(\alpha_k) \leq 0$ holds}\label{descent_condition_check}\Comment{\cref{inexact_bt_lemma}}
\State{Go to line \ref{gd_update_step}}\Comment{Backtracking Successful}
\EndIf
\State{$\alpha_{k} \leftarrow \underline{\rho}\alpha_k$}
\Comment{Adjust the starting step size}
\EndFor
\State $\epsilon_k \leftarrow \underline{\nu} \epsilon_k \label{BT_decrease} $  \Comment{Backtracking Failed and needs higher accuracy}
\State $\delta_k \leftarrow \underline{\nu} \delta_k  $  \label{BT_decrease_delta}

\EndFor
\State{$\theta_{k+1} \leftarrow \theta_k - \alpha_{k} z_k$}\Comment{Gradient descent update}\label{gd_update_step}
\State{$\epsilon_{k+1} \leftarrow \overline{\nu} \epsilon_{k}$ }\label{increase_epsilon_k}\Comment{Increasing $\epsilon_k$}
\State{$\delta_{k+1} \leftarrow \overline{\nu} \delta_{k}$ }\label{increase_delta_k}\Comment{Increasing $\delta_k$}
\State{$\alpha_{k+1} \leftarrow \overline{\rho} \alpha_{k}$}\label{increase_beta_k}\Comment{Increasing $\alpha_k$}
\EndFor
\end{algorithmic}
\end{algorithm}

\begin{algorithm}
\caption{Calculating an inexact hypergradient, ensuring it is a descent direction. Hyperparameters: $\eta \in (0,1)$ ensures that the computed inexact gradient $z$ is a descent direction; $\underline{\nu} \in (0,1)$ is used to reduce the accuracies $\delta$ and $\epsilon$.}\label{Inexact_grad}
\begin{algorithmic}[1]
\State{Input: $\theta \in \mathbb{R}^d$, accuracies $\epsilon, \delta >0$.}
\Function{InexactGradient}{$\theta, \epsilon,\delta$}
\While{True}\label{omega_while}
\State{Solve lower-level problem to find $\tilde{x}(\theta)$ such that $\| \nabla_{x}h(\tilde{x}(\theta),\theta)\| \leq \epsilon \mu$.\label{LL}}
\State{Solve $\nabla_x^2 h(\tilde{x}(\theta),\theta) q = \nabla g(\tilde{x}(\theta))$ with tolerance $\delta$.}
\State{Calculate $z = - (\nabla^2_{x\theta} h(\tilde{x}(\theta),\theta))^T q$.}\label{step:Jacob}
\State{$\omega \leftarrow$ Calculate upper bound \eqref{AP_bound_hg} for error $\|e\|$ using \cref{errorBound}.}\label{boundstep}
\If{$\omega \leq (1-\eta)\|z\|$}
\State{\Return $z, \epsilon, \delta$}
\EndIf
\State{$\delta \leftarrow \underline{\nu} \delta$,}\label{decrease_delta}
\State{$\epsilon \leftarrow \underline{\nu} \epsilon$.}\label{decrease_eps}
\EndWhile
\EndFunction
\end{algorithmic}
\end{algorithm}

\subsection{Line search with approximate hypergradient and inexact function evaluations}
In order to find a suitable step size when we do not have any information about the Lipschitz constant of the gradient, line search methods are usually employed \cite{Nocedal2006Numerical}.  In Armijo rule, at each iteration $k$, for a given direction $z_k$ and a starting step size $\beta_k$, the task is to find the smallest $i_k \in \mathbb{N}_0$ such that 
\begin{equation}\label{backtrack_linesearch}
    f(\theta_k - \beta_k \rho^{i_k} z_k) \leq f(\theta_k) - \zeta \beta_k \rho^{i_k}  \nabla f(\theta_k)^T z_k
\end{equation}
holds for $0<\rho,\zeta<1$. 
After finding $i_k$, we denote step size by $\alpha_k \coloneqq \beta_k \rho^{i_k}$. 
In our setting, we do not have access to the function $f$ nor its gradient; instead, we only have $g(\tilde{x}(\theta_k))$ and an inexact direction $z_k$. Therefore, we will introduce a line search condition with inexact gradient and function evaluations to find a suitable step size.

We commence by deriving a condition under which the approximate hypergradient is a descent direction. Subsequently, we establish computable upper and lower bounds for the exact upper-level function by leveraging the inexact components. Moving forward, we introduce our line search.

\paragraph{Descent direction}\label{sec:descent_direction} The following proposition provides a sufficient condition to ensure that $-z_k$ is a descent direction for $f$ at $\theta_k$, that is $ z_k^T \nabla f(\theta_k) > 0 $. 
\begin{proposition}\label{prop1} Let $e_k \coloneqq z_k - \nabla f(\theta_k)$ be the error in the approximate hypergradient. If $\|e_k\| < \|z_k\|$, then $-z_k$ is a descent direction for $f$ at $\theta_k$.
\end{proposition}

\begin{proof}
Since
$$     z_k^T \nabla f(\theta_k) =  z_k^T (z_k -e_k) = \|z_k\|^2 - z_k^Te_k,$$
$-z_k$ is a descent direction if and only if $z_k^Te_k < \|z_k\|^2.$

Now, from the assumption, by multiplying both sides of the inequality $\|e_k\| < \|z_k\|$ by $\|z_k\| > 0$ we have $$\|e_k\| \| z_k \| < \| z_k \|^2.$$ Moreover, using Cauchy--Schwarz, we have $$z_k^Te_k \leq \| z_k \| \| e_k\|< \|z_k\|^2,$$
thus $-z_k$ is a descent direction.
\end{proof}
Note that the condition ${\|e_k\|}<\|z_k\|$ in \cref{prop1} must hold uniformly to demonstrate the convergence of a descent method \cite[Theorem 3.2]{Nocedal2006Numerical}. Therefore, instead, we will assume the stronger condition ${\|e_k\|}\leq (1-\eta) {\|z_k\|}$ for $0<\eta<1$. 

In addition, due to the unavailability of $\nabla f(\theta_k)$, we establish a computable bound for $z_k^T\nabla f(\theta_k)$ to facilitate our analysis.
\begin{lemma}\label{boundNablaZ_lemma}
Let ${\|e_k\|}\leq (1-\eta) {\|z_k\|}$. Then,
$\eta \|z_k\|^2 \leq \nabla f(\theta_k)^T z_k.$
\end{lemma}
\begin{proof}
Combining Cauchy--Schwarz and ${\|e_k\|}\leq (1-\eta){\|z_k\|}$ yields
$e_k^Tz_k \leq \|z_k\|^2 - \eta \|z_k\|^2. $ 
Since $\nabla f(\theta_k)^T z_k = \|z_k\|^2 - e_k^Tz_k$, we find the desired bound.
\end{proof}

We now derive computable estimates for $f(\theta_k)$ which will be used to establish our alternative sufficient decrease condition. 

\paragraph{Verifiable line search with approximate hypergradient}\label{sec:bt_linesearch}
\rv{To obtain an inexact and computable sufficient decrease condition, we first derive computable upper and lower bounds for $f$ in \cref{computable_bounds}. Then, \Cref{inexact_bt_lemma} provides a backtracking line search condition for sufficient decrease. We adopt the inequality $\psi(\alpha_k) \leq 0$ from \eqref{psi} as our inexact backtracking line search rule in \cref{HOAH_back_new}, where, for $\alpha_k = \underline{\rho}^{i_k} \beta_k$, $i_k \geq 0$ is the smallest integer satisfying the inequality. Note that this is a practical line search as we can compute all of its components.}
\begin{lemma}\label{computable_bounds}
    \rv{Let $g$ be $L_{\nabla g}$-smooth, with $\hat{x}(\theta_k)$ as defined in \eqref{lowergeneral}, and let $\tilde{x}(\theta_k)$ be an approximation of \rvv{$\hat{x}(\theta_k)$} such that $\|\hat{x}(\theta_k) - \tilde{x}(\theta_k)\| \leq \epsilon_k$.} We have the following lower and upper bounds for $f(\theta_k) = g(  {\hat{x}(\theta_k)})$\rvv{:}
    \begin{subequations}
    \begin{align}
        &g(  {\hat{x}(\theta_k)}) \leq g(   {\Tilde{x}(\theta_k)})  + \| \nabla g(   {\Tilde{x}(\theta_k)}) \| \epsilon_k + \frac{L_{\nabla_g}}{2}\epsilon_k^2, \label{upperApprox}\\
    &g(  {\hat{x}(\theta_k)}) \geq g(   {\Tilde{x}(\theta_k)}) - \| \nabla g(   {\Tilde{x}(\theta_k)})\| \epsilon_k \rv{- \frac{L_{\nabla_g}}{2}\epsilon_k^2}\rvv{.} \label{lowerApprox}
    \end{align}
\end{subequations}
\end{lemma}
\begin{proof}
    Since $g$ is $L_{\nabla g}$-smooth, using the inequality \eqref{descent_lemma} we have
\begin{equation*}
     \rv{g(  {\hat{x}(\theta_k)})} \leq \rv{g(   {\Tilde{x}(\theta_k)})} + \nabla g(\Tilde{x}(\theta_k))^T(\hat{x}(\theta_k) - \Tilde{x}(\theta_k)) +\frac{L_{\nabla g}}{2} \|\Tilde{x}(\theta_k) - \hat{x}(\theta_k)\|^2.
\end{equation*}
Now, by rearranging and utilizing the Cauchy-Schwarz inequality, we derive
\begin{equation*}
     g(  {\hat{x}(\theta_k)}) - g(   {\Tilde{x}(\theta_k)})    \leq \frac{L_{\nabla g}}{2}\|\Tilde{x}(\theta_k) - \hat{x}(\theta_k)\|^2 +\| \nabla g(\Tilde{x}(\theta_k))\| \|(\Tilde{x}(\theta_k) - \hat{x}(\theta_k))\|.
\end{equation*}
Since $\|\Tilde{x}(\theta_k) - \hat{x}(\theta_k)\| \leq \epsilon_k$, the inequality
 \eqref{upperApprox} holds.

Furthermore, using \eqref{upperApprox}, we can write 
$$
g(  {\hat{x}(\theta_k)}) -g(   {\Tilde{x}(\theta_k)}) \leq    \| \nabla g(   {\Tilde{x}(\theta_k)}) \| \epsilon_k + \frac{L_{\nabla_g}}{2}\epsilon_k^2.
$$
Since the right-hand-side of the inequality above is positive, we have
$$
g(  {\hat{x}(\theta_k)}) -g(   {\Tilde{x}(\theta_k)}) \geq - \| \nabla g(   {\Tilde{x}(\theta_k)}) \| \epsilon_k - \frac{L_{\nabla_g}}{2}\epsilon_k^2,
$$
which yields the desired inequality \eqref{lowerApprox}.
\end{proof}
\rv{
\begin{remark}\label{rk:convex_bound}
    Note that if $g$ is $L_{\nabla g}$-smooth and convex, from \eqref{convexity}, we can write 
    \begin{equation*}
        f(\theta_k) = g(  {\hat{x}(\theta_k)})\geq g(   {\Tilde{x}(\theta_k)}) + \nabla g(   {\Tilde{x}(\theta_k)})^T(  {\hat{x}(\theta_k)} -    {\Tilde{x}(\theta_k)}).
    \end{equation*}
Following steps similar to \cref{computable_bounds}, we obtain
    \begin{equation}\label{convexApprox}
        g(\hat{x}(\theta_k)) \geq g(\tilde{x}(\theta_k)) - \|\nabla g(\tilde{x}(\theta_k))\|\epsilon_k, 
    \end{equation}
    which provides a tighter bound than \eqref{lowerApprox}.
\end{remark}
}

The following lemma employs the bounds \eqref{upperApprox} and \eqref{lowerApprox} to provide a line search rule such that all of its components are accessible. 
\begin{lemma}\label{inexact_bt_lemma}
    Let $\lambda \in \mathbb{R}$ and suppose that \rv{$g$ is $L_{\nabla g}$-smooth}. Denote 
    \begin{equation*}
    \begin{split}
        \overline{U}(x,\epsilon) &\coloneqq g(x) + \|\nabla g(x)\|\epsilon + \frac{L_{\nabla_g}}{2}\epsilon^2, \\ 
        \underline{U}(x,\epsilon) &\coloneqq g(x) - \|\nabla g(x)\|\epsilon \rv{- \frac{L_{\nabla_g}}{2}\epsilon^2},
            \end{split}
    \end{equation*}
    and
    \begin{equation}\label{psi}
        \psi(\alpha_k) \coloneqq \overline{U}(\tilde{x}(\theta_{k+1}), \epsilon_{k+1}) - \underline{U}(\tilde{x}(\theta_{k}), \epsilon_{k})+\lambda \alpha_k \|z_k\|^2.
    \end{equation}
    
    If the backtracking line search condition
$
        \psi(\alpha_k) \leq 0
$
     is satisfied, then the sufficient descent condition $f(\theta_{k+1})-f(\theta_{k})\leq -\lambda \alpha_k \|z_k\|^2$ holds. 
\end{lemma}
\begin{proof}

Using \eqref{upperApprox} for $\theta_{k+1}$ and  \eqref{lowerApprox} for $\theta_{k}$,  we have 
\begin{equation*}
      f(\theta_{k+1}) - f(\theta_k) \leq g(\tilde{x}(\theta_{k+1})) + \| \nabla g(\tilde{x}(\theta_{k+1})) \| \epsilon_{k+1} + \frac{L_{\nabla_g}}{2}\epsilon_{k+1}^2  - g(\tilde{x}(\theta_k)) +  \| \nabla g(\tilde{x}(\theta_k)) \| \epsilon_k \rv{+ \frac{L_{\nabla_g}}{2}\epsilon_{k}^2}.
\end{equation*}
The inequality above together with the definition of $\overline{U}(x,\epsilon)$, $\underline{U}(x,\epsilon)$, and $\psi(\alpha_k) \leq 0$ imply $f(\theta_{k+1})-f(\theta_{k})\leq -\lambda \alpha_k \|z_k\|^2$ as required. 
\end{proof}

\rv{
\begin{remark}\label{inexact_bt_convex}
    If \(g\) is convex, then by using \eqref{convexApprox} instead of \eqref{lowerApprox} in \cref{inexact_bt_lemma} and defining
    \begin{equation}\label{psi_convex}
        \tilde{\psi}(\alpha_k) \coloneqq \psi(\alpha_k) - \frac{L_{\nabla g}}{2}\epsilon_k^2,
    \end{equation}
    the sufficient descent condition \(f(\theta_{k+1}) - f(\theta_{k}) \leq -\lambda \alpha_k \|z_k\|^2\) holds if \(\tilde{\psi}(\alpha_k) \leq 0\) is satisfied.
\end{remark}
}

Now, since the components of the line search rule $\psi(\alpha_k) \leq 0$ depend on $\epsilon_k$, we investigate the accuracy $\epsilon_k$ for which a step size $\alpha_k$ exists that satisfies $\psi(\alpha_k) \leq 0$ a priori.

\paragraph{Existence of a step size in the line search}\label{sec:step_existence}
\rv{In the following discussion, our a priori analysis demonstrates that, given a non-stationary $\theta_k$ and an inexact hypergradient $z_k$ serving as a descent direction, there exists a non-empty, non-singular interval of accuracies for which, for each accuracy, a corresponding interval of step sizes satisfies the line search condition $\psi(\alpha_k)\leq 0$ \eqref{psi} introduced in \cref{inexact_bt_lemma}. The following auxiliary lemma aids in achieving the aforementioned goals.}

\begin{lemma}\label{BacktrackError}
Let $\lambda\in \mathbb{R}$, $\eta\leq 1$, and ${\|e_k\|}\leq (1-\eta){\|z_k\|}$. Denote
\begin{equation}\label{grad_summation}
    w_k \coloneqq \|\nabla g(\tilde{x}(\theta_k))\|+ \|\nabla g(\tilde{x}(\theta_{k+1}))\|,
\end{equation}
$\bar{\epsilon}_k \coloneqq \max\{\epsilon_k , \epsilon_{k+1}\}$, and $\phi(\alpha_k) \coloneqq 2w_k\bar{\epsilon}_k + \rv{2}L_{\nabla_g}\bar{\epsilon}_k^2+ (\alpha_k^2 \frac{L_{\nabla f}}{2}+ (\lambda-\eta ) \alpha_k)\|z_k\|^2$. 
Then, 
\begin{equation}
    \begin{aligned}
        \psi(\alpha_k) \leq \phi(\alpha_k),
    \end{aligned}
\end{equation}
\rv{
where $\psi$ is as defined in \cref{inexact_bt_lemma}.}
\end{lemma}
\begin{proof}
    Since $f(\theta) = g(\hat{x}(\theta))$ is $L_{\nabla f}-$smooth, applying \eqref{descent_lemma} to $g(\hat{x}(\theta_{k+1})) = g(\hat{x}(\theta_k - \alpha_k z_k))$ in \eqref{lowerApprox} for $\theta_{k+1}$ gives
    \begin{align*}
        g(\tilde{x}(\theta_{k+1})) 
        &\leq g(\hat{x}(\theta_{k+1})) + \| \nabla g(\tilde{x}(\theta_{k+1}))\|\epsilon_{k+1} \rv{+ \frac{L_{\nabla g}}{2} \epsilon_{k+1}^2} \notag\\
        &\leq g(\hat{x}(\theta_k))  - \alpha_k \nabla f(\theta_k)^Tz_k + \alpha_k^2 \frac{L_{\nabla f}}{2} \|z_k\|^2 + \| \nabla g(\tilde{x}(\theta_{k+1}))\|\epsilon_{k+1} \rv{+ \frac{L_{\nabla g}}{2} \epsilon_{k+1}^2}.
    \end{align*}
   
   Leveraging the bounds \eqref{upperApprox} for $g(\hat{x}(\theta_{k}))$ and \cref{boundNablaZ_lemma} for $\nabla f(\theta_k)^Tz_k$, we derive
    \begin{multline*}
     g(\tilde{x}(\theta_{k+1})) \leq g(\tilde{x}(\theta_{k})) + \|\nabla g(\tilde{x}(\theta_{k}))\|\epsilon_{k} + \frac{L_{\nabla g}}{2} \epsilon_k^2 + \left(\alpha_k^2 \frac{L_{\nabla f}}{2} - \alpha_k \eta \right)\|z_k\|^2 \\ + \|\nabla g(\tilde{x}(\theta_{k+1}))\|\epsilon_{k+1} \rv{+ \frac{L_{\nabla g}}{2} \epsilon_{k+1}^2}.
    \end{multline*}

Adding $\|\nabla g(\tilde{x}(\theta_{k+1})) \| \epsilon_{k+1} + \frac{L_{\nabla_g}}{2}\epsilon_{k+1}^2 - g(\tilde{x}(\theta_k)) + \|\nabla g(\tilde{x}(\theta_k))\| \epsilon_k \rv{+ \frac{L_{\nabla_g}}{2}\epsilon_k^2} + \lambda \alpha_k \|z_k\|^2$ to both sides we obtain
    \begin{align*}
     \psi(\alpha_k) 
     \leq 2\|\nabla g(\tilde{x}(\theta_{k}))\|\epsilon_{k} + 2 \|\nabla g(\tilde{x}(\theta_{k+1})) \|\epsilon_{k+1} + {L_{\nabla g}} \big(\epsilon_k^2 + \epsilon_{k+1}^2\big) + \left(\alpha_k^2 \frac{L_{\nabla f}}{2} + \alpha_k (\lambda - \eta) \right)\|z_k\|^2 .
    \end{align*}

Now, utilizing $\Bar{\epsilon}_k = \max\{\epsilon_k, \epsilon_{k+1}\} $, and $w_k = \| \nabla g(\tilde{x}(\theta_k))\| + \| \nabla g(\tilde{x}(\theta_{k+1}))\|$, we derive the desired result.
\end{proof}
\rv{The lemma below shows the validity of the line search condition in \cref{inexact_bt_lemma} by establishing the existence of a non-empty interval of step sizes that satisfy it, given an interval of valid accuracies specified a priori. } 
\begin{lemma}\label{BacktrackError2}
Assume $\|z_k\| \neq 0$, $\lambda < \eta$, and ${\|e_k\|}\leq (1-\eta){\|z_k\|}$. Considering $w_k$ and $\bar{\epsilon}$ as defined in \cref{BacktrackError} and denoting
\begin{equation}\label{epsilon_upper}
    s_k \coloneqq \frac{1}{\rv{2}L_{\nabla g}}{\left(\sqrt{w_k^2 + \frac{L_{\nabla g}}{L_{\nabla f}} (\eta -\lambda)^2\|z_k\|^2}-w_k\right)},
\end{equation} 
for all 
$
    \bar{\epsilon}_k \in  [0 , s_k),
$
there exist $0\leq \underline{\alpha}_k < \overline{\alpha}_k$ such that for all $\alpha_k \in [\underline{\alpha}_k,\overline{\alpha}_k]$ the sufficient decrease condition $\psi(\alpha_k) \leq 0$ holds. Moreover, denoting 
\begin{equation}\label{ratio_alpha}
    \hat{s}_k\coloneqq \sqrt{(\eta-\lambda)^2- \frac{\rv{4}L_{\nabla f}(w_k\bar{\epsilon}_k + L_{\nabla g}\bar{\epsilon}_k^2)}{\|z_k\|^2}}>0,
\end{equation}
the interval may be defined by
\begin{equation}\label{alpha_min_max}
    \underline{\alpha}_k = \frac{1}{L_{\nabla f}}\left({\eta - \lambda}-\hat{s}_k\right), \quad \overline{\alpha}_k = \frac{1}{L_{\nabla f}}\left({\eta - \lambda}+\hat{s}_k\right).
\end{equation}
Note that $\lim_{\bar{\epsilon_k} \to 0} \underline{\alpha}_{k} = 0$, $\lim_{\bar{\epsilon_k} \to 0} \overline{\alpha}_{k} = \frac{2(\eta - \lambda)}{L_{\nabla f}} > 0$, and $\hat{s}_k$ is monotonically decreasing in $\bar{\epsilon}_k$.
\end{lemma}
\begin{proof}

From \cref{BacktrackError} it follows that $\phi(\alpha_k)\leq 0$ implies $\psi(\alpha_k) \leq 0$. Note that $\phi(\alpha_k)$ is quadratic in $\alpha_k$ with positive leading coefficient. The roots of $\phi$ are given by
 $\underline{\alpha}_k$ and $\overline{\alpha}_k$ \eqref{alpha_min_max}. These roots are real if 
 \begin{equation}\label{existenceAlpha}
    (\eta-\lambda)^2\|z_k\|^2 - \rv{4} L_{\nabla f} (w_k\bar{\epsilon}_k + L_{\nabla g}\bar{\epsilon}_k^2) \geq 0.
\end{equation} 
Following a similar argument, the inequality \eqref{existenceAlpha} holds if $\bar{\epsilon}_k \leq s_k$, given that the left-hand side of \eqref{existenceAlpha} is quadratic in $\bar{\epsilon}_k$ with a negative leading coefficient and roots $s_k$ and $-s_k$. Moreover, $s_k > 0$ by definition, as $\|z_k\|\neq 0$. Since $\bar{\epsilon}_k\geq 0$, we have $\bar{\epsilon}_k \in [0, s_k)$.  Referring to \eqref{existenceAlpha} and the definition of $\hat{s}_k$, we conclude that $0<\hat{s}_k\leq {\eta-\lambda}$, implying $0 \leq\underline{\alpha}_k < \overline{\alpha}_k$.
\end{proof}
As $s_k$ \eqref{epsilon_upper} depends on ${\epsilon}_k$ and ${\epsilon}_{k+1}$, and thus on $\bar{\epsilon}_k$, we need the following lemmas to show that for $\bar{\epsilon}_k$ sufficiently small, the inclusion $\bar{\epsilon}_k \in  [0 , s_k)$ holds a priori  for any $z_k$, $\tilde{x}(\theta_k)$, and $\tilde{x}(\theta_{k+1})$ satisfying the conditions of computing the inexact gradient and the solution of the lower-level problem. It shows that if $\theta_k$ is non-stationary, there exists a nonempty positive range of accuracies for which we can find a nonempty positive range of step sizes satisfying the line search condition $\psi(\alpha_k) \leq 0$. For simplicity, we set $\delta_k = \bar{\epsilon}_k$ for the remainder of this section, however, the general case holds with a similar argument.
\begin{lemma}\label{nonzero_grad}
    Let $\|\nabla f(\theta_k)\|>0$, $s_k$ be as defined in \eqref{epsilon_upper}, and $\lambda \neq \eta$, then $\lim_{\bar{\epsilon_k}\to 0} s_k >0$.
\end{lemma}
\begin{proof}
    As $\lim_{\bar{\epsilon_k} \to 0} \tilde{x}(\theta_k) = \hat{x}(\theta_k)$, $\lim_{\bar{\epsilon_{k}} \to 0} \tilde{x}(\theta_{k+1}) = \hat{x}(\theta_{k+1})$, and $\nabla g$ is continuous, we have 
    $
        \lim_{\bar{\epsilon_k} \to 0}w_k =  \bar{w}_k, 
    $ for some $\bar{w}_k \geq 0$. Moreover, 
    since $\lim_{\bar{\epsilon_k} \to 0} z_k = \nabla f(\theta_k)$,
     $\|\nabla f(\theta_k)\| >0$, $\lambda\neq\eta$, and $L_{\nabla f}>0$,
    \begin{equation*}
                \lim_{\bar{\epsilon_k} \rightarrow 0} s_k = \frac{1}{\rv{2}L_{\nabla g}}{\left (\sqrt{\bar{w}_k^2 + \frac{L_{\nabla g}}{L_{\nabla f}} (\eta -\lambda)^2\|\nabla f(\theta_k)\|^2}-\bar{w}_k\right)}>0.
    \end{equation*} 
    as desired.
\end{proof}

Now, we need the following lemma to show that for $\bar{\epsilon}_k$ small enough, the inclusion $\bar{\epsilon}_k \in [0, s_k)$ holds a priori.
\begin{lemma}\label{abstract_sequences}
    Let $\{a_j\}_{j=0}^\infty \subseteq [0,\infty)$ and $\{b_j\}_{j=0}^\infty \subseteq [0,\infty)$ be sequences, with $a_j\rightarrow a>0$ and $b_j \rightarrow 0$, respectively. Then, there exists $J\in \mathbb{N}$ such that for all $j \geq J$,  $b_j \in [0,a_j]$.
\end{lemma}


\begin{proof}
    Since $a_j\rightarrow a>0$, 
     there exists $J_1\in \mathbb{N}, \ \text{such that for all} \  j\geq J_1,\ a_j  \geq \frac{a}{2}.$
    On the other hand, as $b_j \rightarrow 0$, 
    {there exists } $J_2\in \mathbb{N}, \ \text{such that for all} \  j\geq J_2,\ b_j \leq \frac{a}{2}.$
    Taking $J = \max \{J_1, J_2\}$, we have 
    $\forall j\geq J,\ 0\leq b_j \leq \frac{a}{2} \leq a_j,$
    which gives us the desired conclusion. 
\end{proof}

\begin{corollary}\label{corollary}
    Let $\|\nabla f(\theta_k)\|>0$, $\bar{\epsilon}_k\to0$, and $0<\lambda<\eta<1$. Then, there exists $\epsilon>0$ such that if $\bar{\epsilon}_k\leq \epsilon$, then $\bar{\epsilon}_k <s_k$. 

\end{corollary}
That means, if the accuracy is small enough, there exist a range of step sizes satisfying $\psi(\alpha_k)\leq 0$. Note that $\epsilon$ is independent of $z_k$, $\tilde{x}(\theta_k)$, and $\hat{x}(\theta_k)$.
\begin{proof}
    The assertion follows by setting $b_j = (\bar{\epsilon}_k)_j$ and  $a_j =(s_k)_j$ using \cref{nonzero_grad} and \cref{abstract_sequences}.
\end{proof}
\normalsize
\begin{lemma}\label{step_hit_interval} 
    Let $\lambda < \eta$, $0<\underline{\rho}<1$, \rv{and $\beta_k > 0$ as the initial step-size guess, with $\alpha_k = \beta_k \underline{\rho}^j$ for $j \in \mathbb{N}_0$}. Let $\underline{\alpha}_{k}$ and $\overline{\alpha}_{k}$ be as defined in \eqref{alpha_min_max}. 
    \begin{itemize}
        \item If $\bar{\epsilon}_k\geq 0$ \rvv{is} sufficiently small, then there exists $j_k \in \mathbb{N}_0$ such that $\underline{\alpha}_{k}\leq \beta_k \underline{\rho}^{j_k}\leq \overline{\alpha}_{k}$, i.e.  $\psi(\alpha_k) \leq 0$ holds for $\alpha_k = \beta_k \underline{\rho}^{j_k}$.
        \item If $i_k>0$ is the smallest integer for which  $\psi(\alpha_k) \leq 0$ holds, then $\alpha_k= \beta_k \underline{\rho}^{i_k}> \underline{\rho} \overline{\alpha}_k$.
    \end{itemize}
\end{lemma}
\begin{proof}
By \cref{BacktrackError2}, for all $\bar{\epsilon}_k\in[0,s_k)$, $\overline{\alpha}_{k} > 0$. Let $\bar{\epsilon}_k\in[0,s_k)$; since $\lim_{j \to \infty} \beta_k \underline{\rho}^{j} = 0$, there exists an $j_k \in \mathbb{N}_0$ such that for all $j \geq j_k$, $\beta_k \underline{\rho}^{j} \leq \overline{\alpha}_{k}$.
Referring again to \cref{BacktrackError2}, we know that $\lim_{\bar{\epsilon}_k \to 0} \underline{\alpha}_{k} = 0$. Therefore, there exists $\epsilon \in [0,s_k)$ such that for all $\bar{\epsilon}_k\leq \epsilon$, $\underline{\alpha}_{k} \leq \beta_k \underline{\rho}^{j_k}$. Additionally, $\overline{\alpha}_{k}$ as a function of $\bar{\epsilon}_k$ is decreasing on the domain $[0,s_k)$. Hence, $\beta_k \underline{\rho}^{j_k} \leq \overline{\alpha}_{k}$ holds for $\bar{\epsilon}_k\leq \epsilon$. 

As $i_k>0$ be the smallest integer for which $\psi(\alpha_k) \leq 0$ holds, then $\alpha_k = \beta_k \underline{\rho}^{i_k} \leq \overline{\alpha}_k$. If $\alpha_k \leq \underline{\rho} \overline{\alpha}_k$, then $\beta_k \underline{\rho}^{i_k-1} \leq \overline{\alpha}_k$, which contradicts with the assumption that $i_k$ is the smallest integer for which $\psi(\alpha_k) \leq 0$ holds. Hence, $\beta_k \underline{\rho}^{i_k} \leq \underline{\rho}\overline{\alpha}_{k}$
\end{proof}

Based on \cref{step_hit_interval}, we can guarantee the existence of a suitable step size during the backtracking line search process if $\bar{\epsilon}_k$ is small enough. 
\subsection{Convergence Analysis}\label{sec:convergence}
Prior to discussing the convergence theorem, we shall introduce the following auxiliary lemma.
\begin{lemma}\label{aux}
    Let $\epsilon > 0$ and $0<\underline{c}<1<\overline{c}$. Further, let $\{a_k\}_{k=0}^\infty \subseteq [0,\infty)$  be the sequence 
    $$a_{k+1} = \begin{cases}
        \overline{c} a_{k} \quad &\text{\rvv{if}} \quad  a_{k} <\epsilon, \\
        \underline{c}^{i_k} a_{k} \quad &\text{\rvv{if}} \quad  a_{k} \geq \epsilon,
    \end{cases}$$
    and $i_k$ be the smallest non-negative integer such that  $\underline{c}^{i_k} a_{k}< \epsilon$. Then, for all $k \geq 0$, \rvv{it holds that} $a_k \geq \min\{\underline{c} \epsilon, a_0\}$.
\end{lemma}
\begin{proof}
    Assume, by induction,  that for all $k = 0, \dots, t$, $a_t \geq \min\{\underline{c} \epsilon, a_0\}$.
    Now, consider $a_{t+1}$. If $a_{t} < \epsilon$, by definition, $a_{t+1} = \overline{c} a_t\geq \overline{c} \min\{\underline{c} \epsilon, a_0\} > \min\{\underline{c} \epsilon, a_0\}$. Otherwise, $a_{t} \geq \epsilon$. If $\underline{c} \epsilon \geq \underline{c}^{i_{t}} a_t$, then $\epsilon \geq \underline{c}^{i_{t}-1} a_t$, which contradicts the definition of $i_{t}$. Hence, $\underline{c} \epsilon <\underline{c}^{i_{t}} a_t$, which implies $a_{t+1} = \underline{c}^{i_{t}} a_t > \underline{c} \epsilon \geq \min\{\underline{c} \epsilon, a_0\}$. 
\end{proof}
\begin{lemma}\label{uniform_bound}
    Let Assumptions \ref{ass1} and \ref{ass2} hold, $\beta_0>0$\rv{, $\lambda < \eta$}, and $0<\underline{\rho}<1<\overline{\rho}$. Let $i_k$ be the smallest non-negative integer such that $\alpha_k = \underline{\rho}^{i_k}\beta_k \in [\underline{\alpha}_k,\overline{\alpha}_k]$ and $\alpha_k> \underline{\rho} \overline{\alpha}_k$, if $i_k>0$. 
    Further, let $\beta_{k+1} = \overline{\rho}\alpha_k$, $\|z_k\|\neq 0$, and ${\|e_k\|}\leq (1-\eta) {\|z_k\|}$. Denoting \begin{equation}\label{bound_alpha_tau}
    \tau \coloneqq \min \left \{ \underline{\rho} \left(\frac{\eta - \lambda}{L_{\nabla f}} \right) ,  \rv{\min \Big \{ }\underline{\rho} \left (\frac{2(\eta - \lambda)}{L_{\nabla f}}\right ), \beta_0 \rv{ \Big \}}  \right \},\end{equation} we have $\alpha_k \geq \tau$.

\end{lemma}
\begin{proof}
    If $i_k>0$, since $\alpha_k = \underline{\rho}^{i_k}\beta_k$, using \cref{step_hit_interval} and \eqref{alpha_min_max} we have
    \begin{equation}\label{lower_temp}
        \alpha_k > \underline{\rho} \overline{\alpha}_k \geq \underline{\rho} \left(\frac{\eta - \lambda}{L_{\nabla f}} \right).
    \end{equation}
    Moreover, from \cref{BacktrackError2}, \eqref{alpha_min_max} yields 
    \begin{equation}\label{lower_temp2}
        \alpha_k \leq \overline{\alpha}_{k}\leq \frac{2(\eta - \lambda)}{L_{\nabla f}}.
    \end{equation}

If $i_k=0$, then $\alpha_k = \beta_k$. Considering the update $\beta_{k+1} = \overline{\rho} \alpha_k$, utilizing \cref{aux} by setting $a_k = \beta_k$, $\underline{c} = \underline{\rho}$, $\overline{c} = \overline{\rho}$, and $\epsilon = \frac{2(\eta - \lambda)}{L_{\nabla f}}$, based on \eqref{lower_temp2}, we derive $\beta_k\geq \min \{\underline{\rho} \left (\frac{2(\eta - \lambda)}{L_{\nabla f}}\right ), \beta_0\}$. 
Taking $\tau$ as defined in \eqref{bound_alpha_tau} yields the desirable result.
\end{proof}
\begin{theorem}\label{convergencePK}
    Let the assumptions of \cref{uniform_bound} hold and $0<\lambda<\eta<1$. We have
    $$\lim_{k\rightarrow \infty}\|z_k\| = 0.$$
\end{theorem}

\begin{proof}
     \cref{inexact_bt_lemma} and the backtracking line search rule $\psi(\alpha_k) \leq 0$ yield
    \begin{equation}
            f(\theta_{k+1}) - f(\theta_k) \leq -\lambda \alpha_k \|z_k\|^2 .\label{upper_temp}
    \end{equation}
Utilizing \cref{uniform_bound}, since $0<\lambda<\eta<1$, we have $\tau>0$, and
recursively from the inequality \eqref{upper_temp} we can conclude 
\begin{equation}\label{finiteSum}
    \begin{split}
        \sum_{k=1}^K \lambda \tau \|z_k\|^2 \leq f(\theta_1) - f(\theta_{K+1}).
    \end{split}
\end{equation}
As $f$ is bounded below and $\lambda \tau>0$ we conclude $\lim_{k\rightarrow \infty}\|z_k\| = 0$.
\end{proof}

\begin{corollary}\label{grad_convergence_cor}
Under the assumptions of \cref{convergencePK}\rvv{,} we have
$$\lim_{k\rightarrow \infty} \| \nabla f(\theta_k) \| = 0.$$
\end{corollary}
\begin{proof}
    Since ${\|e_k\|}\leq (1-\eta)\|z_k\|$, utilizing \cref{convergencePK} we have $\lim_{k\rightarrow \infty} \|z_k\|=0$, which implies $\lim_{k\rightarrow \infty} \|e_k\| = 0$. Hence 
    \begin{equation*} 
    \lim_{k\rightarrow \infty} \|\nabla f(\theta_k)\| \leq \lim_{k\rightarrow \infty} (\|e_k\| +\|z_k\|) = 0,\end{equation*}
    as required.
\end{proof}
Note that, at each iteration $k\in \mathbb{N}_0$, based on \cref{prop1} and \cref{convergencePK}, the approximate hypergradient should satisfy the inequality ${\|e_k\|}\leq(1 -\eta){\|z_k\|}$. Otherwise, we can decrease the error in the hypergradient $\|e_k\|$ by means of multiplying $\epsilon_k$ and $\delta_k$ by a factor $0<\underline{\rho}<1$ as we do in the steps \ref{decrease_delta} and \ref{decrease_eps}, then restarting the calculation of $z_k$ using \cref{Inexact_grad} until ${\|e_k\|}\leq(1 -\eta){\|z_k\|}$ holds. Moreover, the steps of \cref{HOAH_back_new} are designed in a way to satisfy both the required accuracy for the line search $\psi(\alpha_k) \leq 0$. MAID also allows $\epsilon_k$ and $\delta_k$ to increase in steps \ref{increase_epsilon_k} and \ref{increase_delta_k} if they do not decrease in the subsequent iteration and remain suitable for the line search condition, which matches with $\bar{\epsilon}_k$ as mentioned in \cref{BacktrackError} and \cref{BacktrackError2}. This process in \cref{HOAH_back_new} and \cref{Inexact_grad} provides us with an adaptive way of choosing the accuracy.

\begin{proposition} \label{prop_gradient_terminates}
    If $\|\nabla f(\theta_k)\|>0$, the loop at line \ref{omega_while} in \cref{Inexact_grad} terminates in finite time by finding a descent direction $-z_k$.
\end{proposition}
\begin{proof}
    From the bound \eqref{AP_bound_hg} for the upper bound $\omega$ of the error $\|e_k\| = \|z_k - \nabla f(\theta_k)\|$, we have
    $\lim_{\epsilon_k,\delta_k \to 0} \omega = 0$. On the other hand, $\lim_{\epsilon_k,\delta_k \to 0} \|z_k\| = \|\nabla f(\theta_k)\| >0$. Since $\epsilon_k\to 0$ and $\delta_k \to 0$ according to lines \ref{decrease_eps} and \ref{decrease_delta} of \cref{Inexact_grad}, respectively, and utilizing \cref{abstract_sequences}, we conclude that the inequality $\omega \leq (1-\eta)\|z_k\|$ holds for $\epsilon_k$ and $\delta_k$ sufficiently small.
\end{proof}

\begin{proposition}\label{inner_loop_theorem}
The loop at line \ref{bt_loop} in \cref{HOAH_back_new} terminates in finite time by successfully performing the update $\theta_{k+1}= \theta_k - \underline{\rho}^{i_k}\beta_k z_k$ for some $i_k\in \mathbb{N}_0$.
\end{proposition}

\begin{proof}
If the inner loop at the line \ref{inner_loop} fails to find a step size within the given maximum number of iterations $j$, two scenarios may occur. Firstly, if $\epsilon_k$ and $\delta_k$ are not sufficiently small, they are decreased in steps \ref{BT_decrease} and \ref{BT_decrease_delta}, followed by updating the descent direction $z_k$ and a possible further decrease in accuracies $\epsilon_k$ and $\delta_k$ in step \ref{updated_descend_direction} as a part of the next iteration of the outer loop at line \ref{bt_loop}. This sequence of accuracy values will eventually become small enough to satisfy the assumptions in  \cref{corollary} and \cref{step_hit_interval}. Consequently, the required result holds. 

Secondly, if $\underline{\rho}^{j-1} \beta_k > \overline{\alpha}_k$ and the backtracking process is unsuccessful, then the maximum number of backtracking iterations $j$ goes to infinity during the iterations of the outer loop \ref{bt_loop}, and $\underline{\rho}^{j} \beta_k \to 0$. Since $\overline{\alpha}_k>0$, this implies that there exists $j_k \in \mathbb{N}_0$ such that for all $ i_k\geq j_k$, $\underline{\rho}^{i_k} \beta_k \leq \overline{\alpha}_{k}$. Therefore, the loop \cref{bt_loop} terminates in a finite process.
\end{proof}
\begin{theorem}\label{alg_covergence}
    Under Assumptions \ref{ass1} and \ref{ass2}, the iterates $\theta_k$ of \cref{HOAH_back_new} satisfy
    $$\lim_{k \rightarrow \infty}\|\nabla f(\theta_k)\| = 0.$$
\end{theorem}
\begin{proof}
    From \cref{inner_loop_theorem} we know that the loop at line \ref{bt_loop} terminates after a finite number of iterations and finds a suitable step size $\alpha_k$ satisfying the line search condition $\psi(\alpha_k) \leq 0$. Moreover, the update of the initial step size in line \ref{increase_beta_k}  fulfills $\beta_{k+1} = \overline{\rho} \alpha_k$. Furthermore, $z_k$ computed in \cref{Inexact_grad} satisfies  ${\|e_k\|}\leq(1-\eta){\|z_k\|}$. Hence, all the assumptions of \cref{convergencePK} are satisfied by \cref{HOAH_back_new} and convergence follows from  \cref{grad_convergence_cor}.
\end{proof}
\begin{remark}
\rvv{
The results presented in this section assume that the upper-level loss  $g$  does not depend on the parameters  $\theta$ . However, if  $g$  depends on  $\theta$  and is differentiable with respect to  $\theta$ , the hypergradient takes the following form:
\begin{equation}
    \nabla f(\theta) = \nabla (g \circ \hat{x})(\theta)= \nabla_\theta g(\hat{x}(\theta), \theta) +  \partial \hat{x}(\theta)^T \nabla_x g(\hat{x}(\theta), \theta).
\end{equation}
This modification affects the calculation of the inexact hypergradient, while the line search condition remains unchanged. This is because the bounds in \cref{computable_bounds} were derived by expanding  $g$  with respect to its first argument (i.e.,  $x$). Moreover, assuming that $\nabla_{\theta} g(\hat{x}(\theta), \theta)$ is  $L_{\nabla_{\theta} g}$-Lipschitz in $x$, the right-hand side of bound \eqref{AP_bound_hg} will include an additional term,  $L_{\nabla_{\theta} g} \epsilon_k$.
}
\end{remark}
\section{Numerical Experiments}\label{sec:numerical_experiments}

In this section, we present the numerical results of our proposed algorithm MAID (\cref{HOAH_back_new}) and compare it with HOAG \cite{Pedregosa} and the dynamic derivative-free algorithm DFO-LS \cite{DFO}. For HOAG, we take the adaptive step size strategy used in the numerical experiments of \cite{Pedregosa} with accuracies either following a geometric sequence $\epsilon_k = 0.9^k \epsilon_0$, a quadratic sequence $\epsilon_k = {k^{-2}} \epsilon_0$, or a cubic sequence $\epsilon_k = {k^{-3}} \epsilon_0$ for $k=1,2, \dots$. The implementation of HOAG\footnote{\url{https://github.com/fabianp/hoag}} and the Fast Iterative Shrinkage-Thresholding Algorithm (FISTA) variant of DFO-LS\footnote{\url{https://github.com/lindonroberts/inexact_dfo_bilevel_learning}} is based on the available code from the corresponding papers \cite{Pedregosa} and \cite{DFO}, respectively. \textcolor{red}{We will make the codes of our implementations available on GitHub upon the acceptance of the paper}. Moreover, we use warm-starting in the lower-level iterations, which means we save and use the lower-level solution as the initial point for the next call of the lower-level solver. This strategy is also utilized in DFO-LS and HOAG. For the constant $L_{H^{-1}}$ in the bound \eqref{AP_bound_hg} used in \cref{Inexact_grad}, we approximate it as $$\frac{\|\nabla_x^2 h(\tilde{x}(\theta),\theta)^{-1} \nabla g(\tilde{x}(\theta)) - \nabla_x^2 h(\tilde{x}(\theta),\theta)^{-1} u\|}{\|\nabla g(\tilde{x}(\theta)) - u\|},$$ where $u \sim \mathcal{N}(0,1)$ is a random perturbation of $\nabla g(\tilde{x}(\theta))$. We then store the maximum value of this approximation across the upper-level iterations. In this approach, since we have already calculated $\nabla_x^2 h(\tilde{x}(\theta),\theta)^{-1} \nabla g(\tilde{x}(\theta))$ to compute the hypergradient, we possess an efficient approximation that improves by taking its maximum across upper-level iterations. We use a similar process for estimating $L_J$. \rv{Moreover, we estimate the $\|\nabla_{x\theta}^2 h(x(\theta),\theta)\|$ on the right-hand side of \eqref{AP_bound_hg} using one iteration of the power method (\rvv{to compute largest  eigenvalue of $\nabla_{x\theta}^2 h(x(\theta),\theta)^T\nabla_{x\theta}^2 h(x(\theta),\theta)$}).} We include these calculations in our lower-level computational cost. Note that in all experiments where the upper-level loss is convex, the bound \eqref{psi_convex} from \cref{inexact_bt_convex} will be used in step \ref{descent_condition_check} of \cref{HOAH_back_new}.
\subsection{Validating the performance with known optimal parameters}\label{subsec:quad}
Similar to \cite{boundMatthias}, we first consider a simple linear least-squares problem taken from \cite{Li_Gu_Huang_2022}: 
\begin{subequations}\label{Quadratic}
\begin{align}
    \min_{\theta \in \mathbb{R}^{10}} f(\theta) &\coloneqq \|A_1\hat{x}(\theta) - b_1\|^2  \\
    s.t. \quad \hat{x}(\theta) &\coloneqq \arg\min_{x \in \mathbb{R}^{10}} \|A_2x + A_3\theta - b_2\|^2, 
\end{align}
\end{subequations}
where $A_i \in \mathbb{R}^{1000 \times 10}$ have random entries uniformly distributed between 0 and 1, and $b_i \in \mathbb{R}^{1000}$ are given by $b_1 = A_1\hat{x}_1 + 0.01y_1$ and $b_2 = A_2\hat{x}_2 + A_3\Bar{\theta} + 0.01y_2$, where $\hat{x}_1, \hat{x}_2,$ and $\Bar{\theta} \in \mathbb{R}^{10}$ have i.i.d.\ entries uniformly distributed between 0 and 1, and $y_1, y_2 \in \mathbb{R}^{1000}$ are independent standard Gaussian vectors with entries from $\mathcal{N}(0,1)$. For our experiments, we pick $\theta_0$ as the vector of all ones.
In this problem, we have the capability to analytically compute $\hat{x}(\theta)$, $\nabla f(\theta)$, and all of the Lipschitz constants ($L_{\nabla f}$, $L_{\nabla g}$, $L_{H^{-1}}$, and $L_J$), along with determining the optimal $\theta^*$. Therefore, this problem serves as a benchmark to evaluate the performance of our inexact algorithm.

For the experimental evaluation, we have chosen several configurations for the parameters in our algorithm. Specifically, these encompass $\epsilon_0 = \delta_0 = 10^{-1}$, $\epsilon_0 = \delta_0 = 10^{-3}$, $\epsilon_0 = \delta_0 = 10^{-5}$. Moreover, in MAID, we set the hyperparameters $\overline{\rho} = \frac{10}{9}$, $\underline{\rho} = 0.5$, $\overline{\nu} = 1.25$, and $\underline{\nu} = 0.5$. Across all instances of fixed accuracy, the line search mechanism of MAID ($\psi(\alpha_k) \leq 0$) is employed to determine the step size. Subsequently, we employ FISTA adapted to strongly convex problems to solve the lower-level problem \cite{chambolle_pock_2016}.

Since the main computational cost involved in solving bilevel optimization problems lies in solving the lower-level problems, we have imposed a cap on the total number of lower-level iterations. 
Furthermore, recognizing the computational analogy between the cost of each CG iteration in the IFT approach, dominated by a single Hessian-vector product, and each iteration of the lower-level solver (e.g., FISTA), dominated by a single gradient evaluation, we treat CG iterations with comparable significance to that of an iteration of the lower-level solver. Consequently, the amalgamation of lower-level solver iterations and CG iterations is regarded as constituting the cumulative computational cost of the lower-level phase. For solving the problem \eqref{Quadratic}, we set a termination budget of $1.5 \times 10^{5}$ lower-level iterations.

\begin{figure}[tbhp]
\centering \subfloat[Accuracy $\epsilon$ of the lower-level solver in each upper-level iteration]{\label{fig:QD_eps}\includegraphics[width = 0.40\textwidth]{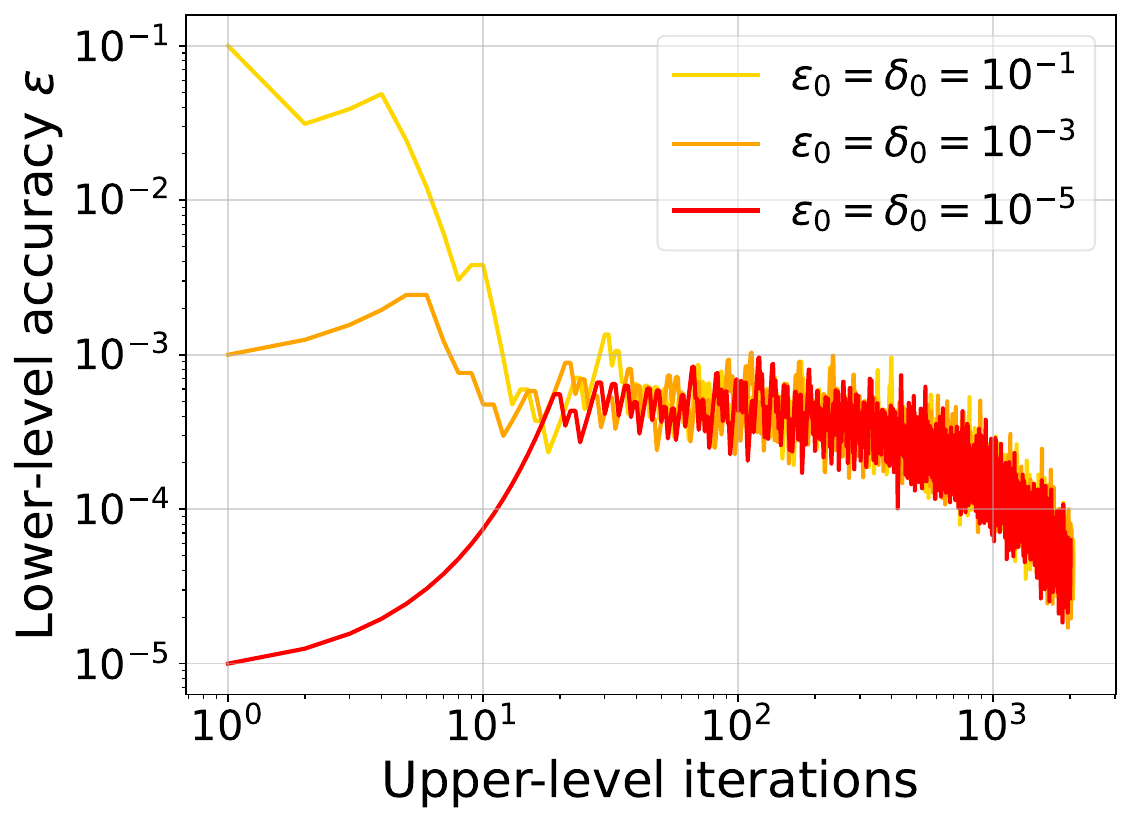}} 
\hspace{5pt}\subfloat[Upper-level loss vs total calls of linear solver and lower-level solver]{\label{fig:QD_LL_CG}\includegraphics[width = 0.40\textwidth]{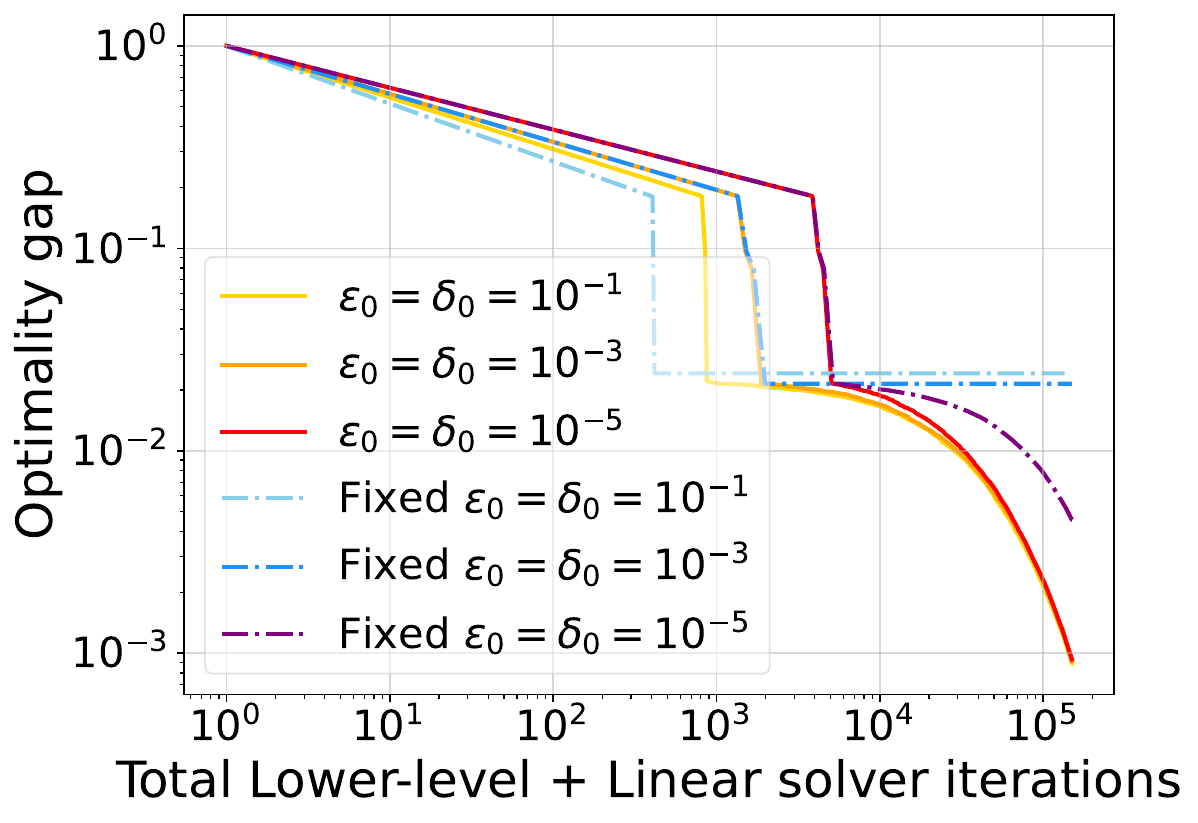}}
\caption{The value of $\epsilon$ per upper-level iteration, alongside the performance and computational cost of MAID, is illustrated across various settings for solving the problem \eqref{Quadratic}. All MAID configurations achieve lower loss values compared to the high fixed accuracy, and they do so at a lower computational cost. Furthermore, they consistently converge to the same range of the required accuracy $\epsilon = \delta$, regardless of the initial choices of $\epsilon_0$.}
\label{fig:QD_eps_delta_loss}
\end{figure}

\Cref{fig:QD_eps} illustrates the determination of the necessary values for $\epsilon$ in solving the problem~\eqref{Quadratic} through the MAID algorithm. As shown in \cref{fig:QD_eps}, irrespective of the initial $\epsilon_0 = \delta_0$ as the accuracy of the lower-level solver and the linear solver, respectively, MAID adaptively reduces $\epsilon$ and $\delta$, ultimately opting for accuracies around $2 \times 10^{-5}$. 

\cref{fig:QD_LL_CG} illustrates the computational advantages of adaptivity in MAID compared to fixed low and high accuracies, with the main computational costs stemming from lower-level and linear solver iterations. Setting the low fixed accuracy to $\epsilon = 10^{-1}$ and $\epsilon = 10^{-3}$ results in loss stalling after a few upper-level iterations due to insufficient accuracy for line search progression. Conversely, using the higher fixed accuracy of $\epsilon = 10^{-5}$ yields a higher loss than dynamic configurations, due to exhausting the total lower-level iteration budget. This underscores the importance of the adaptive approach in MAID, as the suitable fixed accuracy is not known a priori; it may be too small, leading to failure (e.g., $\epsilon = 10^{-1}$), or require significantly higher computational costs for progress (e.g., $\epsilon = 10^{-5}$). Overall, dynamic configurations successfully solve problem \eqref{Quadratic} with lower computational costs compared to a fixed accuracy of $10^{-5}$, regardless of the starting accuracy.
\begin{figure}[tbhp]
\centering {\includegraphics[width = 0.4\textwidth]{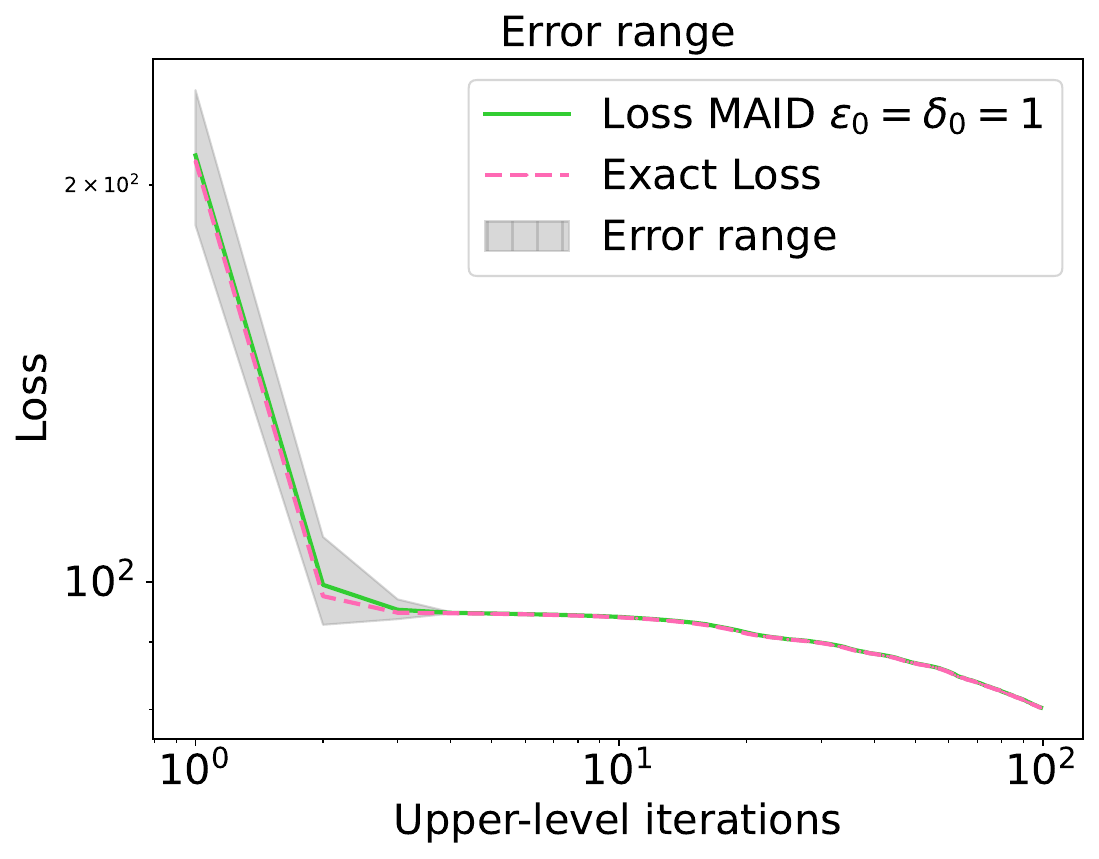}} 
\caption{The range of error in upper-level function evaluation given error bounds \eqref{upperApprox} and \eqref{convexApprox}, inexact loss given the approximated lower-level solution, and exact loss evaluation $f(\theta) = g(\hat{x}(\theta))$.}
\label{fig:QD_error}
\end{figure}
While in \cref{fig:QD_LL_CG} we plot the optimality gap, it is also possible to compute the exact solution of the lower-level problem in the specific problem \eqref{Quadratic}. This enables the calculation of the precise loss at each upper-level iteration. This ability allows us to verify both the sufficient decrease condition and the progress of MAID. To facilitate this verification, we present the exact loss alongside the loss obtained using MAID with $\epsilon_0 = \delta_0 = 1$ in \cref{fig:QD_error}. Additionally, this figure visually represents the bounds for inexact loss evaluation, as defined by \eqref{upperApprox} and \eqref{convexApprox}, wherein the area between these bounds encapsulates both exact and inexact loss values. Furthermore, as observed in \cref{fig:QD_error}, enhancing the accuracy of the lower-level solver leads to a narrower interval between the upper and lower bounds, thereby bringing the exact and inexact loss values closer. 
\subsection{Multinomial logistic loss}\label{subsec:regression}
In order to compare MAID with HOAG \cite{Pedregosa} as a well-known gradient-based algorithm with inexact hypergradients, we replicate the multinomial logistic regression bilevel problem used in \cite{Pedregosa, maclaurin2015gradientbased} on the MNIST\footnote{\url{http://yann.lecun.com/exdb/mnist/}} dataset. Similar to \cite{Pedregosa}, we take the entire $M = 60000$ training samples of MNIST and rescale them to $12\times 12$ pixels. Moreover, we utilize all $N = 10000$ samples of the validation set for the upper-level problem. The bilevel problem corresponding to this model has the following form:
\begin{subequations}\label{multilogistic}
\begin{align}
    \min_{\theta \in \mathbb{R}^{d}} f(\theta) &\coloneqq \sum_{i=1}^N \Psi(\Bar{b}_i, \Bar{a}_i^T\hat{x}(\theta))  \\
    s.t. \quad \hat{x}(\theta) &\coloneqq \arg\min_{x \in \mathbb{R}^{p\times q}} \sum_{j=1}^M \Psi(b_j, a_j^Tx) + \frac{1}{2}\sum_{k=1}^{p}\sum_{l=1}^{q} e^{\theta[{(k-1)q+l}]} x_{k,l}^2, 
\end{align}
\end{subequations}
where $p = 12 \times 12 = 144$ represents the number of features, $q = 10$ represents the number of classes, and $d = p \times q = 1440$ denotes the number of parameters. $\Psi$ represents the multinomial logistic loss \cite[Section 10.3]{pml1Book}, while $a_j \in \mathbb{R}^p$ and $b_j \in \mathbb{R}^q$ correspond to the training images and labels for each $1 \leq j \leq M$. Similarly, $\Bar{a}_i \in \mathbb{R}^p$ and $\Bar{b}_i \in \mathbb{R}^q$ denote the validation images and labels in the upper-level problem for each $1 \leq i \leq N$. To solve this problem, we employ MAID, alongside HOAG using its heuristic adaptive step size with default parameters. We set the initial accuracy for both algorithms to $\epsilon_0 = \delta_0 = 10^{-1}$, as specified in \cite{Pedregosa}. Moreover, to ensure a fair comparison, similar to HOAG, we set the hyperparameters $\overline{\rho} = \frac{10}{9}$, $\underline{\rho} = 0.5$, $\overline{\nu} = 1.05$, and $\underline{\nu} = 0.5$. In HOAG, the initial step size $\beta_0 = \sqrt{d}/\|z_0\|$ was determined using the same strategy as outlined in \cite{Pedregosa}. We set the initial upper-level step size $\alpha_0$ in MAID equal to the initial step size in HOAG. Additionally, we consider a fixed total lower-level budget of $6 \times 10^{5}$ iterations and a maximum of $300$ upper-level iterations.

\begin{figure}[tbhp]
    \centering
    \subfloat[\footnotesize Upper-level step size of MAID vs HOAG]{\label{fig:logistic_step}\includegraphics[width=0.326\textwidth]{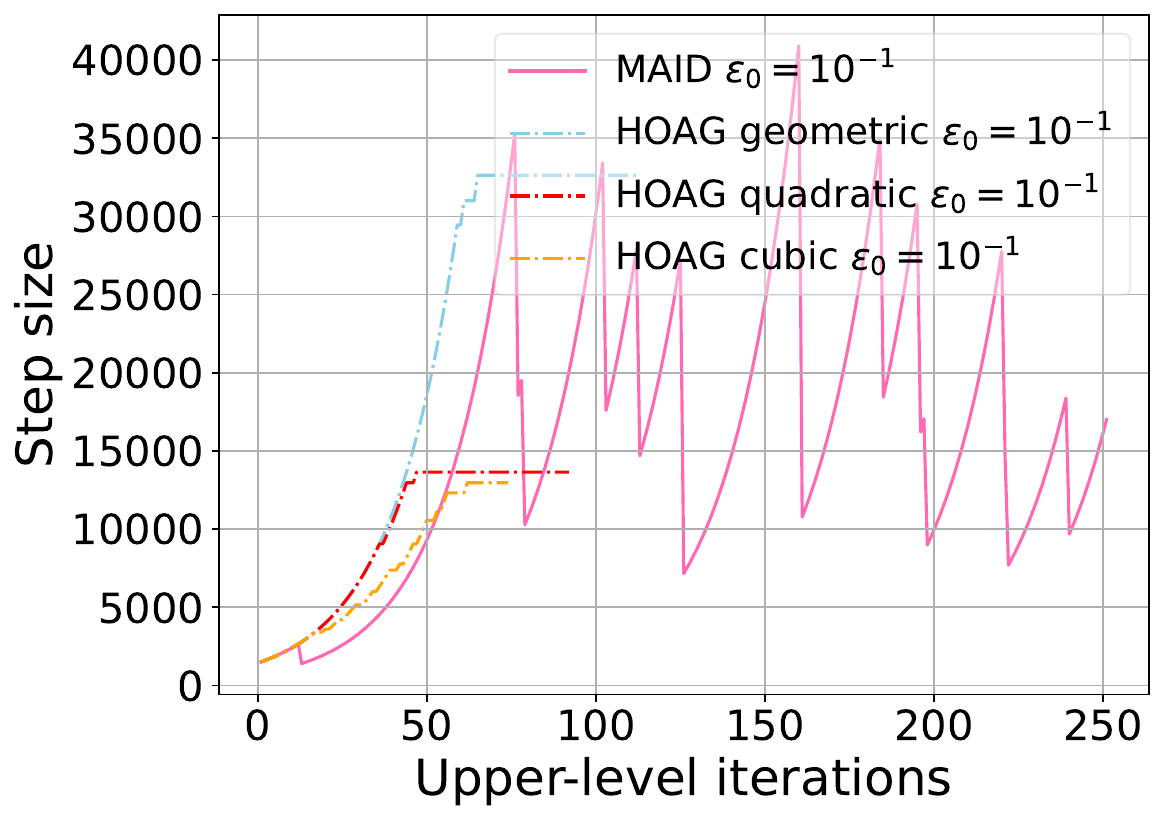}} \hfill \vspace{5pt}  
    \subfloat[\footnotesize Accuracy of the lower-level solver in each upper-level iteration]
{\label{fig:logistic_eps}\includegraphics[width=0.326\textwidth]{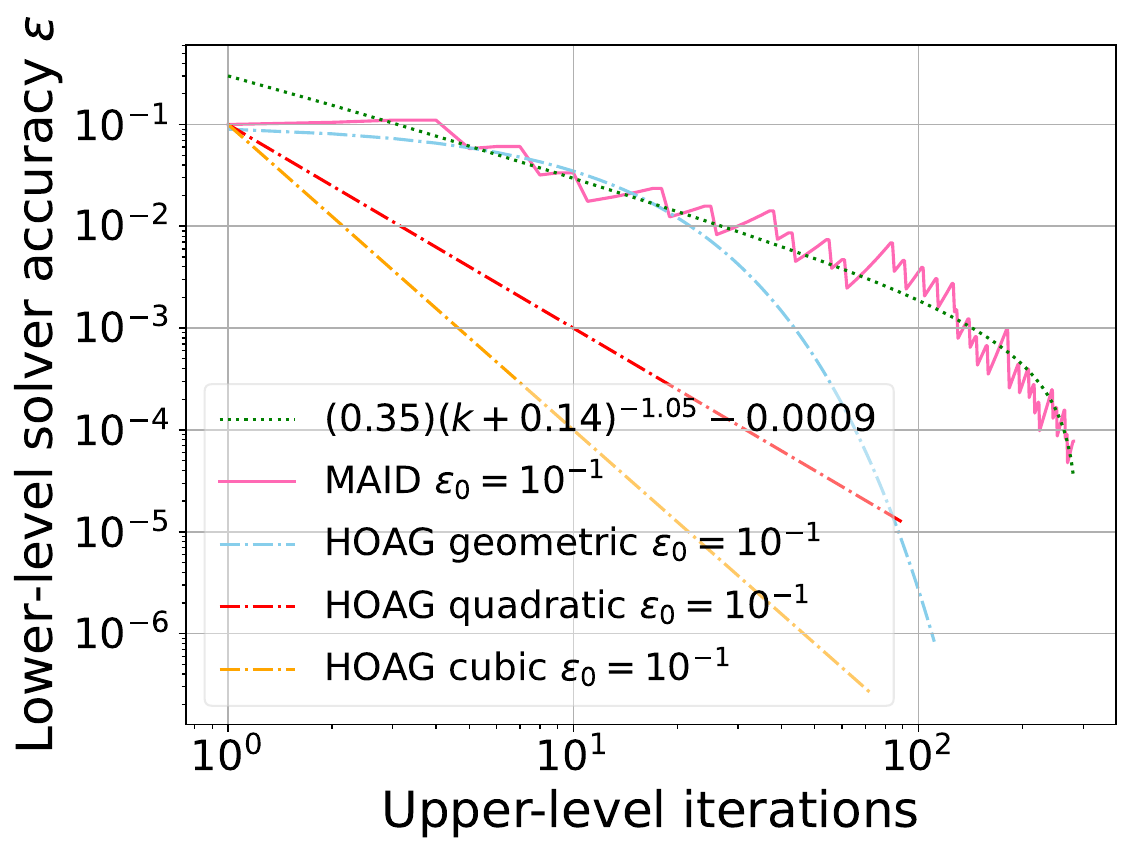}}
     \hfill 
     \vspace{5pt}  
    \subfloat[\footnotesize Upper-level loss vs total lower-level computational cost]{\label{fig:logistic_loss_LL}\includegraphics[width=0.326\textwidth]{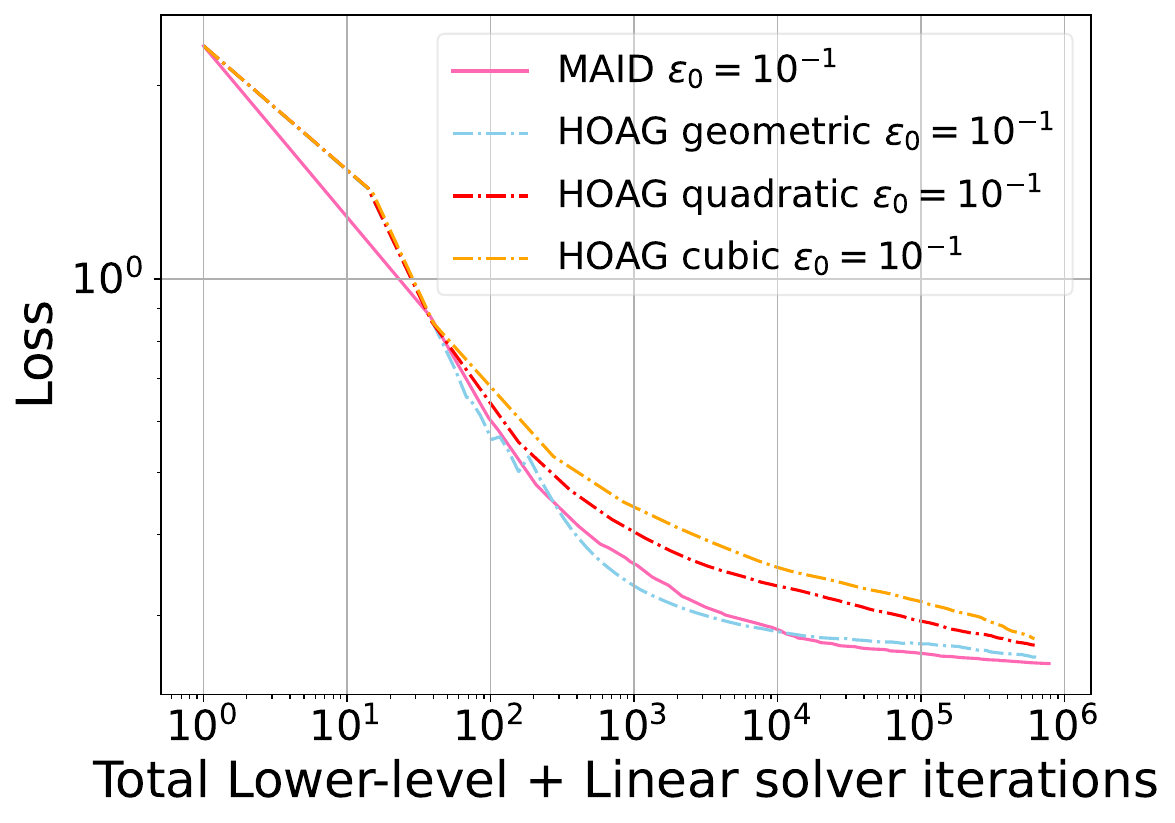}}
    \caption{The value of $\epsilon$, and upper-level loss per total lower-level computational cost in the algorithms MAID and HOAG for solving the problem \eqref{multilogistic}. It illustrates that MAID adaptively adjusts the upper-level step size and the required accuracy for the solving the lower-level problem, while HOAG with the adaptive step size does not adjust the step size after a few iterations and exhibits a decreasing sequence for the accuracy $\epsilon$. In all HOAG and MAID settings, $\epsilon_0 = \delta_0$ is set.}
    \label{fig:logistic_step_eps_loss}
\end{figure}
Considering the fixed computational budget, \cref{fig:logistic_step_eps_loss} offers a comparison between HOAG and MAID in terms of upper-level step size, the accuracy of the lower-level solver, and loss. Throughout this experiment, all configurations of HOAG, employing the adaptive step size strategy from \cite{Pedregosa} and depending on their accuracy regime, demonstrated the ability to effectively adjust the step size and accuracy, leading to loss reduction in \eqref{multilogistic} within a few upper-level iterations and their corresponding lower-level computational cost. However, HOAG's adaptive step size eventually stops changing and becomes fixed at different levels, depending on the accuracy regime, while diminishing accuracy for the lower-level solver to negligible values. As a result, the lack of adaptivity in accuracy and reaching high accuracies lead to slower loss reduction using the quadratic and cubic sequences of accuracies.

\begin{figure}[tbhp]
    \centering
\includegraphics[width=0.326\textwidth]{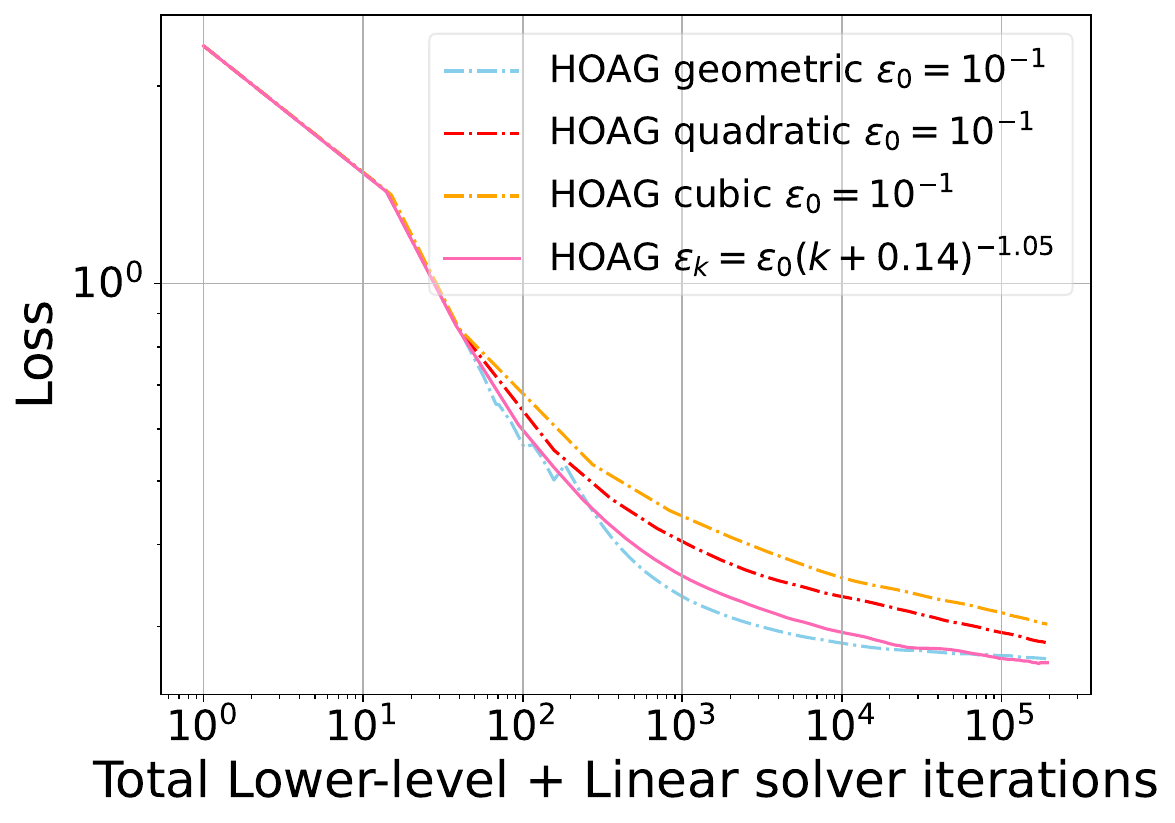}
    \caption{\rv{Comparison of HOAG with geometric, quadratic, and cubic accuracy sequences against a custom sequence obtained by fitting an accuracy regime to the sequence chosen by MAID.}}
    \label{fig:logistic_loss_costume}
\end{figure}
In contrast, while using the same initial step size and accuracy as HOAG, MAID consistently identified an appropriate range for accuracy and upper-level step size. This led to driving the loss to lower values than those achieved by HOAG within the same budget. Moreover, upon fitting a polynomial regime to the accuracy sequence taken by MAID, the sequence $\epsilon_k = {0.35}{(k+0.14)^{-1.05}}\epsilon_0 - 0.0009$ for $k=1,2,\dots$ could potentially be a suitable option for HOAG when compared to the three other named sequences. \rv{\cref{fig:logistic_loss_costume} supports this guess, showing that HOAG with the specified sequence performs comparably to the best of the named sequences, which in this case is geometric.} Importantly, since this sequence is summable, the theoretical convergence of HOAG also holds. Nevertheless, knowing such a sequence a priori is not straightforward, and MAID finds it adaptively.

In summation, MAID demonstrates superior adaptability and performance compared to HOAG, achieving lower loss values within the same lower-level computational budget. MAID achieved this by adaptively adjusting accuracy and step size, and showcased an enhanced efficiency. Additionally, the step size selection in MAID comes with a theoretical convergence guarantee, whereas the strategy in HOAG is heuristic.

\subsection{Total variation denoising}\label{subsec:TV}

In this part, we consider the smoothed version of the Rudin-Osher-Fatemi (ROF) image-denoising model \cite{RUDIN1992259} as implemented in \cite{DFO}. Here, a bilevel problem is formulated to learn the smoothing and regularization parameters of total variation (TV) regularization. In contrast to the first numerical experiment in this section, access to the optimal hyperparameters is not possible a priori.

Similar to \cite{DFO}, we apply this model to $2D$ images and the training set is $m = 25$ images of the Kodak dataset\footnote{\url{https://r0k.us/graphics/kodak/}} resized to $96 \times 96$ pixels, converted to grayscale, and Gaussian noise with $\mathcal{N}(0, \sigma^2)$ with $\sigma = 0.1$ is added. The bilevel problem we examine is as follows:
\begin{subequations}\label{bilevel_TV}
    \begin{align}
    &\min_\theta \frac{1}{m} \sum_{t=1}^{m} \frac{1}{2}\|\hat{x}_t(\theta) - x^*_t \|^2 \label{upper_TV}\\ & s.t.  \quad \hat{x}_t(\theta) = \arg \min_x {\frac{1}{2}\|x-y_t\|^2} + e^{\theta[1]} {\sum_{i,j} \sqrt{|\nabla x_i|^2+|\nabla x_j|^2 + (e^{\theta[2]})^2}}\label{lower_TV},
\end{align}
\end{subequations}

Here, $\nabla x_i$ and $\nabla x_j$ represent the $i$-th and $j$-th elements of the forward differences of $x$ in the horizontal and vertical directions, respectively. It is important to note that this problem satisfies the assumption \ref{ass1}. To initiate the process, we set $\theta_0 = (-5, -5)$ and utilize Algorithm \ref{HOAH_back_new} \rv{with hyperparameters $\overline{\rho} = \frac{10}{9}$, $\underline{\rho} = 0.5$, $\overline{\nu} = 1.25$, and $\underline{\nu} = 0.5$}, along with the FISTA variant of DFO-LS with dynamic accuracy \cite{DFO} to solve the problem \eqref{bilevel_TV}. The DFO-LS with dynamic accuracy is a derivative-free bilevel algorithm that employs FISTA for solving the lower-level problem with an accuracy determined dynamically from the trust region radius in the derivative-free optimization algorithm for solving least-squares \cite{cartis2017derivativefree} in the upper-level problem. This algorithm serves as an efficient approach for solving bilevel problems and in particular for the problem \eqref{bilevel_TV}. However, it should be noted that, as discussed in \cite{Crockett_2022}, in contrast with gradient-based algorithms, it may not scale well with the number of hyperparameters due to the nature of the derivative-free algorithm used for the upper-level problem. To maintain consistency with the first numerical example, we set a limit on the total number of lower-level iterations in this case.

\begin{figure}[tbhp]
\centering \subfloat[\footnotesize Accuracy of the lower-level solver in each upper-level iteration]{\label{fig:TV_eps}\includegraphics[width = 0.4\textwidth]{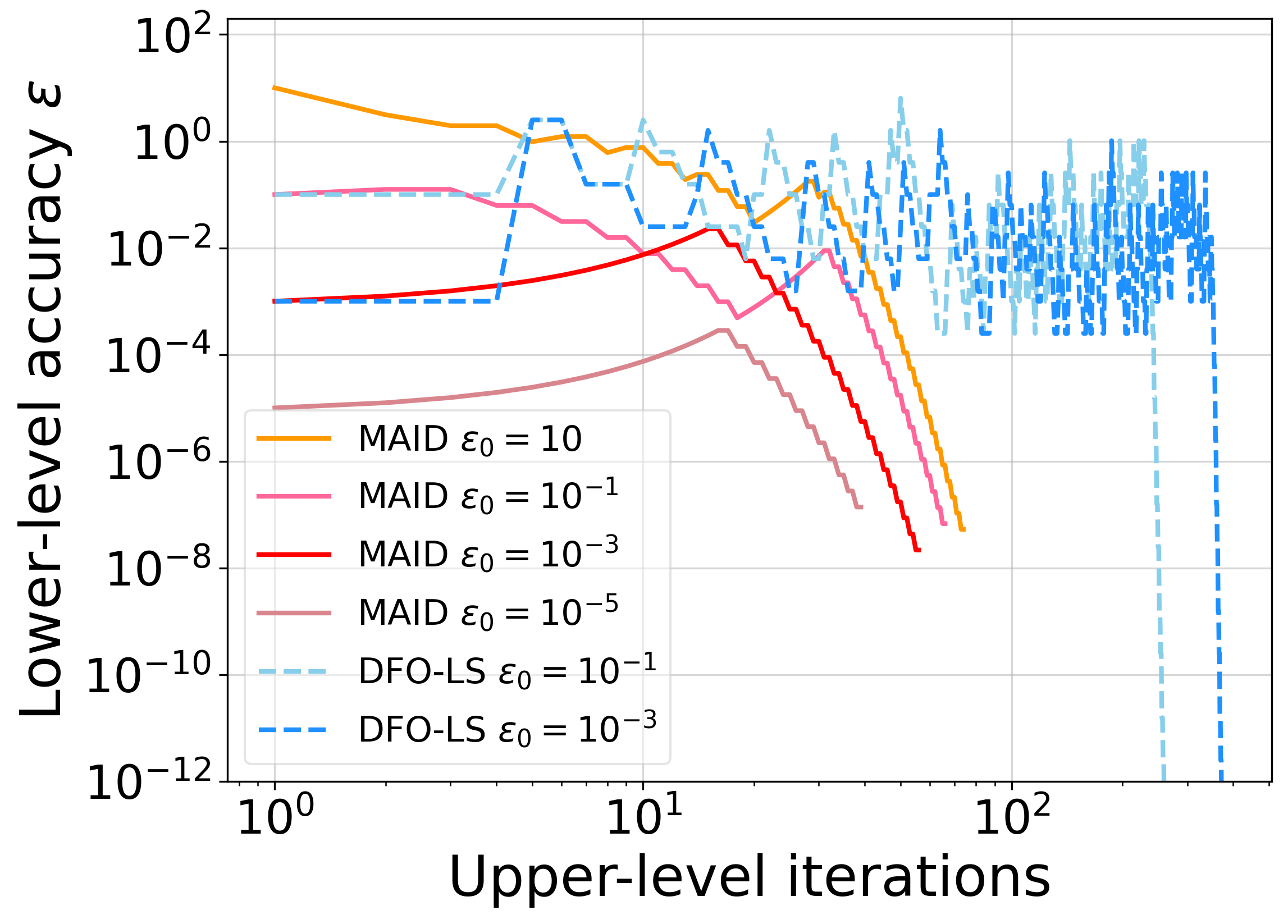}}\hspace{10pt}
\subfloat[\footnotesize Upper-level loss vs total lower-level computational cost]{\label{fig:TV_loss_ll_cg}\includegraphics[width = 0.4\textwidth]{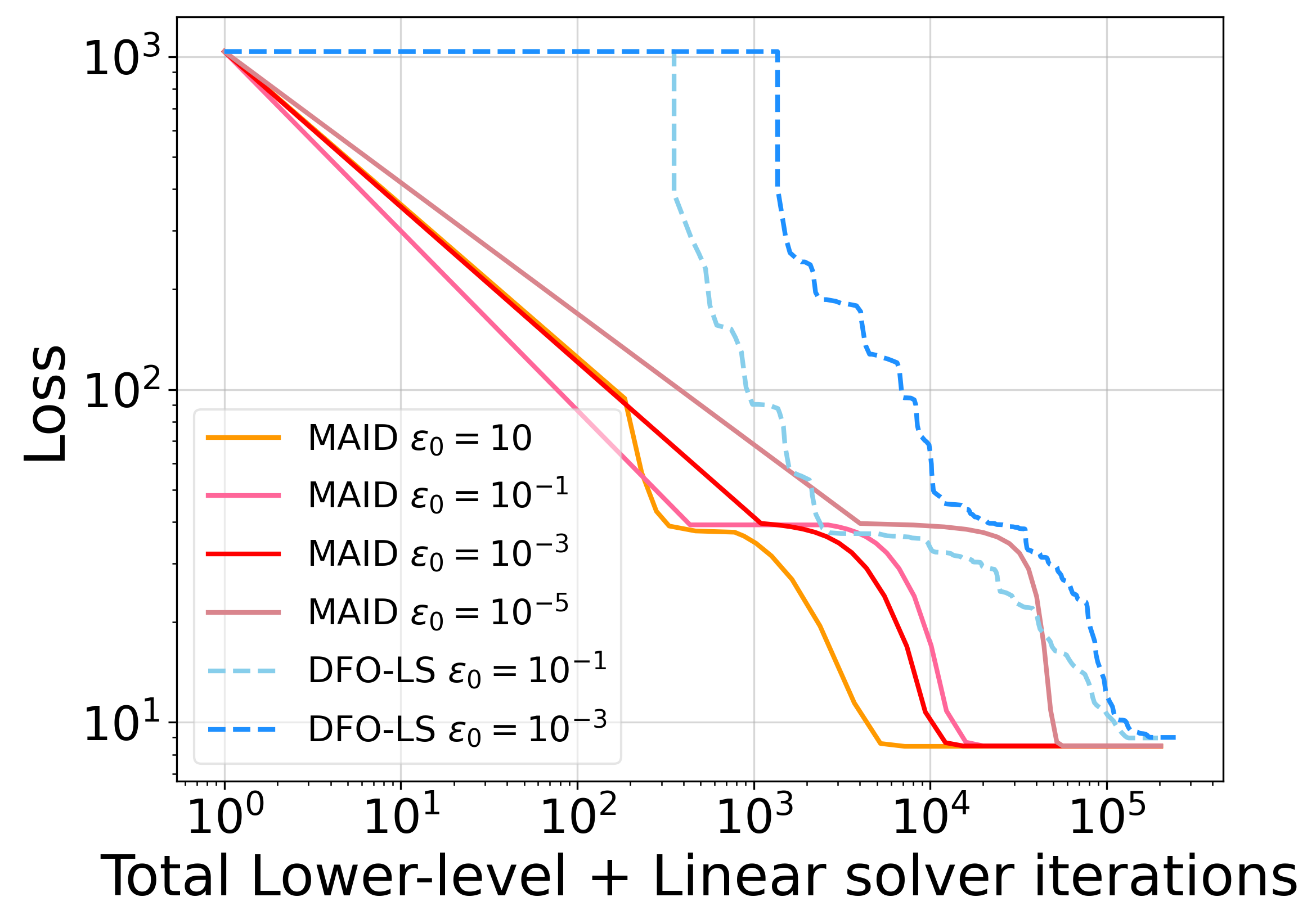}}
\caption{The values of $\epsilon$, and upper-level loss per total lower-level computational cost in the algorithms MAID and DFO-LS for solving the problem \eqref{bilevel_TV}. In various MAID settings, $\epsilon_0 = \delta_0$ is set.}
\label{fig:TV_step_eps_loss}
\end{figure}
Illustrated in \cref{fig:TV_eps}, MAID demonstrates that, regardless of the starting accuracy of the lower-level solver $\epsilon$---chosen with different orders of magnitude---it adapts the accuracy and converges to the same range after a few upper-level iterations. This highlights the robustness of MAID in the selection of an initial accuracy, enabling it to find values as large as necessary and as small as needed to make progress in backtracking line search. On the other hand, the FISTA variant of DFO-LS initially selects an equal or higher value of $\epsilon$ compared to MAID when using the same starting accuracies. However, at a certain point, it begins to significantly decrease $\epsilon$. 

In terms of loss, as indicated in \cref{fig:TV_loss_ll_cg}, with an equal total lower-level computational budget (assuming CG iterations of DFO-LS as zero), MAID surpasses the performance of the FISTA variant of DFO-LS, even in the case where the initial $\epsilon$ is chosen too small. All different setups of MAID converged to the loss value of $8.49$, while DFO-LS with $\epsilon_0= 10^{-1}$ and $\epsilon_0= 10^{-3}$ reached $8.97$ and $9.00$, respectively. Furthermore, it is important to note that DFO-LS does not utilize any hypergradient information, which could place it at a disadvantage when compared to gradient-based methods.

\rv{To further compare MAID with the existing inexact first-order method HOAG, we selected the best initial setting from \cref{fig:TV_loss_ll_cg} (i.e., $\epsilon_0 = \delta_0 = 10$) and tuned HOAG with geometric, quadratic, and cubic accuracy sequences. \cref{fig:TV_step_eps_loss_HOAG} presents a performance comparison between these variants of HOAG and MAID.
As shown in \cref{fig:TV_eps_HOAG}, MAID adaptively selected the required accuracy, while the HOAG variants followed three different a priori accuracy sequences, each giving different results in \cref{fig:TV_loss_ll_cg_HOAG}. After $10^4$ lower-level computations (lower-level solver plus linear solver iterations), both MAID and the quadratic variant HOAG converged. Extending the computational budget allowed HOAG with cubic accuracy to converge as well, though at a significantly higher computational cost. In contrast, HOAG with a geometric sequence converged to a higher loss. These outcomes, illustrated in \cref{fig:TV_step_eps_loss_HOAG}, highlight MAID's adaptability and robustness in selecting suitable accuracy without tuning or requiring prior knowledge. Moreover, MAID's faster convergence with lower computational costs underscores its better performance compared to HOAG for solving \cref{bilevel_TV}.}

\begin{figure}[tbhp]
\centering \subfloat[\footnotesize \rv{Accuracy of the lower-level solver in each upper-level iteration}]{\label{fig:TV_eps_HOAG}\includegraphics[width = 0.4\textwidth]{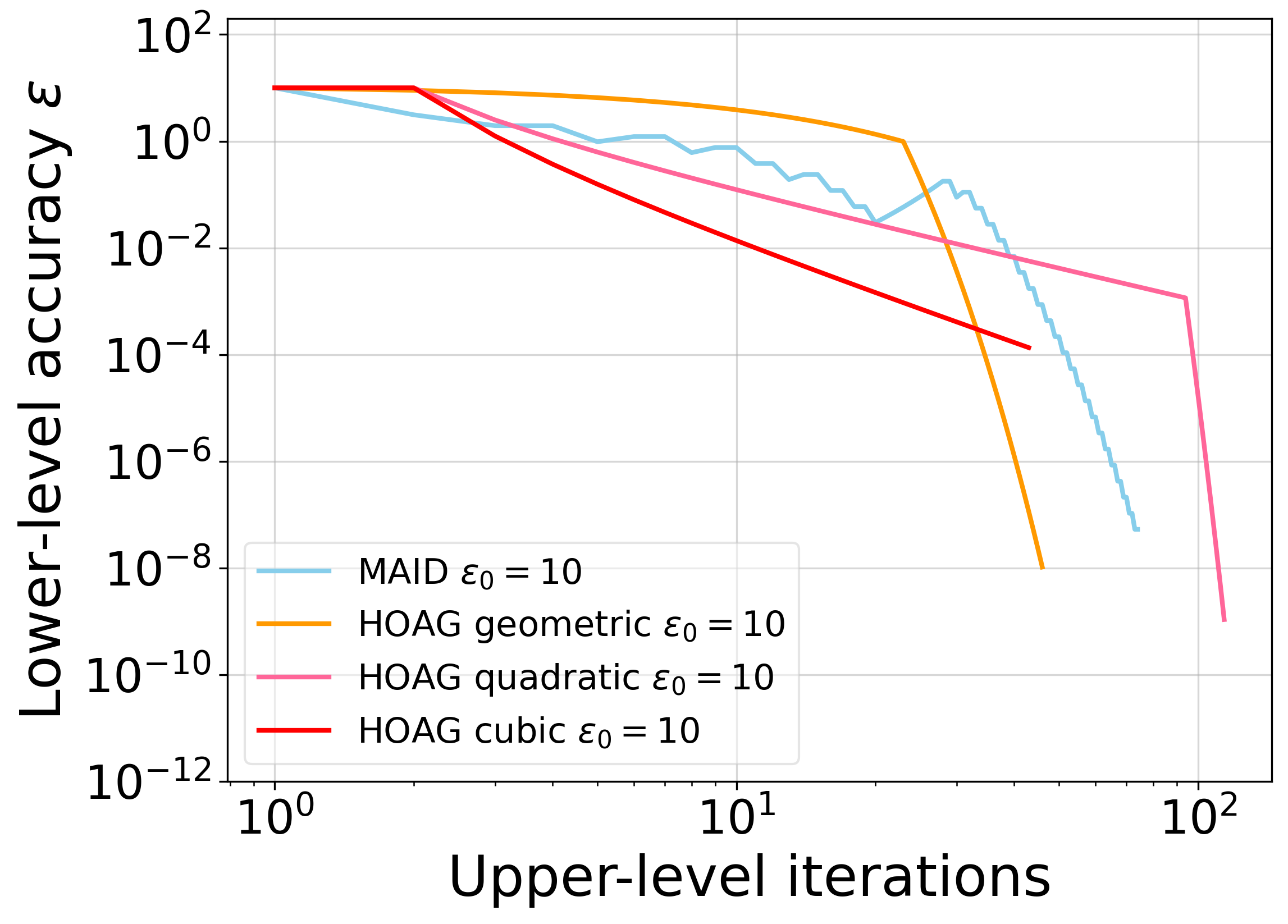}}\hspace{10pt}
\subfloat[\footnotesize \rv{Upper-level loss vs total lower-level computational cost}]{\label{fig:TV_loss_ll_cg_HOAG}\includegraphics[width = 0.4\textwidth]{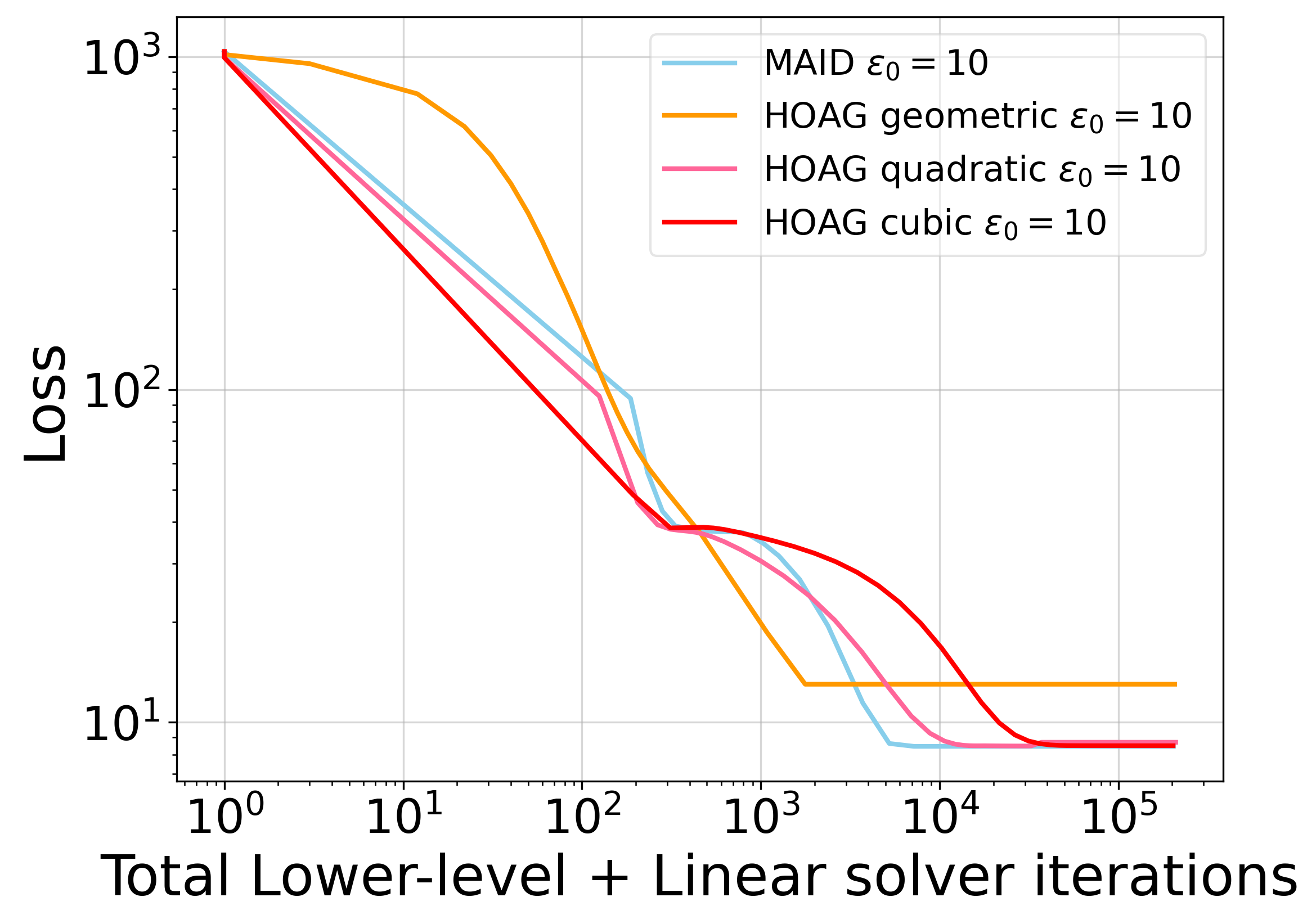}}
\caption{\rv{The values of $\epsilon$, and upper-level loss per total lower-level computational cost in the algorithms MAID and HOAG for solving the problem \eqref{bilevel_TV}. In all settings of HOAG and MAID $\epsilon_0 = \delta_0$ is set.}}
\label{fig:TV_step_eps_loss_HOAG}
\end{figure}
Now, to evaluate the effectiveness of the learned parameters obtained through solving \eqref{bilevel_TV} using MAID and dynamic DFO-LS, We solved the training problem \eqref{bilevel_TV} for $20$ images of the Kodak dataset, each of size $256 \times 256$ with Gaussian noise $\sigma = 0.1$, using MAID with $\epsilon_0 = \delta_0 = 10^{-1}$ and the FISTA variant of dynamic DFO-LS with $\epsilon_0 = \Delta_{0} = 10^{-1}$, where $\Delta_{0}$ represents the trust region radius in the first upper-level iteration. Subsequently, we applied these learned parameters to the lower-level total variation denoising problem for a test noisy image of size $256 \times 256$ and standard deviation $\sigma = 0.1$.
\begin{figure}[tbhp]
\centering 
\subfloat[\centering Ground truth]{\label{fig:Original_TV}\includegraphics[width = 0.24\textwidth]{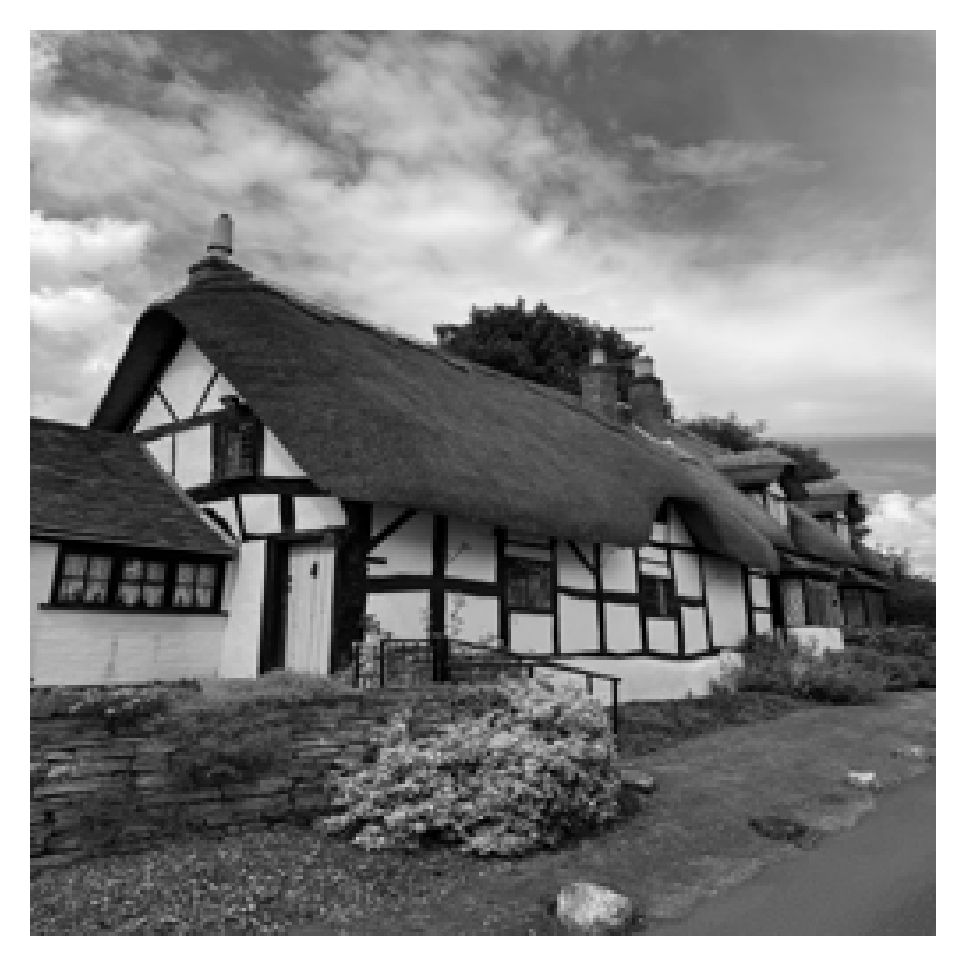}} \hfill \hspace{2pt}
\subfloat[\centering Noisy image, 
$\text{PSNR} = 19.99 \ \text{dB}$]{\label{fig:Noisy_TV}\includegraphics[width = 0.24\textwidth]{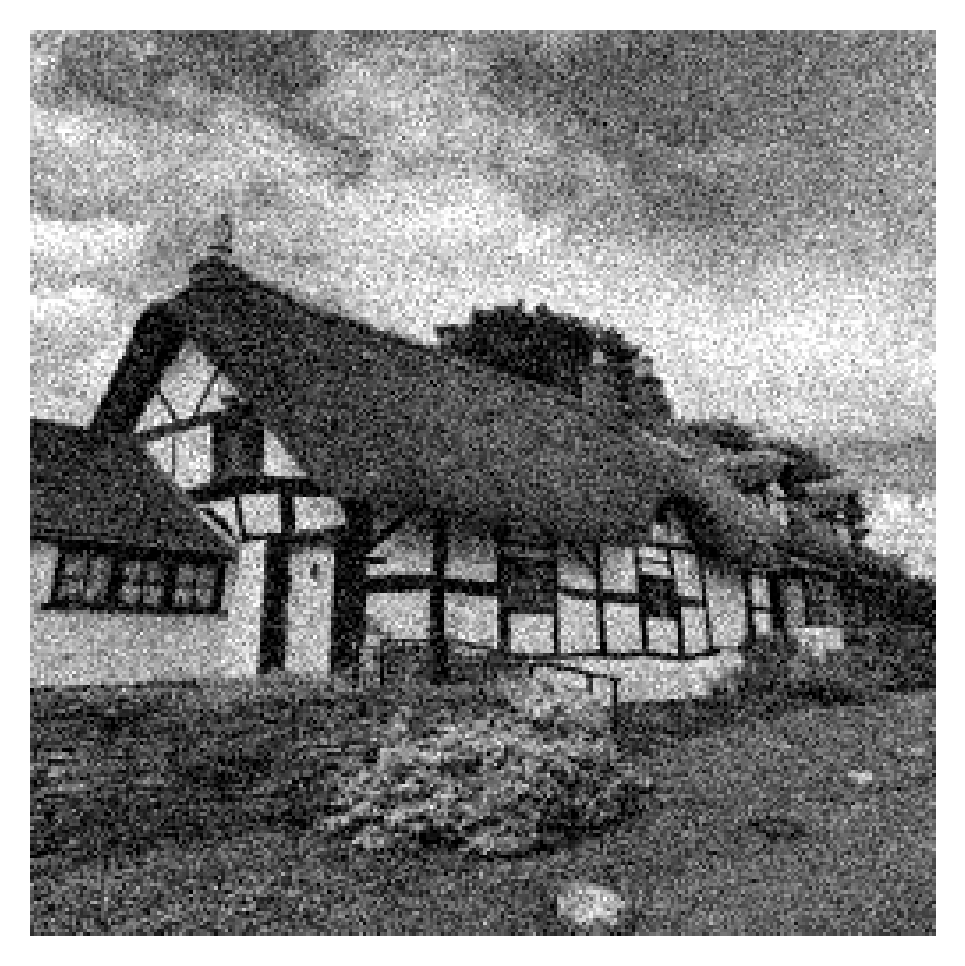}} \hfill \hspace{2pt}
\subfloat[\centering MAID, $\text{PSNR} = 26.90 \ \text{dB}$]{\label{fig:MAID_denoise_TV}\includegraphics[width = 0.24\textwidth]{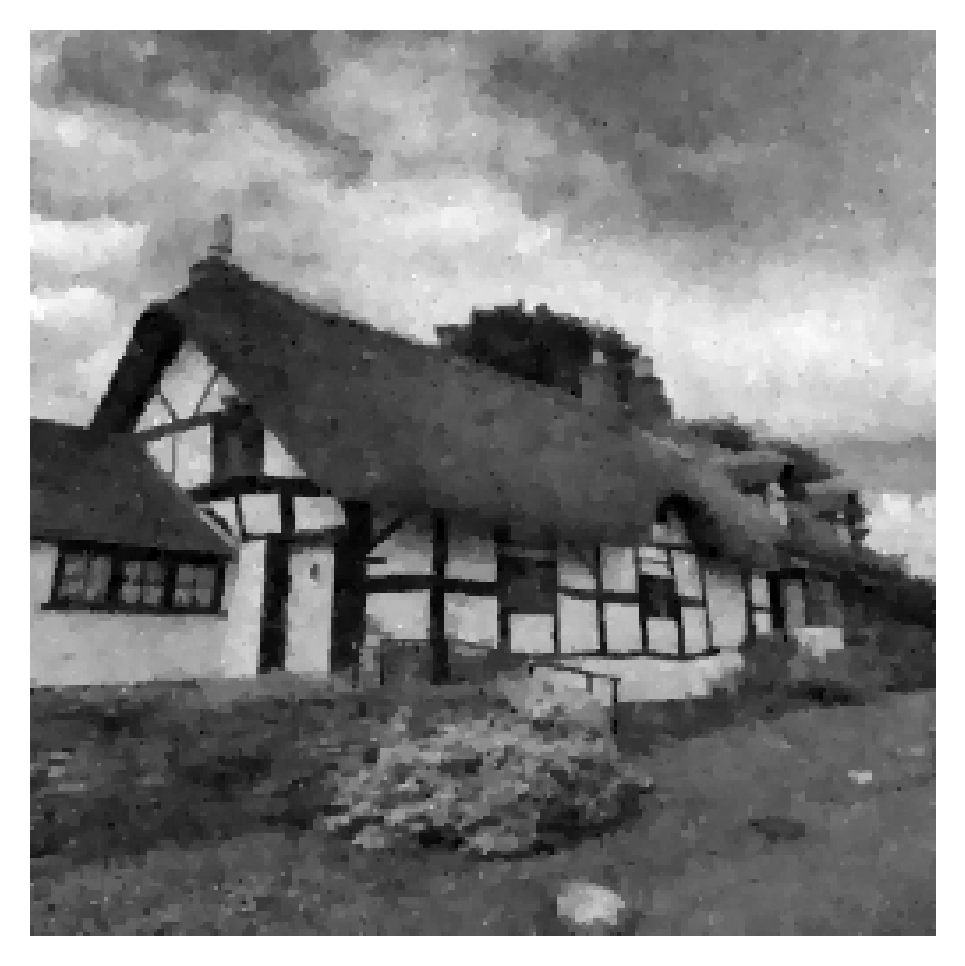}}\hfill \hspace{2pt}
\subfloat[\centering DFO-LS, $\text{PSNR} = 26.71 \ \text{dB}$]{\label{fig:DFO_denoise_TV}\includegraphics[width = 0.24\textwidth]{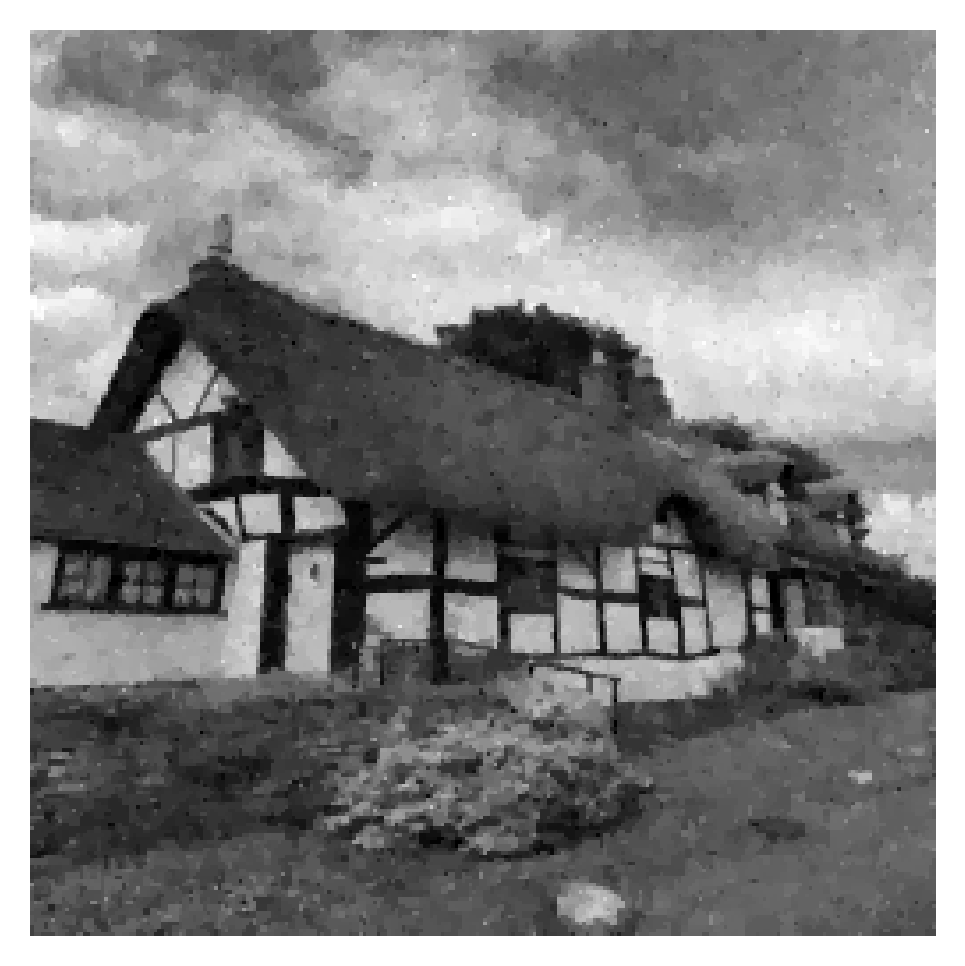}}

\caption{Total variation denoising with learned parameters by MAID ($\epsilon_0 = \delta_0 = 10^{-1}$), dynamic DFO-LS ($\epsilon_0 = 10^{-1}$).}
\label{denoising_TV}
\end{figure}

Figure \ref{denoising_TV} shows the denoising results for both algorithms. It displays the original and noisy image as well as the denoised images obtained using MAID and DFO-LS, respectively. Each denoised image is accompanied by its corresponding peak signal-to-noise ratio (PSNR) value. The learned hyperparameters using MAID result in slightly higher PSNR values than DFO-LS learned hyperparameters on test images, which is consistent with the training results. \rv{This indicates that both MAID and DFO-LS effectively solve \cref{bilevel_TV} with adaptability and robustness to initial accuracy. However, MAID leverages first-order information, resulting in superior performance and scalability.} 

\rv{Using the same setup as in the comparison with DFO-LS illustrated in \cref{denoising_TV}, we solve \cref{bilevel_TV} using MAID and HOAG variants. Given a computational budget of $10^4$, the resulting denoised test image based on parameters learned by MAID and HOAG is shown in \cref{fig:MAID_denoise_TV_HOAG}. As expected from \cref{fig:TV_loss_ll_cg_HOAG}, only the quadratic variant of HOAG achieves comparable results to MAID, while the cubic variant requires more iterations, and the geometric variant remains stalled in a suboptimal solution, potentially requiring further tuning. Thus, \cref{denoising_TV_HOAG} underscores the performance and robustness of MAID over HOAG in this example.}

\begin{figure}[tbhp]
\centering 
\subfloat[\centering \rv{MAID, $\text{PSNR} = 26.90 \ \text{dB}$}]{\label{fig:MAID_denoise_TV_HOAG}\includegraphics[width = 0.24\textwidth]{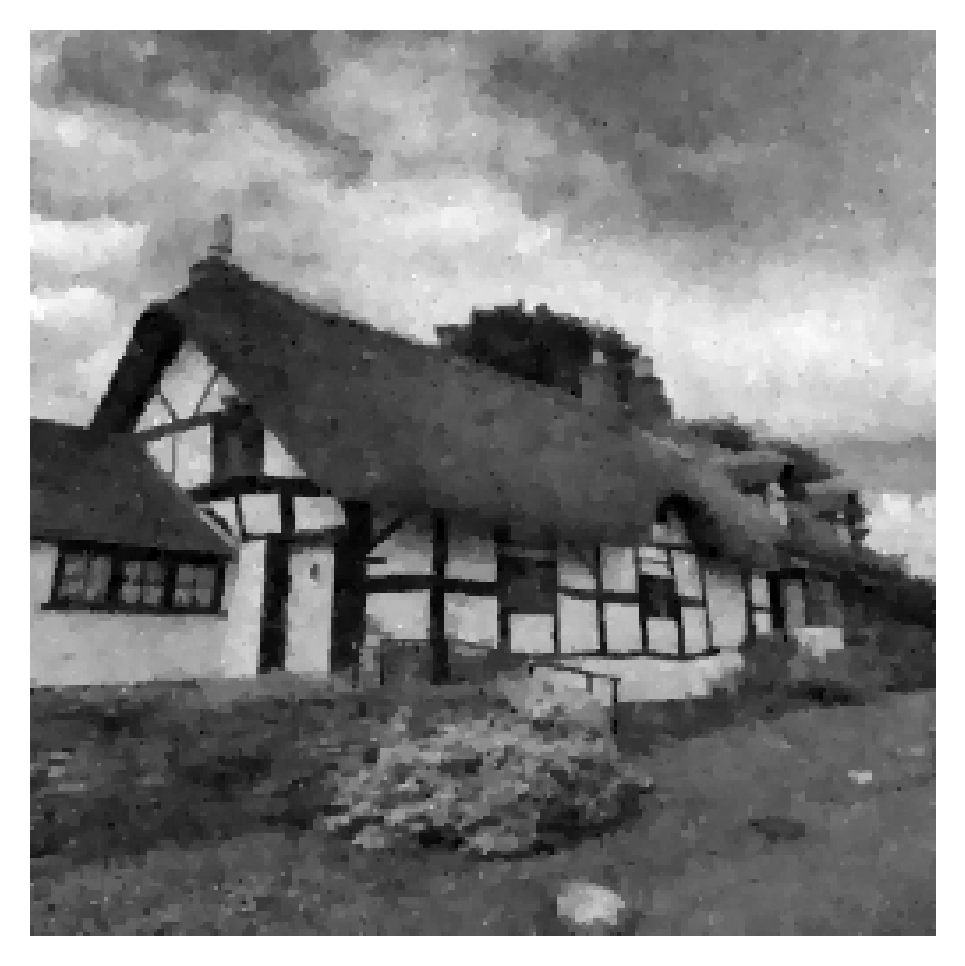}}\hfill \hspace{2pt}
\subfloat[\centering \rv{HOAG quadratic, $\text{PSNR} = 26.90 \ \text{dB}$}]{\label{fig:HOAF_quad_TV}\includegraphics[width = 0.24\textwidth]{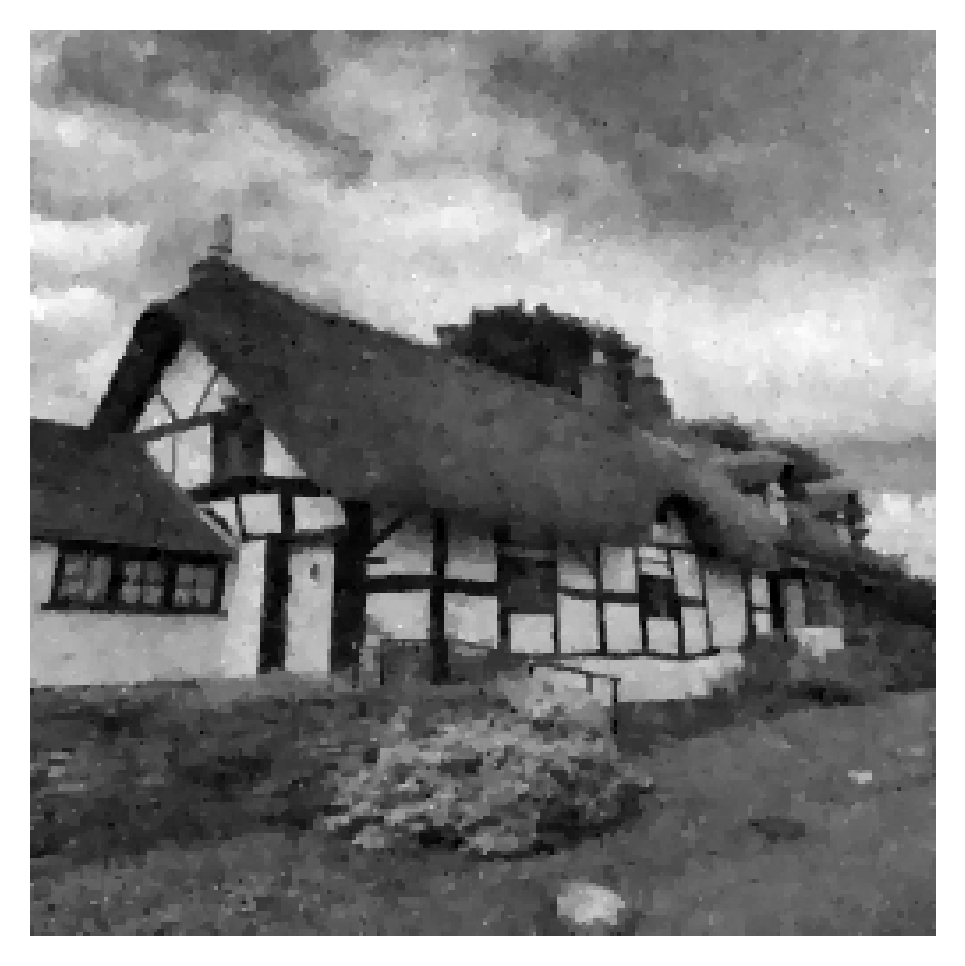}}
\hfill \hspace{2pt}
\subfloat[\centering \rv{HOAG cubic, $\text{PSNR} = 25.17 \ \text{dB}$}]{\label{fig:HOAG_cub_TV}\includegraphics[width = 0.24\textwidth]{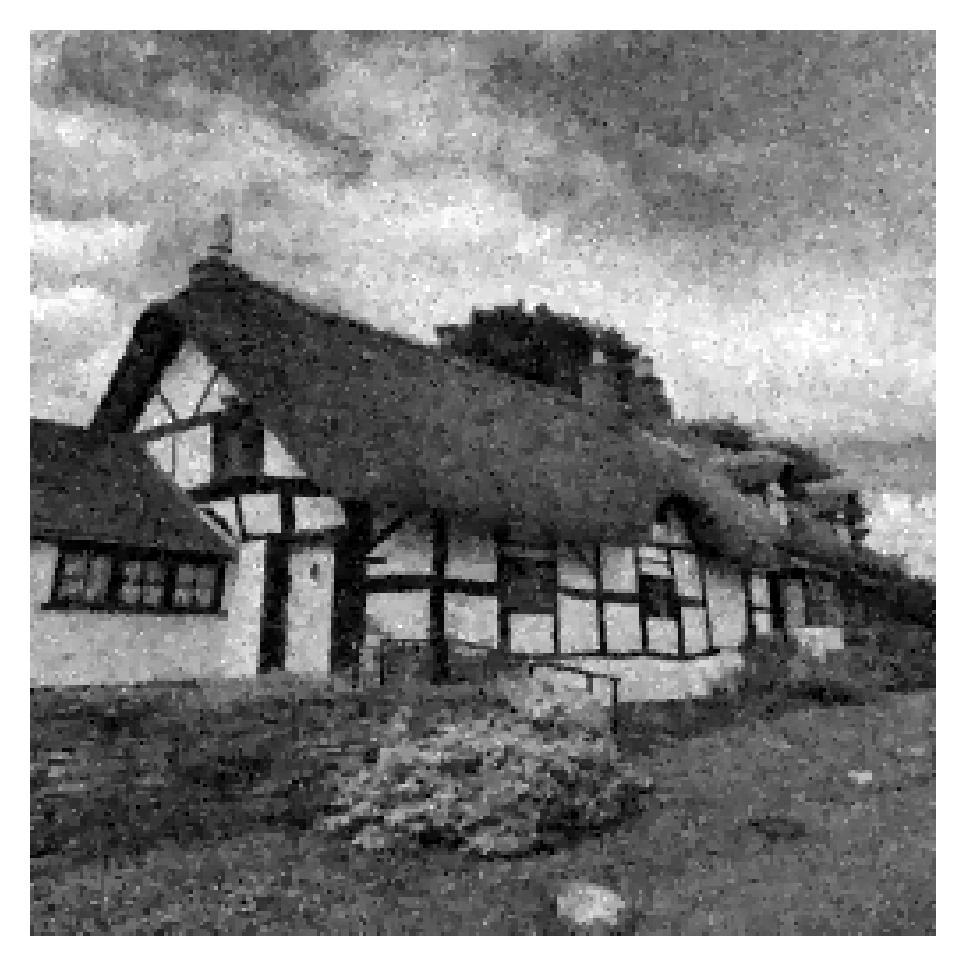}}
\hfill \hspace{2pt}
\subfloat[\centering \rv{HOAG geometric, $\text{PSNR} = 21.77 \ \text{dB}$}]{\label{fig:HOAG_geo_TV}\includegraphics[width = 0.24\textwidth]{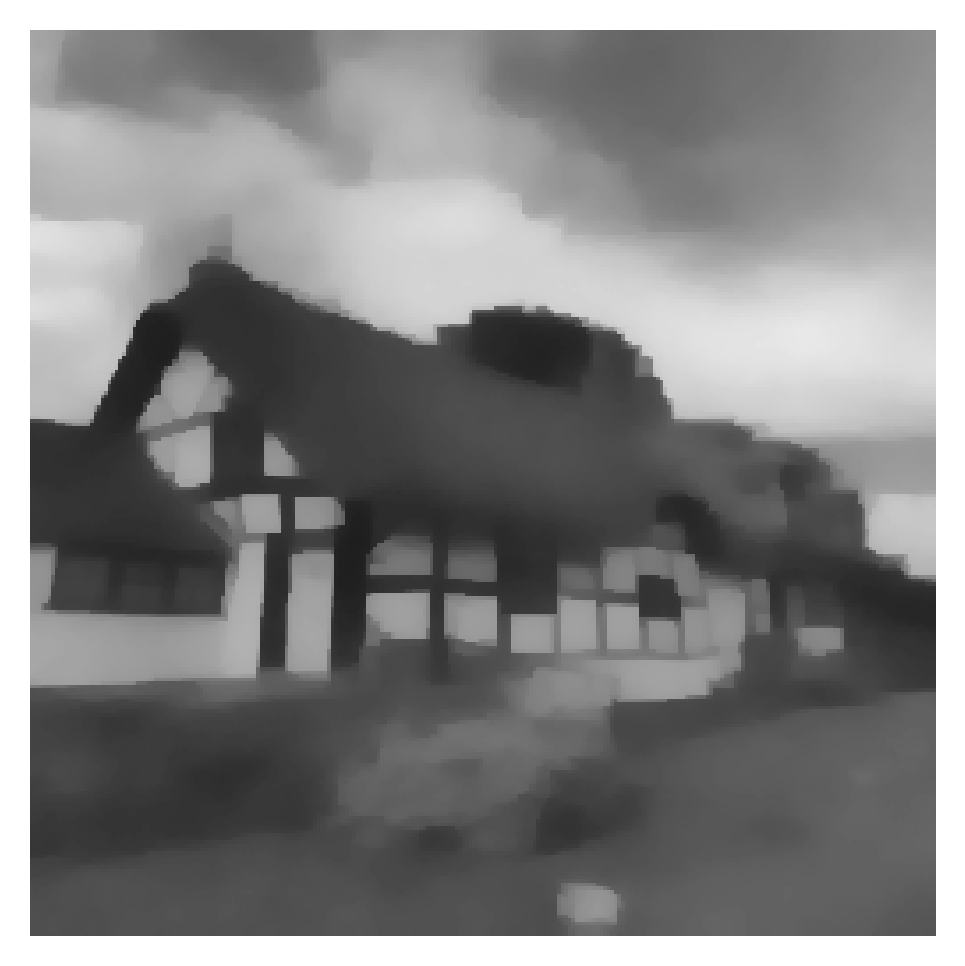}}
\caption{\rv{Total variation denoising with learned parameters by MAID and HOAG ($\epsilon_0 = \delta_0 = 10$) with $10^4$ computational budget.}}
\label{denoising_TV_HOAG}
\end{figure}

\rv{
We now demonstrate MAID’s performance on a problem involving a non-convex upper-level loss. Consider the total variation denoising lower-level problem \eqref{lower_TV} within the bilevel setting \eqref{bilevel_TV}. We replace the upper-level loss in \eqref{upper_TV} with the following smooth, non-convex loss from \cite{Tuckey,Carmon2017ConvexUP}:
\begin{equation}\label{robust_MSE_nonconvex}
    g(x) = \frac{\|x - x^*\|^2}{1 + \|x - x^*\|^2}.
\end{equation}
This results in the modified upper-level problem below:
\begin{equation}\label{upper_robust_MSE_nonconvex}
    \min_\theta \frac{1}{m}\sum_{t = 1}^m \frac{\|x_t - x^*_t\|^2}{1 + \|x_t - x^*_t\|^2}.
\end{equation}
Using the same setting as in \eqref{bilevel_TV} and applying \eqref{psi} in \cref{HOAH_back_new}, we solve this problem. \cref{fig:TV_robust_step_eps_loss} illustrates MAID’s convergence behavior and robustness to the initial accuracy when solving the bilevel problem with the non-convex upper-level loss \eqref{robust_MSE_nonconvex}.}
\begin{figure}[tbhp]
\centering \subfloat[\footnotesize \rv{Accuracy of the lower-level solver in each upper-level iteration}]{\label{fig:TV_robust_eps}\includegraphics[width = 0.4\textwidth]{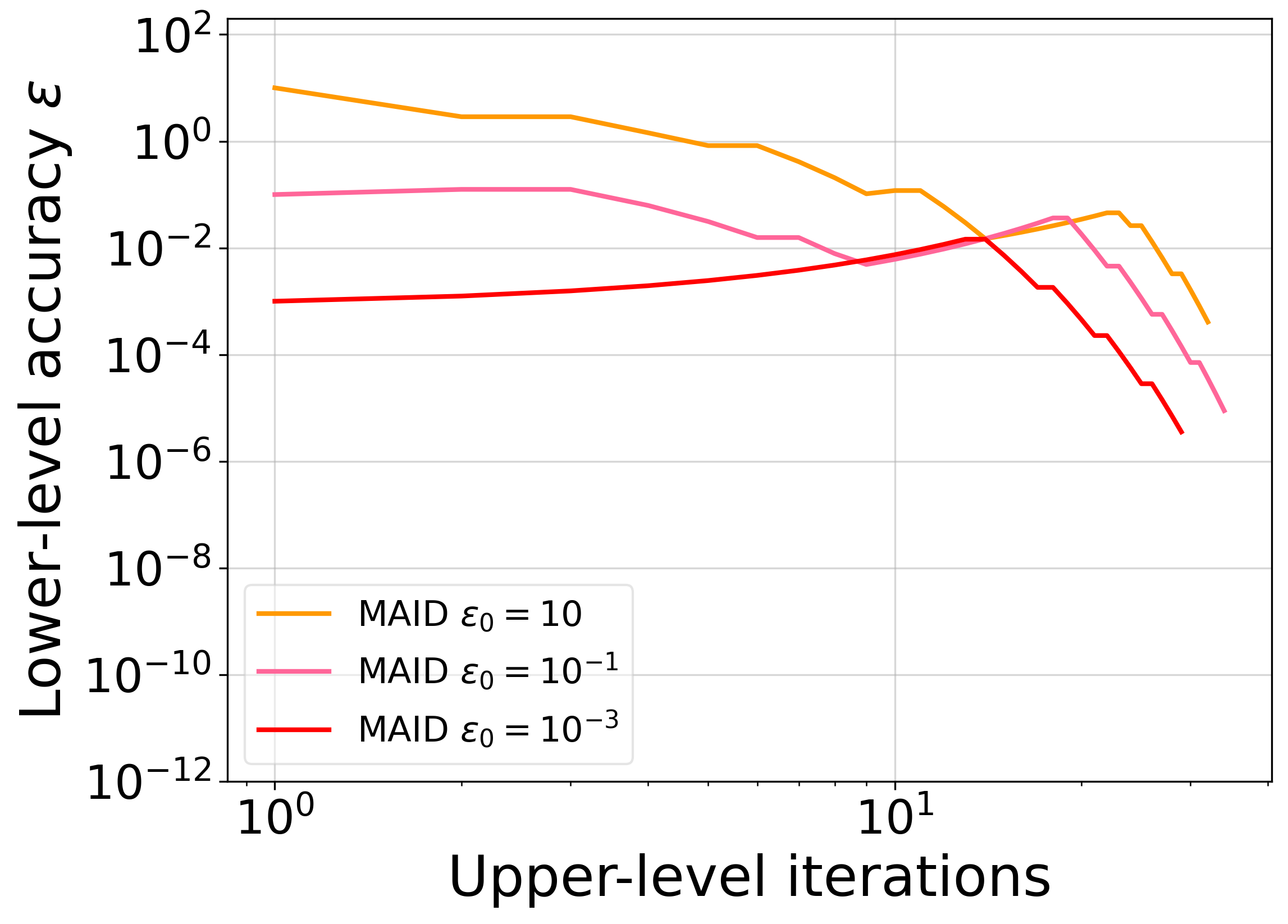}}\hspace{10pt}
\subfloat[\footnotesize \rv{Upper-level loss vs total lower-level computational cost}]{\label{fig:TV_robust_loss_ll_cg}\includegraphics[width = 0.4\textwidth]{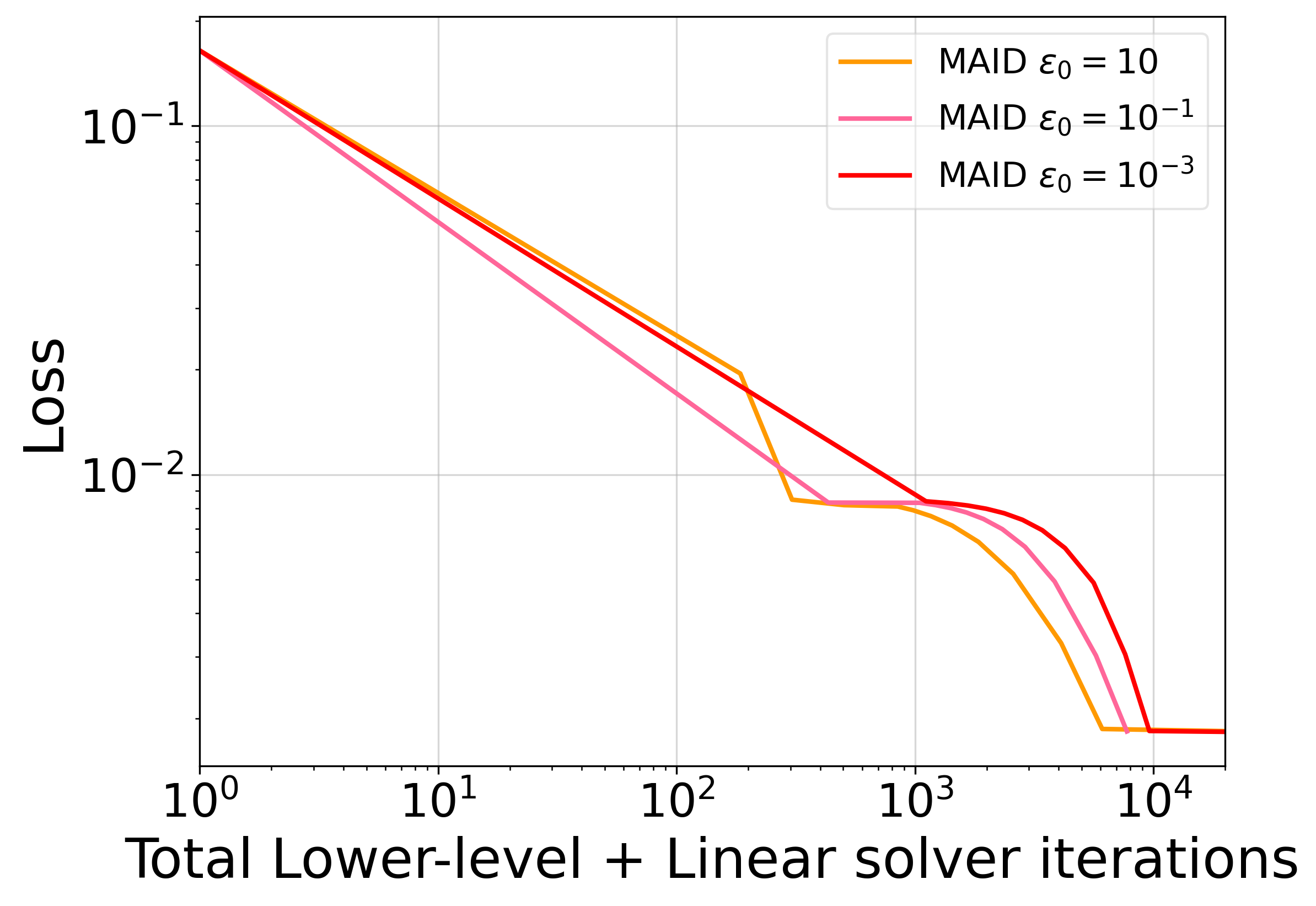}}
\caption{\rv{Values of $\epsilon$, and upper-level loss per total lower-level computational cost in MAID for learning parameters in the ROF model \eqref{lower_TV} with the non-convex upper-level loss \eqref{robust_MSE_nonconvex} and corresponding upper-level problem \eqref{upper_robust_MSE_nonconvex}. In various MAID settings, $\epsilon_0 = \delta_0$ is set.}}
\label{fig:TV_robust_step_eps_loss}
\end{figure}

\subsection{Field of experts denoising}\label{subsec:FoE}
In this part, we explore image denoising using a more data-adaptive regularizer known as the Field of Experts (FoE) approach. This approach was previously employed in a bilevel setting in \cite{Yunjin_Chen_2014} to learn the parameters of the regularizer in the variational setting. However, their choice of regularizer was non-convex and does not align with our assumptions. Therefore, we adopt the convex case, as described in \cite{Crockett_2022}, resulting in the following bilevel problem:
\begin{subequations}\label{FoE}
\begin{align}
&\min_\theta \frac{1}{m} \sum_{i=1}^{m} \frac{1}{2}\|\hat{x}_i(\theta) - x^*_t \|^2\\
&s.t. \quad \hat{x}_i(\theta) = \arg \min_x \left\{ \frac{1}{2}\|x-y_i\|^2 + e^{\theta[0]} \sum_{k=1}^K e^{\theta[k]} \| c_k \ast x \|_{\theta[{K+k}] } \right\}, \quad i = 1,\dots, m,
\end{align}
\end{subequations}
where each $c_k$ is a $2D$ convolution kernel. We set $K = 30$ and take each $c_k$ as a $7 \times 7$ kernel with elements $[\theta[{2K+ 49(k-1)+1}], \dots,\theta[{2K+ 49 k}]]$. Also, $\|x\|_{\upsilon} = \sum_{j = 1}^n \sqrt{x_j^2+\upsilon^2}$ represents the smoothed $1$-norm with smoothing parameter $\upsilon$ for any $x\in \mathbb{R}^n$. So, this problem has $1531$ hyperparameters that we represent by $\theta$. We initialize the kernels by drawing randomly from a Gaussian distribution with mean $0$ and standard deviation $1$, and we set the initial weights to $e^{-3}$. During training, similar to TV denoising, we use 25 images from the Kodak dataset and rescale them to $96\times 96$ pixels, converting them to grayscale. Additionally, we add Gaussian noise with $\mathcal{N}(0, \sigma^2)$, where $\sigma = 0.1$, to the ground truth images. \rv{Additionally, as in the previous experiment, we set the MAID hyperparameter $\overline{\rho} = \frac{10}{9}$, $\underline{\rho} = 0.5$, $\overline{\nu} = 1.25$, and $\underline{\nu} = 0.5$. } 

\begin{figure}[tbhp]
\centering 
\subfloat[\footnotesize Accuracy of the lower-level solver in each upper-level iteration]{\label{fig:FoE_eps}\includegraphics[width = 0.4\textwidth]{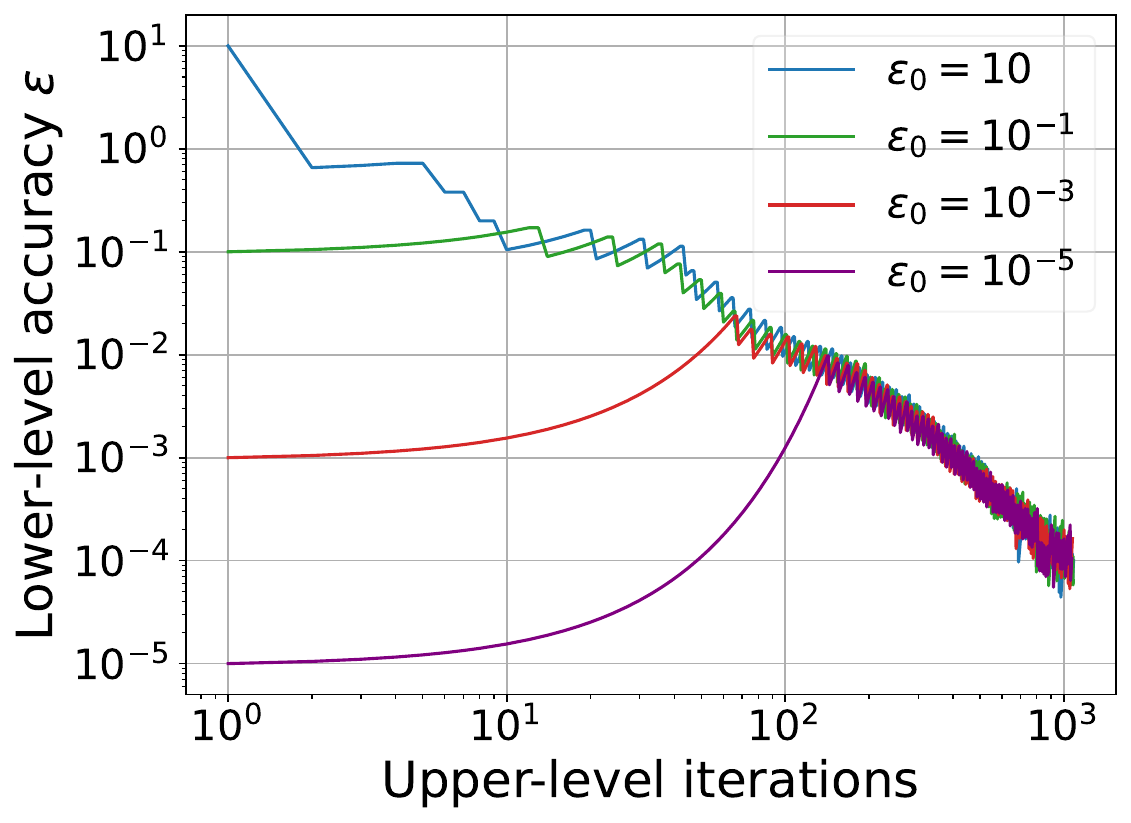}}\hspace{10pt}
\subfloat[\footnotesize Upper-level loss vs total lower-level computational cost]{\label{fig:FoE_LL}\includegraphics[width = 0.4\textwidth]{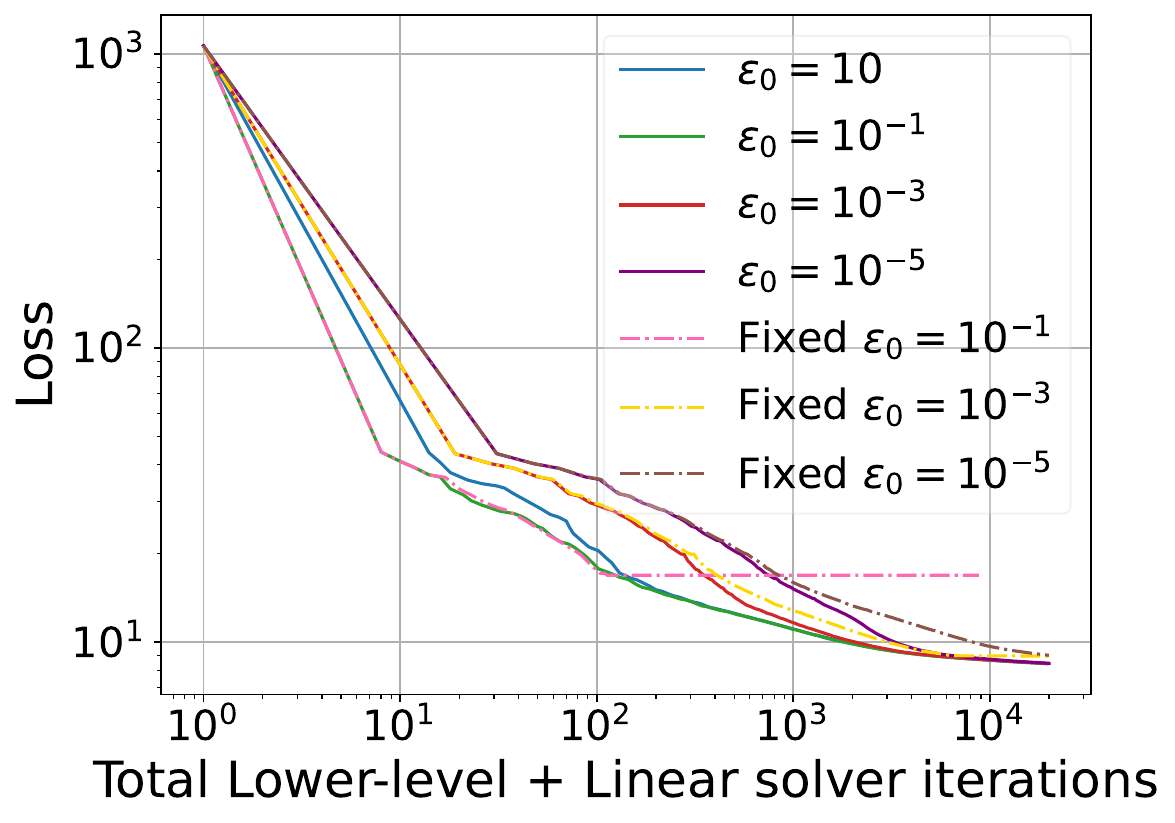}}
\caption{Investigating the impact of MAID configurations on the accuracy of the lower-level solver, and loss at each lower-level computation unit with varied $\epsilon_0 = \delta_0$ for solving the problem \eqref{FoE}.}
\label{fig:FoE_step_eps_LL}
\end{figure}
As seen in \cref{fig:FoE_eps}, by dynamically adjusting the accuracy within the framework of MAID, After a few upper-level iterations in various settings, the algorithm ultimately selects the same range of accuracies. Akin to the patterns observed in previous numerical experiments, choosing a too-large initial accuracy $\epsilon_0 = 10$ in this case prompts the algorithm to undergo reductions and recalculations to find a descent direction, resulting in higher computational costs initially. However, after sufficiently reducing $\epsilon$, it manages to make progress with a reasonable computational cost. 

The advantage of adaptive accuracy in MAID becomes evident when comparing it to fixed accuracies. Using a low fixed accuracy $\epsilon_0 = 10^{-1}$ along with the line search mechanism of MAID, once can see a rapid progress initially, as indicated by the pink dashed curve, but it reaches a plateau after a few upper-level iterations. Furthermore, it does not use the full lower-level budget because the upper-level progress stalls at some point, and due to warm-starting, the lower-level solution is accurate enough and does not require many more iterations. A medium fixed accuracy, such as $10^{-3}$, performs comparably to adaptive cases; however, it cannot be known a priori and must be chosen experimentally. A high fixed accuracy, like $\epsilon_0 = 10^{-5}$, appears sufficiently small but exhausts the lower-level solver and runs out of budget while achieving higher loss. All adaptive instances demonstrate enhanced overall efficiency compared to their fixed counterparts, as showcased in \cref{fig:FoE_LL}.

\rv{To compare MAID with HOAG, a scalable inexact first-order algorithm for solving \cref{FoE}, we ran HOAG with geometric, quadratic, and cubic accuracy sequences under the same settings, with $\epsilon_0 = \delta_0 = 10^{-1}$. As illustrated in \cref{fig:FoE_eps_HOAG}, MAID adaptively selected larger $\epsilon$ values while maintaining steady progress in \cref{fig:FoE_LL_HOAG}. For any budget exceeding $10^3$, MAID outperforms HOAG. Different HOAG variants vary in performance and likely require further tuning, while MAID remains robust to accuracy choices. The best-performing HOAG variant, quadratic, fluctuates significantly after $2\times 10^{4}$ computations, whereas MAID exhibits a steady, monotonic behavior. Overall, \cref{fig:FoE_eps_LL_HOAG} highlights MAID’s superior performance and stability in this experiment.}

\begin{figure}[tbhp]
\centering 
\subfloat[\footnotesize Accuracy of the lower-level solver in each upper-level iteration]{\label{fig:FoE_eps_HOAG}\includegraphics[width = 0.4\textwidth]{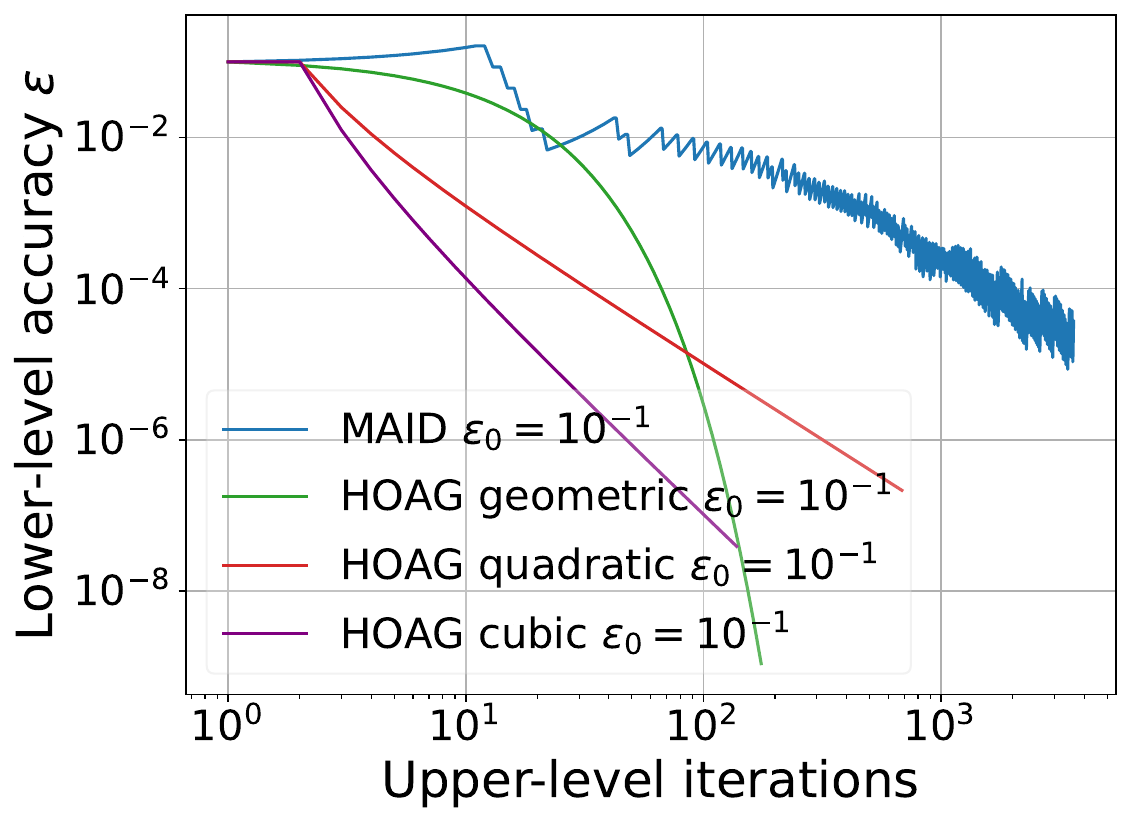}}\hspace{10pt}
\subfloat[\footnotesize Upper-level loss vs total lower-level computational cost]{\label{fig:FoE_LL_HOAG}\includegraphics[width = 0.4\textwidth]{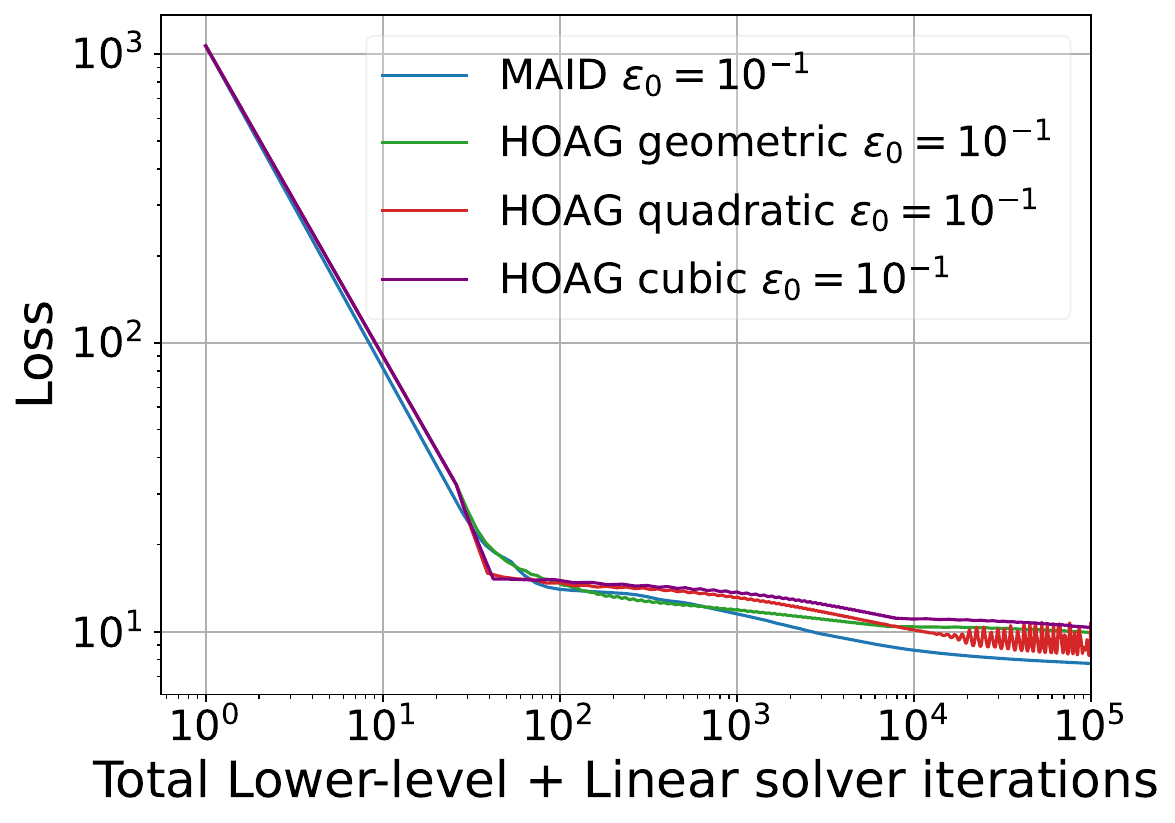}}
\caption{\rv{Comparison of MAID and HOAG on the lower-level solver accuracy and loss per lower-level computation unit, with initial settings $\epsilon_0 = \delta_0 = 10^{-1}$ for solving problem \eqref{FoE}.}}
\label{fig:FoE_eps_LL_HOAG}
\end{figure}
Noting the importance of the problem \eqref{FoE} as a large-scale and data-adaptive problem, to evaluate the learned parameters, we utilized the same set of $20$ images used in the total variation problem, each of size $256\times 256$. Additionally, we employed $48$ filters of size $7\times 7$, initialized with zero-mean random filters. After training using MAID \rv{ and variation of HOAG} with $\epsilon_0 = \delta_0 = 10^{-1}$, we evaluated the learned model on a noisy test image of size $256 \times 256$ pixels with Gaussian noise of $\sigma = 0.1$, as depicted in \cref{fig:Original_FOE} and \cref{fig:Noisy_FOE} respectively.

The denoised image obtained using parameters learned by MAID is shown in \cref{fig:MAID_denoise_FOE}, achieving a PSNR of 29.73 dB. \rv{ In comparison, the best-performing variant of HOAG in this experiment—the quadratic variant—produces a denoised image with a PSNR of 28.79 dB, as shown in \cref{fig:MAID_denoise_FOE_quad}. The quality drops to 28.03 dB and 27.96 dB when using the geometric and cubic variants of HOAG, respectively, as shown in \cref{fig:HOAG_denoise_FOE_geo} and \cref{fig:HOAG_denoise_FOE_cub}. In contrast, MAID maintains robust performance without requiring fine-tuning or prior knowledge of accuracy.}

\begin{figure}[tbhp]
\centering 
\subfloat[\footnotesize Ground truth]{\label{fig:Original_FOE}\includegraphics[width = 0.25\textwidth]{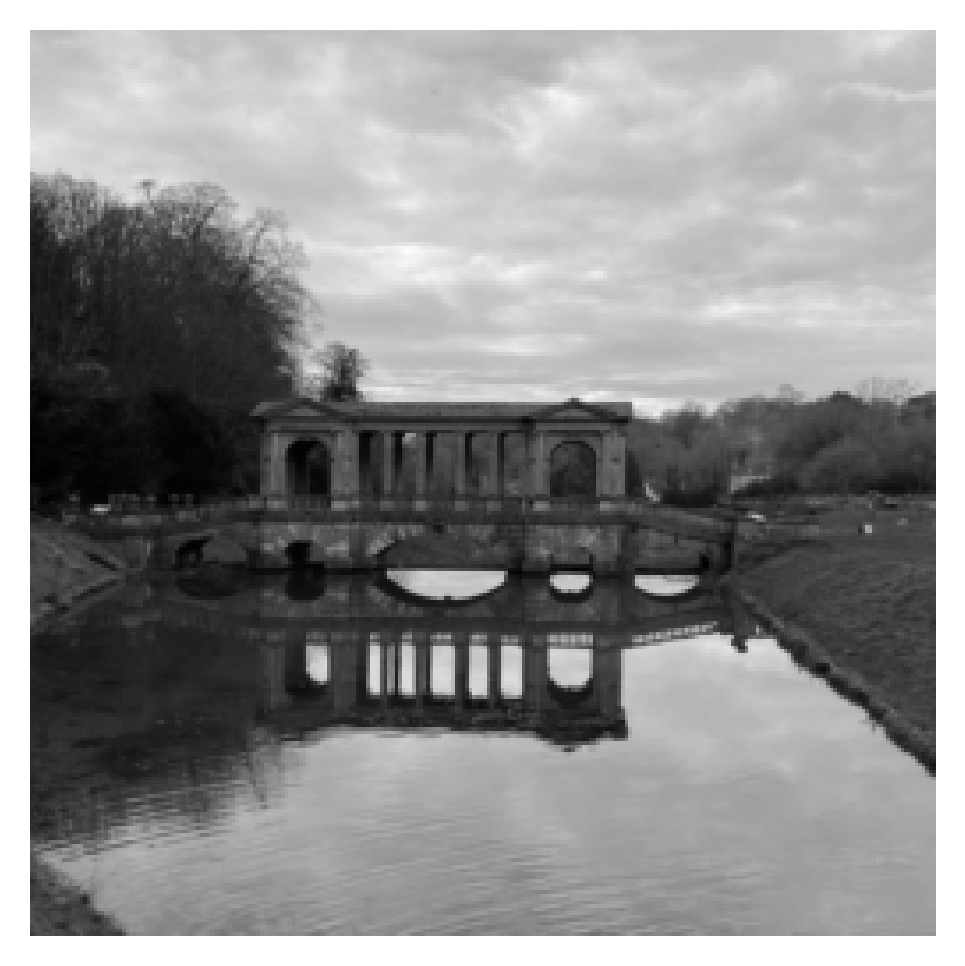}}\hfill \hspace{2pt}
\subfloat[\centering \footnotesize Noisy image, $\text{PSNR} = 20.28 \ \text{dB}$]{\label{fig:Noisy_FOE}\includegraphics[width = 0.25\textwidth]{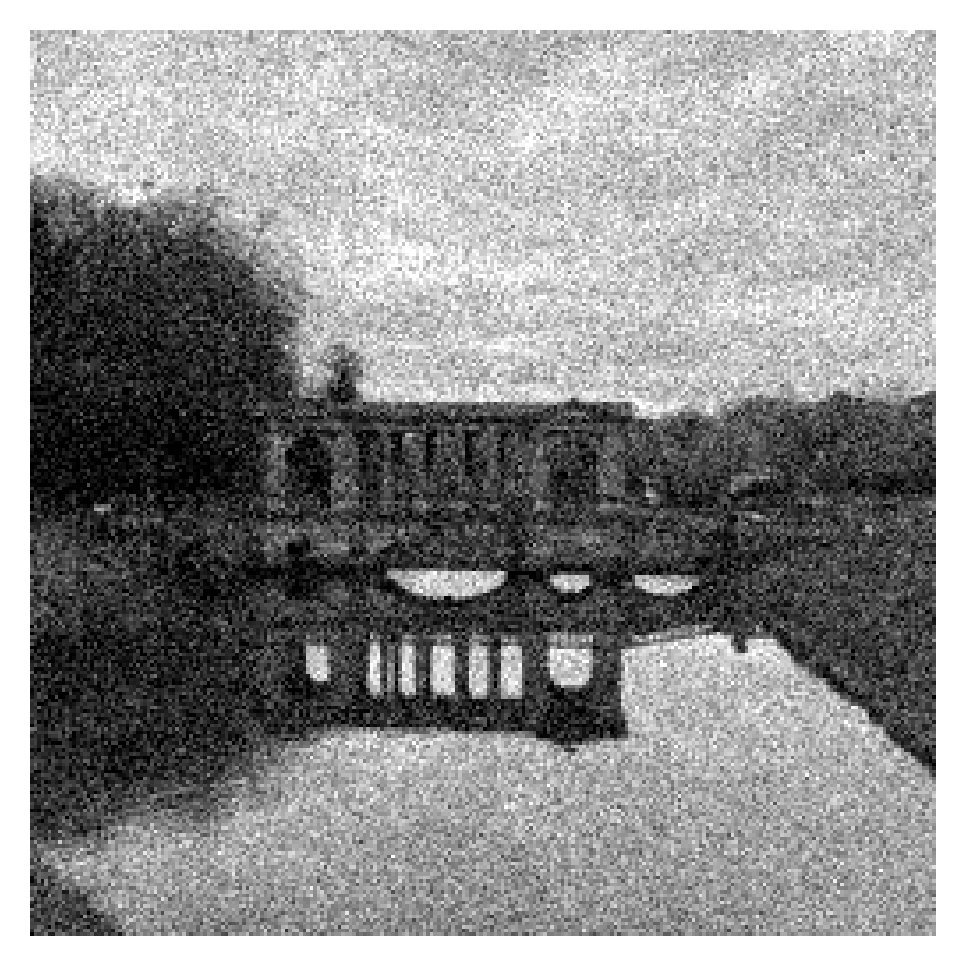}} \hfill \hspace{2pt}
\subfloat[\centering \footnotesize MAID, $\text{PSNR} = 29.73 \ \text{dB}$]{\label{fig:MAID_denoise_FOE}\includegraphics[width = 0.25\textwidth]{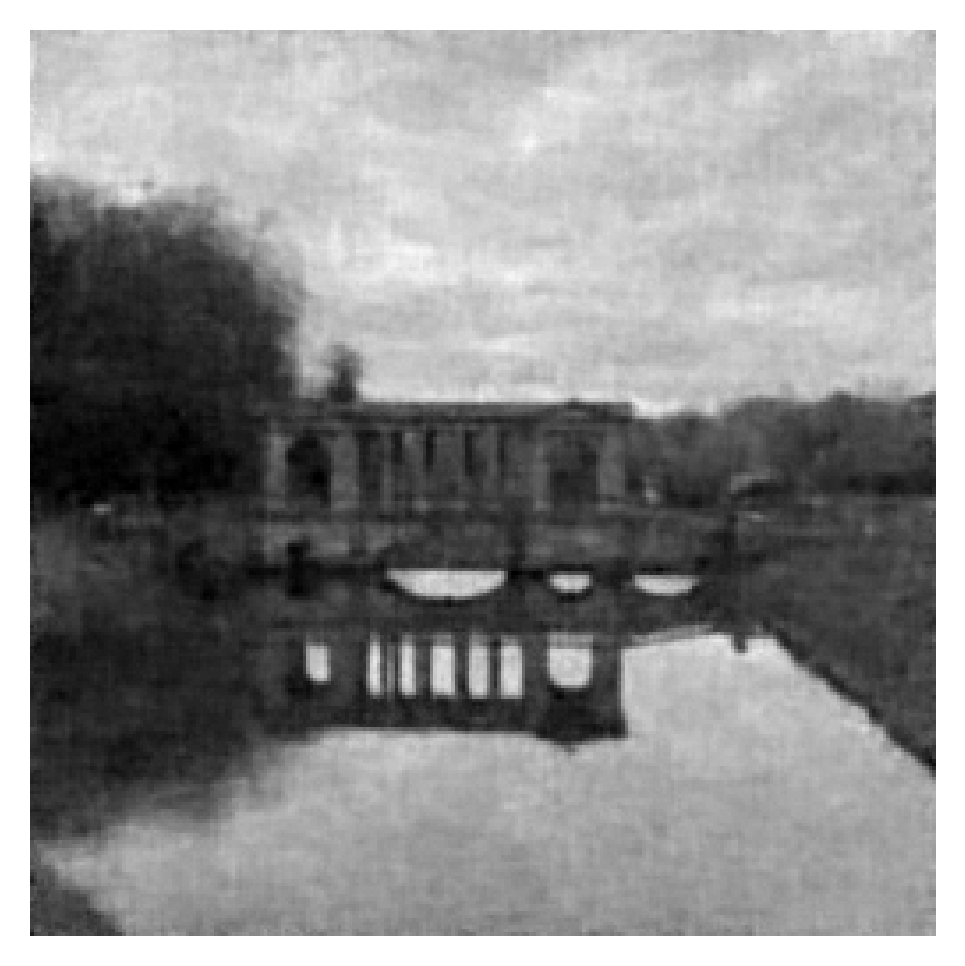}}  
 \vfill \hspace{2pt}
\subfloat[\centering \footnotesize \rv{HOAG quadratic, $\text{PSNR} = 28.79 \ \text{dB}$}]{\label{fig:MAID_denoise_FOE_quad}\includegraphics[width = 0.25\textwidth]{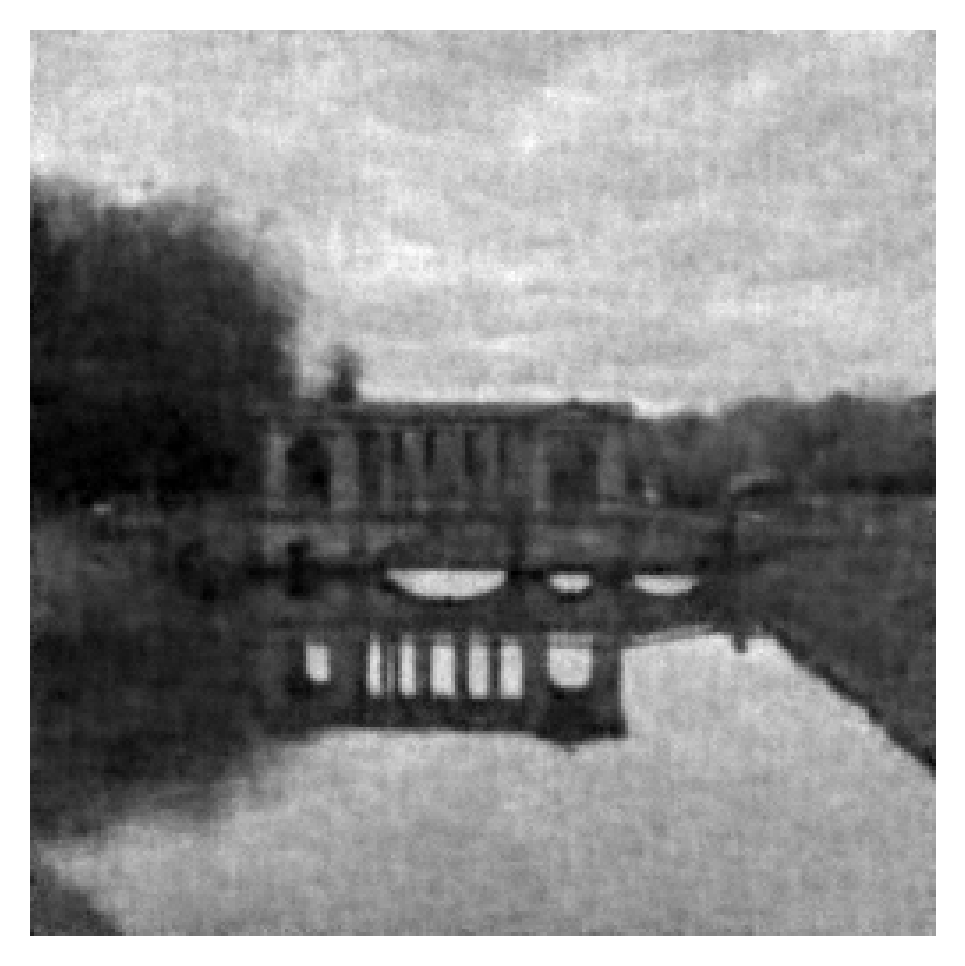}}
\hfill \hspace{2pt}
\subfloat[\centering \footnotesize \rv{HOAG geometric, $\text{PSNR} = 28.03 \ \text{dB}$}]{\label{fig:HOAG_denoise_FOE_geo}\includegraphics[width = 0.25\textwidth]{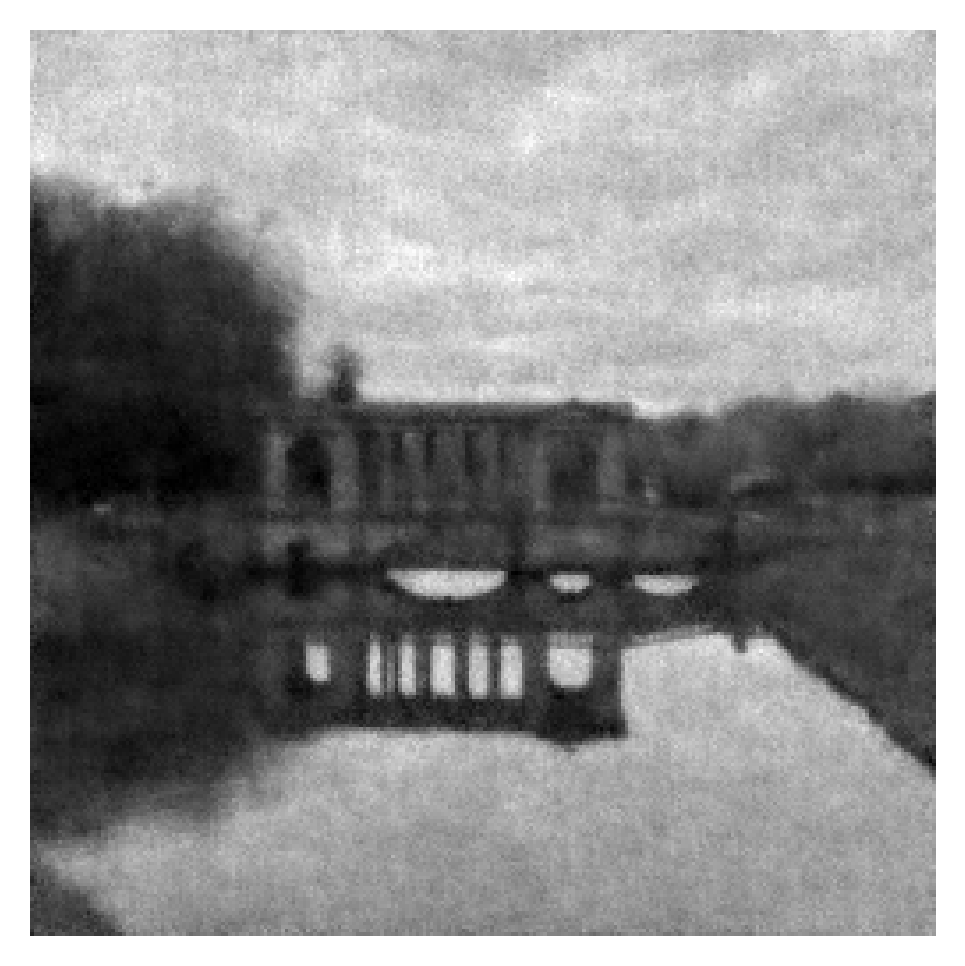}
}
\hfill \hspace{2pt}
\subfloat[\centering \footnotesize \rv{HOAG cubic, $\text{PSNR} = 27.96 \ \text{dB}$}]{\label{fig:HOAG_denoise_FOE_cub}\includegraphics[width = 0.25\textwidth]{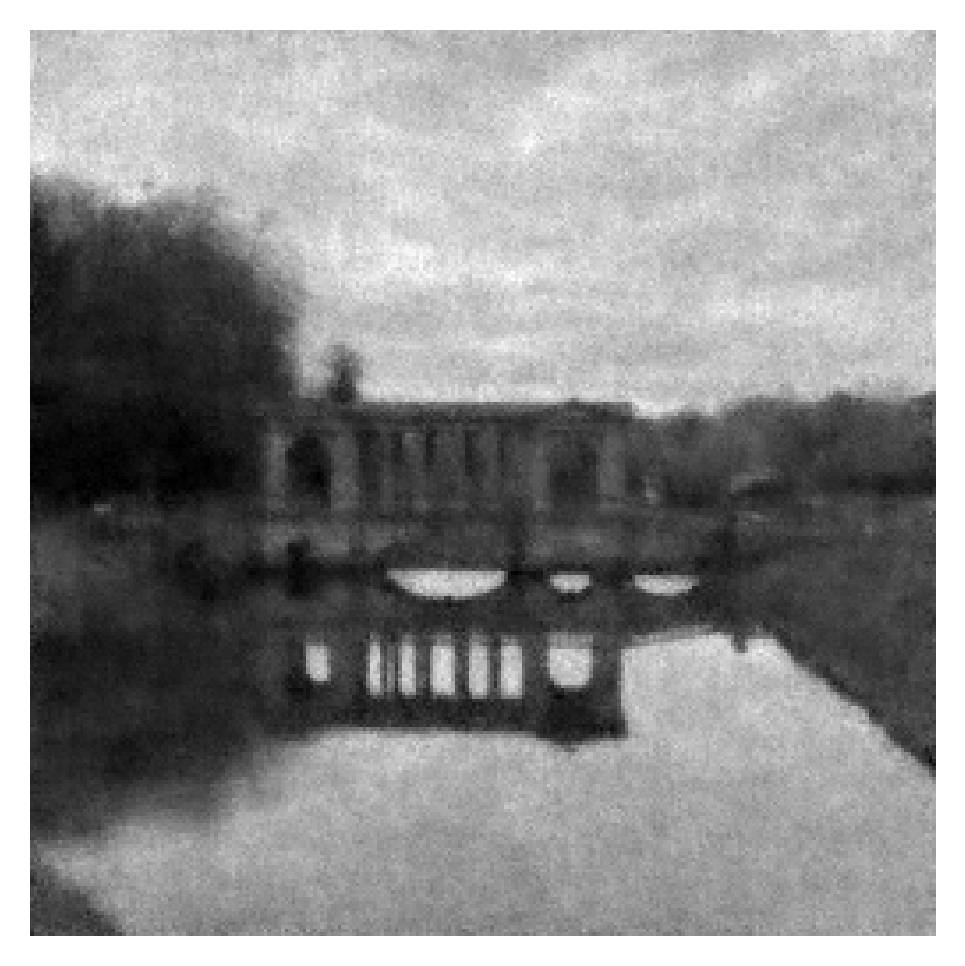}}
\caption{Field of experts denoising with learned parameters utilizing MAID \rv{in comparison with HOAG}.}
\label{denoising_FoE}
\end{figure}
\subsection{Runtime comparison}
\rvv{In our numerical results, the lower-level computational cost includes iterations of the lower-level solver as well as the matrix-vector products in the linear solver and power method involved in computing the bound \eqref{AP_bound_hg}. These were considered the primary computational costs, as CPU time can vary across problems and depends heavily on implementation. To support this choice, we present a comparison of CPU times for the lower-level solver, the matrix-vector product (linear solver plus power method), the computation of the bound \eqref{AP_bound_hg} for checking a descent direction, and the line search condition \eqref{psi} in \cref{inexact_bt_lemma} across all experiments in \cref{table:time_experiments}. The results show that, as expected, the majority of CPU time per upper-level iteration is spent on lower-level iterations and matrix-vector product.}

\rv{In addition to these results, note that, in terms of computational complexity, the extra calculations of $\nabla_{x\theta}^2h(x(\theta), \theta)$ on the right-hand side of \eqref{AP_bound_hg} (for example, when approximating $\|\nabla_{x\theta}^2h(x(\theta), \theta)\|$ using one step of the power method) are approximately twice as costly as computing a single gradient  $\nabla_x h(x, \theta)$, assuming the dimension  $d$  of the parameter  $\theta$  is comparable to the dimension $n$ of the input $x$. This is because the calculations are performed in a matrix-vector product manner, resulting in roughly two additional lower-level gradient computations per upper-level iteration.} \rvv{Moreover, we found the choice of using only one step of the power method to be numerically effective when computing the operator norm in \cref{AP_bound_hg}. Increasing the number of power method iterations did not change the behavior of the algorithm but led to higher computational costs than necessary.}
\begin{table}[htbp]
    \centering
    \setlength{\tabcolsep}{4pt} %
    \begin{tabular}{|c|c|c|c|c|}
        \hline
        \textbf{\rvv{Experiment}} & 
        \textbf{\rvv{FISTA}} &
        \textbf{\rvv{Matrix-vector}} & 
        \textbf{\rvv{Bound \eqref{AP_bound_hg}}} & 
        \textbf{\rvv{Line search \eqref{psi}}} \\ 
        \hline
        \rvv{Quadratic Section \ref{subsec:quad}} & \rvv{$7.8 \times 10^{-2}$} & \rvv{$7.1 \times 10^{-2}$} & \rvv{$6.9 \times 10^{-5}$} & \rvv{$5.8 \times 10^{-5}$} \\ 
        \hline
        \rvv{Regression Section \ref{subsec:regression}} & \rvv{$1.2 \times 10^1$} & \rvv{$2.9 \times 10^0$} & \rvv{$2.3 \times 10^{-2}$} & \rvv{$2.4 \times 10^{-3}$} \\ 
        \hline
        \rvv{TV denoising Section \ref{subsec:TV}} & \rvv{$6.2 \times 10^0$} & \rvv{$1.5 \times 10^{-1}$} & \rvv{$2.6 \times 10^{-4}$} & \rvv{$1.4 \times 10^{-4}$} \\ 
        \hline
        \rvv{FoE denoising Section \ref{subsec:FoE}} & \rvv{$1.8 \times 10^1$} & \rvv{$4.0 \times 10^0$} & \rvv{$2.8 \times 10^{-4}$} & \rvv{$2.4 \times 10^{-4}$} \\ 
        \hline
    \end{tabular}
    \caption{\rvv{Average CPU time (seconds) for total lower-level computational cost (lower-level solver (FISTA) iterations + matrix-vector products), descent direction check (evaluating bound \eqref{AP_bound_hg}), and backtracking line search per upper-level iteration of MAID. All experiments use the settings from \cref{sec:numerical_experiments} with $\epsilon_0 = \delta_0 = 10^{-1}$.}}
    \label{table:time_experiments}
\end{table}
\section{Conclusions and future work}\label{sec:future}
\rv{
In this paper, we proposed the MAID algorithm, along with its convergence analysis and numerical evaluation. For future work, it would be valuable to explore the convergence rate of MAID in addition to the asymptotic convergence results presented here.
One promising direction is to study non-monotone line search techniques \cite{nonmonotone_classic, non_monotone_stochastic_linesearch} as an extension to the current analysis, which is based on an Armijo-type monotone line search.
Additionally, given the success of stochastic line search methods \cite{Painless_Stochastic_linesearch, non_monotone_stochastic_linesearch, LineSearchNoise_stochastic}, investigating stochastic variants of MAID could enable scalability to large datasets and potentially provide speed-ups.}

\rv{
A future direction could involve extending the analysis of our method to accommodate potentially non-convex lower-level functions utilizing the \textit{value function} approach \cite{dempe_bilevel_2020}. Additionally, comparing fully first-order methods \cite{bome,fullyFirstOrderStochastic,penaltyMethodsNonconvexBO} with inexact IFT+CG-based approaches, as well as exploring ways to combine the adaptivity and robustness of MAID with fully first-order methods to leverage their avoidance of second-order computations, presents a promising avenue for further study.}

MAID is the first algorithm which actively exploits the fundamental trade-off in the accuracy and computational cost in estimating function values and gradients in bilevel learning problems. Rather than a-priori selected, MAID selects accuracies for both function values and gradients following the overarching paradigm to be as low accurate as possible to save computational cost but as accurate as needed to guarantee global convergence to critical points. A key novelty is that MAID performs backtracking using inexact computations to assure monotonic descent.
Our numerical results underscore the superiority of MAID over state-of-the-art derivative-free and first-order methods for bilevel learning across a range of problems in data science. Importantly, it demonstrates remarkable robustness in the face of hyperparameters such as the initial accuracy and starting step size choices, alleviating the need for tuning these algorithm parameters specifically for each application. 


\section*{Acknowledgments}
This research made use of Hex, the GPU Cloud in the Department of Computer Science at the University of Bath.

\bibliographystyle{siamplain}
\bibliography{references}
\end{document}